\documentclass[11pt]{amsart}
\usepackage{amssymb, amsthm, amsmath}
\usepackage{mathtools}
\usepackage{mathrsfs}
\usepackage{hyperref, manfnt}
\usepackage[numbers]{natbib}
\usepackage{setspace}
\usepackage{graphicx}
\usepackage{enumitem}
\usepackage{MnSymbol}
\usepackage{fullpage}
\usepackage{tikz}
\usepackage{tikz-cd}
\newcommand*{\DashedArrow}[1][]{\mathbin{\tikz [baseline=-0.25ex,-latex, dashed,#1] \draw [#1] (0pt,0.5ex) -- (1.3em,0.5ex);}}%
\usepackage{wrapfig}
\usepackage{float}
\makeatletter
\newcommand*\bigcdot{\mathpalette\bigcdot@{.5}}
\newcommand*\bigcdot@[2]{\mathbin{\vcenter{\hbox{\scalebox{#2}{$\m@th#1\bullet$}}}}}

\newcommand{\xrightarrowdbl}[2][]{%
  \xrightarrow[#1]{#2}\mathrel{\mkern-14mu}\rightarrow
}
\makeatother

\usepackage{kantlipsum}
\newcommand\Z{\mathbb{Z}}
\usepackage{bm}
\usepackage[cmtip,all]{xy}
\newtheorem*{thm}{Theorem}
\newtheorem*{theorem*}{Theorem}
\newtheorem{theorem}{Theorem}
\newtheorem{defi*}{Definition}
\numberwithin{theorem}{section}
\newtheorem{cor}{Corollary}
\newtheorem{def-lemma}{Definition-Lemma}
\newtheorem{rem}{Remark}
\newtheorem*{cor*}{Corollary}
\newtheorem{lemma}{Lemma}

\newtheorem*{rem*}{Remark}
\numberwithin{lemma}{section}
\newtheorem*{lemma*}{Lemma}

\newtheorem*{claim*}{Claim}

\tikzset{
	symbol/.style={
		draw=none,
		every to/.append style={
			edge node={node [sloped, allow upside down, auto=false]{$#1$}}}
	}
}

\theoremstyle{definition}

\newtheorem{innercustomthm}{Theorem}
\newenvironment{customthm}[1]
  {\renewcommand\theinnercustomthm{#1}\innercustomthm}
  {\endinnercustomthm}
\newtheorem{innercustomcor}{Corollary}
  \newenvironment{customcor}[1]
	{\renewcommand\theinnercustomcor{#1}\innercustomcor}
	{\endinnercustomcor}

\newcommand{\p}{\mathbb{P}}
\newcommand{\ri}{\rightarrow}
\newcommand{\J}{\mathscr{J}}
\newcommand{\mo}{\mathcal{O}}
\newcommand{\K}{\mathbb{C}}
\newcommand{\W}{\mathcal{W}}
\newcommand{\s}{\sigma}

\newcommand{\E}{\mathscr{E}}
\newcommand{\e}{\mathcal{E}}
\newcommand{\Q}{\mathbb{Q}}
\newcommand{\C}{\mathscr{C}}

\newcommand{\ka}{\kappa}
\newcommand{\A}{\alpha}
\newcommand{\T}{\Theta}
\newcommand{\B}{\beta}
\newcommand{\lb}{\mathcal{L}}
\newcommand{\Sym}{\operatorname{Sym}}
\newcommand{\Hom}{\operatorname{Hom}}
\newcommand{\la}{\lambda}
\newcommand{\hr}{\hookrightarrow}
\usepackage{stackengine}
\newcommand{\hollowslash}{\setbox0=\hbox{/}\def\holwd{3pt}%
  \stackengine{-.3pt}{/}{\rlap{\kern.5pt\rule{\holwd}{0.4pt}}}{O}{r}{F}{F}{S}%
  \kern\dimexpr\holwd-\wd0-.2pt\relax%
  \stackengine{-.4pt}{/}{\llap{\rule{\holwd}{0.4pt}\kern-0.1pt}}{U}{l}{F}{F}{S}%
}

\title[\MakeUppercase{chow ring of the moduli space of rank two bundles on genus two curves}]{\MakeUppercase{Rational chow ring of the universal moduli space of semistable rank two bundles over genus two curves} }
\author{Shubham Saha}
\address{Department of Mathematics, University of California, San Diego, La Jolla, CA 92093, USA}
\email{shsaha@ucsd.edu}
\usepackage{microtype}
\begin{document}
	\begin{abstract}
	We determine the rational Chow ring of the universal moduli space of rank $2$ semistable bundles over smooth curves of genus $2$, and show that it is generated by certain tautological classes.  
	In the process, we obtain Chow rings of universal Jacobians over genus $2$ curves with marked Weierstrass points, and the Chow ring of the universal pointed Jacobian. This further provides alternate computations to \cite{larson-24} in the genus $2$ case. 
	\end{abstract}
	\subjclass[2020]{14C17, 14D20, 14H45}
	\maketitle
	\vspace{-1em}
\section{Introduction}
	\noindent The universal moduli stack of vector bundles over algebraic curves has been the focus of extensive study in \cite{bae-thesis, fringuelli,brauer, {larson-24},{melo},{younghan}} among others.
	 Notably, the universal Picard stack has emerged as a central object in the intersection theory of the moduli space of curves, particularly through its connection to the double ramification cycle \cite{younghan}. 
	 Despite these developments, the global geometry of these stacks, in rank one or higher, and their associated good moduli spaces remains only partially understood. 
	 Aside from computations of their Picard groups \cite{fringuelli,larson-24,melo}; and Chow rings of universal Picard stacks over the hyperelliptic locus \cite{larson-24}, much of their intersection-theoretic properties are still unknown in higher ranks (beyond genus $0$ where the Chow ring was calculated in \cite{larson-18}).\\
Let $\mathscr{U}(r,d,g)\xrightarrow{}\mathcal{M}_g$ be the stack of semistable bundles of rank $r$, degree $d$ over smooth curves of genus $g$. The stack $\mathscr{U}(r,d,g)$ is a smooth Artin stack of finite type over $\mathcal{M}_g$ \cite[Theorem 1.2.2]{fringuelli} with a good moduli space $U(r,d,g)\xrightarrow{}{M}_g$.\\
Much like $\mathcal{M}_g$, we can define {tautological classes} on the moduli stacks of bundles over curves. For notational convenience, let $\mathcal{E}_d $ be the Poincar\'e bundle on
	$\C_g\times_{\mathcal{M}_g}\mathscr{U}(r,d,g)\xrightarrow{\pi_\mathscr{U}}\mathscr{U}(r,d,g)$ where $\C_g\ri \mathcal{M}_g$ is the universal curve.
	The \textit{twisted-kappa classes} have been defined in \cite{larson-24} for the universal Picard stack and the definition can be adapted to higher ranks as discussed in \S\ref{tautological-intro}.\\
We compute the rational Chow ring of the good moduli space $U(2,d,2)$ for all values of $d$.
Due to the Artin-stack nature of $\mathscr{U}(2,d,2)$, the natural pullback $A^*({U}(2,d,2))\hr A^*(\mathscr{U}(2,d,2))$ fails to be an isomorphism. \\
We explicitly describe this map by expressing the images of the generators of $A^*(U(2,d,2))$ in terms of the tautological classes on $\mathscr{U}(2,d,2)$.\\
For even degree, the canonical isomorphisms between $\mathscr{U}(2,2d,2)$ for different values of $d$ with the same parity, obtained by suitable twisting of the relative canonical, reduce the problem to the cases $d= 0,1$. 
\begin{customthm}{A}\label{A}
	There are divisors $\nu_d, \tilde{Z}_d\in A^1(U(2,2d,2))\forall d\in \Z$, such that 
	$$A^*(U(2,2d,2))\simeq \Q[\nu_d, \tilde{Z}_d]/(\nu_d^4, \tilde{Z}_d^3).$$
	 The image of the pullback $A^*(U(2,2d,2))\ri A^*(\mathscr{U}(2,2d,2))$ for $d\in \{0,1\}$ is given by 
	$$\nu_d\mapsto \ka_{-1,0,1}-\dfrac{1}{4}\ka_{-1,2,0}, \tilde{Z}_d\mapsto 4d\ka_{0,1,0}-2\ka_{-1,2,0}.$$
\end{customthm}
\noindent As a byproduct of these computations, we recover \cite[Theorem 1.1]{larson-24} for genus $2$, and the rational Picard group as presented in \cite[Theorem A.2]{fringuelli} for the rank $2$, genus $2$ case.\\
For all values of $d$, $A^*(U(2,2d,2))$ are shown to be isomorphic despite the lack of canonical isomorphisms between $U(2,0,2),U(2,2,2)$. 
However, we show that the images of pullbacks 
$A^*(U(2,2d,2))\ri A^*(\mathscr{U}(2,2d,2))$
 have different structures based on the parity of $d$.
\begin{customcor}{A.1}
	The class of the strictly semistable locus in $A^*(U(2,2d,2))$ is given by $4\nu_d\in A^1(U(2,2d,2))$. Applying excision, we have 
	$$A^*(U^s(2,2d,2))\simeq \Q[\tilde{Z}_d]/(\tilde{Z}_d^3).$$
\end{customcor}
\noindent For odd degree, we have canonical isomorphisms between $\mathscr{U}(2,2d-1,2)$ for different values of $d$, obtained by suitable twisting of the relative canonical and taking duals, based on the parity of $d$.\\
This reduces the problem to the case $d=3$. The specific choice $d=3$ is numerically motivated as it makes the description of the generators in the Chow ring simpler and allows the application of tools described in \cite{bertram}.
\begin{customthm}{B}\label{B}
	Let $B_*\in A^2(U(2,3,2))$ be the universal Brill-Noether locus of bundles with $2$ independent sections. 
	There are divisors $H_U,\T_U\in A^1(U(2,3,2))$ such that 
	\begin{gather*}A^*(\mathcal{U}(2,3,2))\simeq A^*(\tilde{U}(2,3,2))=\Q[H_U,\T_U,B_*]/(\T_U^3,H_U^2\T_U^2-4B_*\T_U^2,B_*H_U^2-B_*H_U\T_U+\dfrac{1}{2}B_*\T_U^2,\\
B_*H_U^2-B_*^2,H_U^3+H_U^2\T_U+\dfrac{1}{2}H_U\T_U^2-4B_*H_U).\end{gather*}
	The image of the pullback $A^*(U(2,3,2))\ri A^*(\mathscr{U}(2,3,2))$ is given by
	\begin{gather*}H_U\mapsto \dfrac{1}{2}\left(4\ka_{-1,0,1}-\ka_{-1,2,0}\right), 
		 B_*\mapsto \dfrac{1}{2}\left(\dfrac{1}{2}(\kappa_{-1,2,0}-\kappa_{0,1,0})-\kappa_{-1,0,1}\right)^2- \dfrac{1}{2}\left( \dfrac{\kappa_{-1,3,0}}{3}-{\ka_{-1,1,1}}-\dfrac{\ka_{0,2,0}}{2}+{\ka_{0,0,1}}\right),\\ 
\T_U\mapsto \dfrac{1}{2}\left(3\ka_{0,1,0}-\ka_{-1,2,0}\right).\end{gather*}
\end{customthm}
\noindent We consider $M_*$, the coarse moduli space of $\mathcal{H}^w_{2,6}$ (stack of genus $2$ curves with marked Weierstrass points). It is known that $\mathcal{H}^w_{2,6}$ is a trivial $\mu_2$-gerbe over $M_*$ by \cite[Proposition 6.3]{zhengning}.
The associated section $M_*\ri \mathcal{H}^w_{2,6}$ gives a ``universal" family of genus $2$ curves $\C_*\xrightarrow{\pi}M_*$ with sections $\{\s_i\}_{1\leq i\leq 6}$ parametrizing Weierstrass points.
For degree $d$, the relative Jacobian for the family $\pi$ is denoted by $J_*(d)$. Our computations of $A^*(J_*(d))$ show that the pullback map $A^*(J_2^d)\ri A^*(J_*(d))$ is an isomorphism for all values of $d$ where $J_2^d$ is the universal Jacobian of degree $d$ over genus $2$ curves.
\begin{customthm}{C}\label{C}
The Chow ring of $J_*(d)$ is given by 
$$A^*(J_*(d))\simeq \Q[\T_d]/(\T_d^3)$$ for all $d\in \mathbb{Z}$. Consequently, the Chow rings of $J_{2}^d$ are given by 
$$A^*(J_{2}^1)\simeq \Q[\T]/(\T^3), A^*(J_{2}^2)\simeq \Q[Z]/(Z^3)$$\label{chow-j-2} where $\Theta$ is the theta-divisor in $J_{2}^1$ and $Z\hr J_{2}^2$ is the image of $Z_\C$ under the map $J_*(2)\ri J_{2}^2$.
	\end{customthm}
\noindent Additionally, the computations for $A^*(U(2,3,2))$ uses the Chow ring of $C_2\times_{M_2}J_2^d$ where $C_2\ri M_2$ is the universal curve over $M_2$. To simplify notation, we compute the image of its pullback to $A^*(\C_*\times_{M_*}J_*(3))$. We further denote $\C_*\times_{M_*}J_*(3)$ by $\C_J$.
\begin{customthm}{D}\label{D}
Let $S^*(\C_J)$ be the image of the pullback $q_C^*:A^*(C_2\times_{M_2}J_2^d)\hr A^*(\C_J)$. We have 
$$S^*(\C_J)\simeq \Q[K_{\pi_J}, \pi_J^*\Theta_*,\xi]/(\pi_J^*\Theta_*^3, K_{\pi_J}^2,\xi^2 +K_{\pi_J}\pi_J^*\T_*,\xi K_{\pi_J}, (\xi-\dfrac{3}{2}K_{\pi_J})\pi_J^*\T_*^2).$$
\end{customthm}
\subsection{Outline of the paper}
The foundational strategy of this paper is to extend the classical results in \cite{bertram} and \cite{nar-ram} to ${U}(2,d,2)$, from its fiber over the universal Jacobian. \\
We use $\pi$ to construct $\tilde{U}(2,d,2)$, a finite cover of $U(2,d,2)$, and compute the image of the pullback map $A^*(U(2,d,2))\ri A^*(\tilde{U}(2,d,2))$.
\subsubsection{Even Degree}\label{even-1.1}
We begin by setting up the problem in \S \ref{even-degree} by defining the necessary spaces and morphisms. 
Following the necessary constructions, we extend classical results from \cite{nar-ram} to the universal setting in \S\ref{univ-ext-even}.
The calculation of the Chow rings of universal Jacobians and their defined covers are in \S \ref{uni-jac-chow}. 
Previous computations and arguments are applied in \S\ref{final-steps} to compute $A^*(U(2,2d,2))$ for $d=0,1$. 
The universal moduli spaces of stable bundles are studied in \S\ref{stable-locus}.
We find explicit expressions for the seemingly abstract generators in terms of tautological classes by computing their pullbacks over a universal extension space with a ``Poincar\'e bundle" in \S\ref{even-taut}.\\
\subsubsection{Odd Degree}\label{odd-1.1}
We use the succesive blowup of extension spaces described in \cite{bertram} for degree $3$ bundles. 
The map from the blowup of the extension space to the moduli space is a blowdown along a Brill-Noether locus.
The geometry of this blowup-blowdown configuration is studied in \S \ref{bertram-construct}.
To explictly carry out the constructions in the universal setting, we extend Bertram's arguments to families of genus $2$ curves with a line bundle of relative degree $3$ in \S \ref{gen-U}.
This is followed by certain computations in the Chow ring of the aforementioned blowup, namely $\tilde{\p}_\xi$, which blows down to $\tilde{U}(2,3,2)$ in \S \ref{chow-blowup-ss}.
We use the sections $\{\s_i\}_{1\leq i\leq 6}$ to define several loci in \S\ref{loci-blowup} to understand the blowdown $\tilde{\p}_\xi\xrightarrow{\Phi_\xi} \tilde{U}(2,3,2)$. Specifically, we find the classes contracted by $\Phi_\xi$.
The linear generators of the Chow ring of the universal Brill-Noether locus, namely $B_*$, are constructed in \S \ref{gen-b-*}.\\
We use the computations from \S\ref{chow-blowup-ss}, \S\ref{loci-blowup}, \S\ref{gen-b-*}, and an excision argument in \S\ref{U_*} to compute $A^*(U(2,3,2))$.
The final theorem states generators and relations to describe $A^*(U(2,3,2))$ via its pullback into $A^*(\tilde{U}(2,3,2))$. 
We find explicit expressions for the seemingly abstract generators in terms of tautological classes via their pullback to $\mathscr{U}(2,3,2)$ in \S \ref{odd-taut}.
\subsection*{Acknowledgements} I would like to thank Elham Izadi for suggesting this problem and for her continued guidance. 
	I would also like to thank Dragos Oprea for his valuable suggestions and for the references \cite{king-newstead}, \cite{hitchin}.
	\section*{Notation}\label{notation}
	\noindent All schemes in this paper are taken over the field of complex numbers. 
	All Chow rings are taken with rational coefficients. We use the subspace convention for projective bundles.
	For a fiber product $X\times_S Y$, we denote the projection map onto the $i$-th factor as $\pi_i$ for $i=1,2$. 

	\begin{rem*}
The underlying assumptions on the ground field, as outlined above, are used throughout the paper to establish certain isomorphisms via Theorem \ref{zmt}. 
\end{rem*}
	\section{Background and known results}
	\noindent Zariski's Main Theorem, as stated in \cite[Theorem 4.4.3]{ega}, can be used to find a sufficient condition for schemes to be isomorphic over an algebraically closed field of characteristic $0$.\\
	We use the following theorem throughout the paper to establish isomorphisms between varieties. A variety in this context is a reduced scheme of finite-type. 
	\begin{theorem}\label{zmt}
	Let $f:X\ri Y$ be a morphism of quasi-projective varieties over an algebraically closed field of characteristic $0$. Then $f$ is an isomorphism if the following conditions are satisfied:
	\begin{itemize}
		\item $f$ is a bijection on closed points
		\item $Y$ is normal
		\item $X$ is connected.
	\end{itemize}
	\end{theorem}
	\subsection{A cover of ${M}_2$}\label{cover}
	Let $M_* = \mathcal{M}_{0,6}\simeq \{(p_1,p_2,p_3)|\quad p_i\in \mathbb{A}^1 \setminus \{0,1\}\forall i; p_i \text{ mutually distinct}\}\subset \mathbb{A}^3$. The map $\mathcal{M}_2\rightarrow D_6 = [M_*/S_6]$ is a $\mu_2$-gerbe by \cite[Remark 2.1]{zhengning}.
	 Hence the corresponding coarse moduli spaces for $\mathcal{M}_2, D_6$ are isomorphic. That is, $M_2\simeq M_*/S_6$ where $M_*/S_6$ is the GIT-quotient of the quasi-projective $M_*$.\\
	Therefore, we have a morphism $M_*\rightarrow M_2$ which makes $M_*$ an \'etale $S_6$-cover of $M_2$. We can in fact show that the morphism $M_*\rightarrow M_2$ lifts to $\mathcal{M}_2$. That is, there is a ``universal" family of curves on $M_*$.\\
	To construct such a family, we use arguments presented in \cite[Proposition 3.1]{vis-m2}. We make use of the following sections for the projection $\p^1\times M_*\rightarrow M_*$:
	$$\sigma_1 = {0}\times M_*, \sigma_2 = {1}\times M_*, \sigma_3 = {\infty}\times M_*, (\sigma_{i+3})_{(p_1,p_2,p_3)} = p_i\forall 1\leq i\leq 3$$ 
	We construct the family $\mathscr{C}\rightarrow M_*$ by setting $$\mathscr{C}_*:=\underline{Proj}_{\p^1\times M_*}(\mathcal{O}\oplus \mathcal{O}(-3)),$$ 
	where the product on this sheaf of rings is given by the map $\mathcal{O}(-6)\hr\mathcal{O}$,
	 defined by the Cartier divisor $D_\s:=\Sigma_{1\leq i\leq 6} \sigma_i$. 
	$$\begin{tikzcd}
		\mathscr{C}_*\arrow[r,"\phi"] \arrow[d, "\pi"]&\p^1\times M_*\arrow[ld]\\
		M_*\arrow[ru, , bend right=20]\arrow[ru , "\sigma_1\, \cdots\, \sigma_6", bend right = 80]&
	\end{tikzcd}.$$
	Now, $\phi$ is a degree $2$ cover ramified along $\sigma_i$'s. Using Theorem \ref{zmt}, $(\phi^{-1}(\sigma_i))_{red}$ maps isomorphically to $M_*$ via $\pi$. 
	Therefore, the sections $\sigma_i$ lift to sections of $\mathscr{C}_*\rightarrow M_*$. We identify $\s_i$ with its lift to $\C_*$.
	\\Furthermore, we have a moduli interpretation of $M_*$. Following the notation in \cite{zhengning}, the map $\mathcal{H}_{2,6}^w\rightarrow M_*$ is a $\mu_2$-gerbe. Therefore, $M_*$ is the coarse moduli space to the stack of genus $2$ curves with marked Weierstrass points.

	\begin{rem}\label{c-s6}\noindent
The induced $S_6$-action on $\p^1\times M_*\simeq \p^1\times M_{0,6}$ (see \cite[Proposition 6.1]{zhengning} for details) permutes the sections $(\s_i)_{1\leq i\leq 6}$ accordingly. 
Hence, the divisor $D_\s$ is invariant under this action.\\
Let $\mathcal{R}:= \mo\oplus\mo(-3)\ri \p^1\times M_*$. We thus have natural isomorphisms $g^*\mathcal{R}\simeq \mathcal{R} \forall g\in S_6$.
These isomorphisms induce an $S_6$-action on $\C_*$ with the claimed properties.
We can therefore define an $S_6$-action on $\C_*$ so that $\pi$ is equivariant.
\end{rem}

\subsection{Chow Rings and Stacks}
Much like homology, there is an excision sequence for Chow groups. This is used in the paper to obtain generators for Chow groups under a given stratification.
	\begin{theorem}\label{excision}
		Let $X$ be a scheme and $Z\xrightarrow{i_Z} X$ be a closed subscheme. Let $U:= X\setminus Z\xrightarrow{i_U}X$ be the open complement of $Z\hr X$. Then, the following sequence is exact
$$A_j(Z)\xrightarrow{i_{Z *}}A_j(X)\xrightarrow{i_U^*}A_{j}(U)\rightarrow 0\quad \forall j\geq 0 .$$
	\end{theorem}
	\noindent Let $Y$ be a smooth variety and $X\hr Y$ be a smooth subvariety with $d = \dim Y-\dim X$. 
	We denote $f:\tilde{Y}\ri Y$ as the blowup of $Y$ along $X$:
	$$\begin{tikzcd}
		\tilde{X}\arrow[r,"j"]\arrow[d,"g"]\arrow[rd,phantom,"\square"]& \tilde{Y}\arrow[d,"f"] \\
		{X}\arrow[r,"i"]& Y
	\end{tikzcd}$$
	where $\tilde{X}$ is the exceptional locus of $f$.\\
	Under the aforementioned hypothesis, we have explicit and linearly independent generators for $A^*(\tilde{Y})$ in terms of $A^*(Y)$ and $A^*(X)$ from \cite[Proposition 0.1.3]{beauville}. 
	The intersection product on $A^*(\tilde{Y})$ is described in \cite[Proposition 13.12]{3264}. We state the results below.
	\begin{theorem}\label{blowup-chow}
		The Chow ring $A^*(\tilde{Y})$ is generated by $f^*(A^*(Y))$ and $j_*(A^*(\tilde{X}))$, that is,
		classes pulled back from $Y$ and classes supported on $\tilde{X}$. Let $\zeta := -c_1(\mathcal{N}_{\tilde{X}/\tilde{Y}}) = c_1(\mo_{\tilde{X}}(1))$ 
		be the relative hyperplane class of the projective bundle defined by $g$. The product satisfies
		\begin{gather*}
			f^*\alpha \cdot f^*\beta = f^*(\alpha \cdot \beta)\quad \forall \alpha, \beta\in A^*(Y),\\
			f^*\alpha.j_*\gamma = j_*(\gamma\cdot g^*i^*\alpha) \quad \forall \alpha\in A^*(Y), \gamma\in A^*(\tilde{X}),\\
			j_*\gamma.j_*\delta = -j_*(\zeta.\gamma\cdot \delta) \quad \forall \gamma, \delta\in A^*(\tilde{X}).
		\end{gather*}
	\end{theorem}
	\begin{theorem}
	For $q\geq 1$, the Chow groups $A^q(\tilde{Y})$ (resp., $H^q(\tilde{Y})$) are given by the isomorphisms
	$$A^q(\tilde{Y}) \rightleftarrows A^q(Y)\oplus_{r=0}^{c} A^{q-1-r}({X})\quad \left(\text{resp., }H^q(\tilde{X})\rightleftarrows H^q(X)\oplus_{r=0}^{c} H^{q-2-2r}({X}) \right)$$
	where $c=d-2$. 
	\begin{itemize}
		\item The map $A^q(\tilde{Y}) \rightarrow A^q(Y)\oplus_{r=0}^{c} A^{q-1-r}({X})$ is given by
		$$z\mapsto f^*f_*z +\oplus_{r=0}^c g_*(\gamma_{c-r}\cdot j^*z)$$
		where $\gamma_r:= \zeta^r+\zeta^{r-1}g^*c_1(N)+\cdots+ g^*c_r(N)\forall 0\leq r\leq c $ and $N:= \mathcal{N}_{X/Y}$.
		\item The map $A^q(Y)\oplus_{r=0}^{c} A^{q-1-r}({X})\ri A^q(\tilde{Y})$ is given by
		$$x+\Sigma_{r=0}^cy_r\mapsto f^*x - j_*(\Sigma_{r=0}^c\zeta^r\cdot g^*y_r).$$
	\end{itemize}
	\end{theorem}
	\noindent For a subvariety $V\hr Y$, the class of its strict tranform $\tilde{V}\hr \tilde{Y}$ in $A^*(\tilde{Y})$ is given by \cite[Theorem 6.7]{fulton} as stated below.
	\begin{theorem}\label{transform}
Let $V$ be a $k$-dimensional subvariety of $Y$, and let $\tilde{V}\hr \tilde{Y}$ be the strict transform of $V$, i.e. the blow-up of $V$ along $V\cap X$. Then,
$$[\tilde{V}]= f^*[V]-j_*\{c(E)\cap g^*s(V\cap X, V)\}_k \in A_k(\tilde{Y})$$
where $E:= g^*N/\mathcal{N}_{\tilde{X}/\tilde{Y}} = g^*N/\mo_{\tilde{X}}(-1)$ and $N:= \mathcal{N}_{X/Y}$.
	\end{theorem}
	\begin{cor}
		The computations in the previous theorem get simplified when $\dim (V\cap X)$ is close to the expected dimension. In particular, 
		\begin{itemize}
			\item If $\dim (V\cap X) = \dim V - d$, we have 
			$$f^*V = [\tilde{V}]$$
			\item If $\dim (V\cap X) = \dim V - d +1$ and equidimensional, we have 
			$$ [\tilde{V}] = f^*V - j_*g^*(V\cap X).$$ 
		\end{itemize} 
	\end{cor}
\noindent The following remark shall be used extensively in the upcoming sections and can be found in \cite[Remark 2.2]{vakil}.
	\begin{rem}\label{rem1}
		If $f:X\rightarrow Y$ is a finite, flat map of degree $d$. 
		Then $f^*:A^*(Y)\hookrightarrow A^*(X)$ is injective with $\dfrac{1}{d}f_*$ as its one-sided inverse. 
		In particular, $A^*(X)=\Q\implies A^*(Y)=\Q$. 
	\end{rem}
	\begin{rem}
		Using a Moving Lemma, one can relax the condition of flatness in Remark \ref{rem1} to surjective for $X, Y$ smooth and quasi-projective.
	\end{rem}
\subsection{The Two Stacks}
There are two versions of the moduli functor for families of (semi)stable bundles on smooth curves.
Consequently, we get two different stacks $\mathscr{U}(r,d,g),\mathcal{U}(r,d,g)$ for the moduli of (semi)stable bundles of rank $r$, degree $d$ and genus $g$. 
The categories $\mathscr{U}(r,d,g),\mathcal{U}(r,d,g)$ consist of identical objects 
$$\begin{tikzcd}
	\e_S\arrow[d]\\\C_S\arrow[d,"\pi_S"]\\S
\end{tikzcd}$$
which are families of fiberwise (semi)stable bundles of rank $r$ and relative degree $d$ over smooth curves $\C_S\ri S$ of genus $g$.\\
For both stacks, morphisms between two such families can be defined using the following diagram 
$$\begin{tikzcd}
	\e_S\arrow[d]&\e_T\arrow[d]\\
	\C_S\arrow[d,"\pi_S"]\arrow[r,"\pi_f"]\arrow[rd,phantom,"\square"]&\C_T\arrow[d,"\pi_T"]\\
	S\arrow[r,"f"]&T.
\end{tikzcd}$$
For $\mathscr{U}(r,d,g)$, the morphisms between objects shown above are given by a pair $(\pi_f, \phi_f)$ where $\pi_f:\C_S\ri\C_T$ is such that square above is Cartesian and $\phi_f:\pi_f^*\e_T\simeq \e_S$ is an isomorphism.\\
For $\mathcal{U}(r,d,g)$, a morphism consists of the map $\pi_f$ where $\pi_f:\C_S\ri\C_T$ is such that square above is Cartesian and there exists $ \phi_f:\pi_f^*\e_T\simeq \e_S\otimes \pi_S^*L$ for some line bundle $L\in Pic(S)$.\\
\\ In order to motivate the two versions of the moduli functor, we state a generalized version of \cite[Lemma 5.10]{newstead} below.
	\begin{theorem}\label{newstead}
		Let $\C\xrightarrow{\pi_S} S$ be a family of smooth curves over a (quasi-projective) base $S$ and $\mathcal{E}_1,\mathcal{E}_2$ be families of stable bundles over $\pi_S$. We have $\mathcal{E}_1|_{\C_s}\simeq \mathcal{E}_2 |_{\C_s}\forall s \in S$ if and only if $\exists L\in Pic(S)$ such that $\mathcal{E}_1\simeq \mathcal{E}_2\otimes \pi_S^*L$.
	\end{theorem}
\begin{proof}
	The fiberwise isomorphism over $S$ follows immediately from the existence of $L$. To show the converse, we prove that fiberwise isomorphism implies that the claimed isomorphism is satisfied by 
	$$L= \pi_{S*}{\mathscr{H}om}(\e_2,\e_1)$$
	Cohomology and Base-Change imply that $L$ is a line bundle since $\e_1,\e_2$ are families of stable bundles. The induced morphism 
	$$\e_2\otimes \pi_S^*L = \e_2\otimes \pi_S^*\pi_{S*}\mathscr{H}om(\e_2,\e_2)\ri \e_2\otimes \mathscr{H}om(\e_2,\e_1)\ri \e_1$$
is the claimed isomorphism. \\
To show that the induced morphism is indeed an isomorphism, it is enough to show that it is non-zero since fibers of $\e_1,\e_2$ are stable. This is equivalent to showing that the canonical map 
$$\pi_S^*\pi_{S*}\mathscr{H}om(\e_2,\e_2)\ri \mathscr{H}om(\e_2,\e_1)$$
 is non-zero, which can be verified by restricting the morphism over any fiber of $\pi_S$.
\end{proof}	
\begin{rem}
	It has been established in \cite[Theorem 1.2.2]{fringuelli} that $\mathscr{U}(r,d,g), \mathcal{U}(r,d,g)$ are irreducible, smooth Artin stacks of finite type. 
\end{rem}
\begin{theorem}\label{artin-chow}
For $d$ odd, we have 
$$A^*(U(2,d,2))\simeq A^*(\mathcal{U}(2,d,2))$$
However, for $d$ even, we have the following relations between the Chow rings of $\mathcal{U}(2,d,2)$ and $U(2,d,2)$
$$A^*(U(2,d,2))\hr A^*(\mathcal{U}(2,d,2)), A^*(U^s(2,d,2))\simeq A^*(\mathcal{U}^s(2,d,2))$$
where $U^s(2,d,2)$ is the coarse moduli space of $\mathcal{U}^s(2,d,2)$, the open substack of $\mathcal{U}(2,d,2)$ determined by stable bundles.
\end{theorem}
\begin{proof}
The good moduli space of both $\mathscr{U}(r,d,g), \mathcal{U}(r,d,g)$ is $U(r,d,g)$, and the coarse moduli space of $\mathcal{U}^s(r,d,g)$ is $U^s(r,d,g)$ by \cite[Theorem 11.4, Theorem 13.6]{alper-good} and \cite[\S C.2]{GW-stacks}.\\
For any value of $d$, the isomorphism 
$$A^*(U^s(2,d,2))\simeq A^*(\mathcal{U}^s(2,d,2))$$ follows from \cite[Proposition 6.1]{vistoli-stacks} and \cite[Remark 8.3.4]{olsson} since $\mathcal{U}^s(2,d,2)$ is a Deligne-Mumford stack.\\
For $d$ even, the inclusion $$A^*(U(2,d,2))\hr A^*(\mathcal{U}(2,d,2))$$ is established in \cite[Theorem 1.1, Theorem 1.7]{edidin-satriano}.
\end{proof}
\subsection{Tautological Classes}\label{tautological-intro}
For the case $r=1$ or in codimension $1$, tautological classes on the stacks $\mathscr{U}{(r,d,g)} , \mathcal{U}{(r,d,g)}$ have been defined in many papers including
 \cite{fringuelli}, \cite{larson-24}, and \cite{melo}. We present some results and definitions from \cite{fringuelli} and \cite{larson-24} below, along with some original statements, and include proofs wherever necessary.\\
 Let $\mathcal{E}_d $ be the Poincar\'e bundle for the family 
 $\C_2\times_{\mathcal{M}_2}\mathscr{U}(2,d,2)\xrightarrow{\pi_\mathscr{U}}\mathscr{U}(2,d,2)$ where $\C_2\ri \mathcal{M}_2$ is the universal curve.\\
 We generalize the $\kappa$-classes defined in \cite{larson-24} for $\mathscr{U}(1,d,g)|_{\mathcal{H}_g}$ and the classical $\kappa$-classes on $\mathcal{M}_g$ as defined in \cite{mumford} in the definition below.
 \begin{defi*}
	The \textit{twisted-kappa classes} generate the tautological ring and are defined as follows:
  $$\kappa_{a,b,c}:= \pi_{\mathscr{U}*}(K_{\pi_\mathscr{U}}^{a+1}c_1(\mathcal{E}_d)^b c_2(\mathcal{E}_d)^c)\in A^{a+b+2c}(\mathscr{U}(2,d,2)) \text{ for } a\geq -1, b,c\geq 0.$$
 \end{defi*}
\noindent Furthermore, the classes
$$\Lambda (m,n,l):= d_{\pi_\mathscr{U}}(K_{\pi_\mathscr{U}}^{\otimes m}\otimes {\det\e_d}^{\otimes n}\otimes \e_d^{\otimes l})\in Pic(\mathscr{U}(2,d,2))$$
 as defined in \cite{fringuelli}, can easily be expressed in terms of the $\kappa$ classes by a straightforward application of the Grothendieck-Riemann-Roch Theorem.\\
 We have the following description of the Picard groups of $\mathscr{U}(2,d,2)$(and $\mathscr{U}(r,d,g)$ in general for $r\geq 1, g\geq 2$) in terms of the $\Lambda$-classes in \cite[Theorem A, A.1, A.2, B]{fringuelli} and \cite[Theorem A, B]{melo}. 
\begin{theorem}\label{fringuelli-main}
	The Picard groups of $\mathscr{U}(2,d,2)$ are generated by $\Lambda(1,0,0),\Lambda(1,1,0),\Lambda(0,1,0),\Lambda(0,0,1)$ and the Picard groups of $\mathcal{U}(2,d,2)$ are generated by $\Lambda(1,0,0),\Xi,\T $ with the unique relation 
	$$\Lambda(1,0,0)^{10} = \mo$$ 
	where $$\Xi =\Lambda(0,1,0)^{{(d+1)}/{n_{2,d-1}}} \otimes \Lambda(1,1,0)^{-{(d-1)}/{n_{2,d-1}}} , \T= \Lambda(0,0,1)^{2n_{2,d-1}/{n_{2,d}}}\otimes \Lambda(0,1,0)^\alpha\otimes \Lambda(1,1,0)^\beta$$ 
	and $\alpha,\beta$ are an integral solution to $\alpha(d-1)+\beta(d+1) = -\dfrac{n_{2,d-1}(d-2)}{n_{2,d}} , n_{2,m}=gcd(m,2)\forall m\in\Z$.
\end{theorem}
\noindent The Chow rings of $ \mathscr{J}_{d,g}:= \mathscr{U}(1,d,g)|_{\mathcal{H}_g}, J_{d,g};= \mathcal{U}(1,d,g)|_{\mathcal{H}_g}$, as stated below, were calculated in \cite[Theorem 1.1]{larson-24}. 
\begin{theorem}
	\label{larson-main}
The Chow ring of $\mathscr{J}_{d,g}$ is given by 
$$A^*(\mathscr{J}_{d,g}) = \dfrac{\Q[\kappa_{0,1,0},\kappa_{-1,2,0}]}{<(d\kappa_{0,1,0}-(g-1)\kappa_{-1,2,0})^{g+1}>} \quad \forall g\geq 2,d\in \Z.$$
Furthermore, the Chow ring of its rigidification $J_{g}^d$ is a subring of $A^*(\mathscr{J}_{d,g})$ given by 
$$A^*(J_{g}^d)=\dfrac{\Q[u]}{<u^{g+1}>}\hr A^*(\mathscr{J}_{d,g}); u\mapsto d\kappa_{0,1,0} - (g-1)\kappa_{-1,2,0}.$$
\end{theorem}
\noindent 
 The following result about $\nu_{r,d}^*$ shall prove useful in \S\ref{tautological}.
\begin{theorem}\label{pullback-nu}
	The pullback map $\nu_{r,d}^*: A^*(\mathcal{U}(r,d,g))\ri A^*(\mathscr{U}(r,d,g))$ is injective. That is, $\nu_{r,d}^*$ induces an isomorphism onto its image, given by $A^*(\mathscr{U}(r,d,g))^{\mathcal{B}\mathbb{G}_m}$.
\end{theorem}
\begin{proof}
	We start by showig that $\nu_{r,d}^*$ is injective. 
	Assuming the contrary, let $Z\in ker(\nu_{r,d}^*)\setminus 0$ be a non-trivial cycle in its kernel. 
	Thus, $\exists t:S\ri \mathcal{U}(r,d,g)$ for a scheme $S$ such that $t^*Z\neq 0 $. \\
The morphism $t$ is given by a family 
$$\begin{tikzcd}
	\e_S\arrow[d]\\
	\C_S\arrow[d,"\pi_S"]\\
	S
\end{tikzcd}$$
defined upto twisting by an element in $\pi_S^*Pic(S)$.\\
For any choice of such a family, we can define a lift $t_{\e_S}:S\ri\mathscr{U}(r,d,g)$ of $t$ via $\nu_{r,d}$.\\
Consequently, we have the desired contradiction by $$0 = t_{\e_S}^*\nu_{r,d}^*Z = t^*Z\neq 0.$$
It remains to show that
$$A^*(\mathscr{U}(r,d,g))^{\mathcal{B}\mathbb{G}_m} = Im(\nu_{r,d}^*).$$
Containment of $Im(\nu_{r,d}^*)$ follows as an immediate consequence of the isomorphism  $\mathscr{U}(r,d,g)/\mathcal{B}\mathbb{G}_m \simeq \mathscr{U}(r,d,g)\hollowslash \mathbb{G}_m  = \mathcal{U}(r,d,g)$ by \cite[\S C.2]{GW-stacks}.\\
To show the converse, let $W\in A^*(\mathscr{U}(r,d,g))^{\mathcal{B}\mathbb{G}_m}$ be a cycle. We show that $W\in Im(\nu_{r,d}^*)$.\\
In order to do that, it is enough to define cycles $f^*W\in A^*(S)$ as pullbacks over any morphism $f:S\ri \mathcal{U}(r,d,g)$ such that these pullbacks satisfy the necessary functorial properties.\\
As established above, any such morphism $f:S\ri \mathcal{U}(r,d,g)$ has a lift $f_{\e_S}:S\ri \mathscr{U}(r,d,g)$. We define $f^*W := f_{\e_S}^* W\in A^*(S)$.\\
We show that this definition of $f^*W$ is 
\begin{enumerate}
	\item Independent of the choice of lift, $f_{\e_S}$,
	\item Satisfies the necessary functorial properties.
\end{enumerate}
We have $(1)\implies (2)$ since $W\in A^*(\mathscr{U}(r,d,g))$ and thus, the pullbacks $f_{\e_S}^*W$ satisfy the necessary functorial properties.\\
In order to show independence of chosen lifts, we use the $\mathcal{B}\mathbb{G}_m$-action on $\mathscr{U}(r,d,g)$.
The orbit of an object $(\e_S\ri\C_S\xrightarrow{\pi_S} S)$ is given by $\{(\e_S\otimes\pi_S^*\lb_S\ri\C_S\xrightarrow{\pi_S} S)|\lb_S\in Pic(S)\}$.
Since $W\in A^*(\mathscr{U}(r,d,g))^{\mathcal{B}\mathbb{G}_m}$, we have 
$$f_{\e_S}^*W = f_{\e_S\otimes \pi_S^*\lb_S}^* W\quad \forall \lb_S\in Pic(S).$$
Now, $f_{\e_S}:S\ri \mathscr{U}(r,d,g)$ is a lift of $f:S\ri \mathcal{U}(r,d,g)$ if and only if we have the same for $f_{\e_S\otimes \pi_S^*\lb_S}\forall \lb_S\in Pic(S)$, establishing $(1)$. 
\end{proof}
	\section{Even degree}\label{even-degree}\noindent
 We consider the moduli space of semistable rank $2$ vector bundles with determinant of relative degree $2d$ for the family $\pi$ and call it $\tilde{U}(2,2d,2)$. 
 Let $|2\T_*|_\pi := \p_{M_*}(\pi_*L_{\Theta_*}^2)$ be the relative linear system. 
	We have a unique component of the relative Picard scheme of $\pi$, for every relative degree $d\in \mathbb{Z}$, rigified along the section $\s_1$ in the sense of \cite[Definition 2.8]{kleiman}.
	 We shall refer to this component as $J_*(d)$ from hereon. 
We represent the $2$-torsion points in the abelian scheme $J_*(0)\ri M_*$ by $\Gamma_*$.We have the following diagram
	$$
	\begin{tikzcd}\label{diag_0}
		&\mid 2\Theta_*\mid_\pi\times_{M_*} J_*(d)\arrow[r, "/\Gamma"]\arrow[d]\arrow[rd]& \tilde{U}(2,2d,2)\arrow[r, "q"]&U(2,2d,2)&\\
		&\mid 2\Theta_*\mid_{\pi}\arrow[rd] &J_*(d)\arrow[d]&\\
		&(\Theta_* := )\quad\mathscr{C}_*\hookrightarrow J_*(1)\arrow[r,"\pi"]     &M_*\arrow[r,"/S_6"]\arrow[l, "\sigma_1\,\cdots\,\sigma_6 ",bend left]\arrow[l, , bend left=60]&M_2.
	\end{tikzcd}$$
	The induced morphism $\Gamma_*\ri M_*$, is apriori, a degree $16$ cover.
	The sections $\{\s_i\}_{1\leq i\leq 6}$ define a trivialization of this cover since $2$-torsion points of a genus $2$ curve are given by the pairwise differences of its Weierstrass points. \\
	Consequently, $\Gamma_*\simeq \mu_2^{\oplus 4}\times M_*$, where we identify $\mu_2^{\oplus 4}$ with its associated group variety.
	The \'etale equivalence relation defined using twisting by the $2$-torsion points $\Gamma_*\ri M_*$ is therefore equivalent to the finite-group action of $\mu_2^{\oplus 4}$ on $| 2\Theta_*|_\pi\times_{M_*} J_*(2d)$ and is uniquely defined by the isomorphism $\Gamma_*\simeq \mu_2^{\oplus 4}\times M_*$.\\
	It is worth noting that there is no canonical choice for the isomorphism $\Gamma_*\simeq \mu_2^{\oplus 4}\times M_*$. We shall fix one such isomorphism and refer to the corresponding group action as $\Gamma$ from hereon. 
	\begin{rem}
A point worth noting is that the splitting of $\Gamma_*$ as $\mu_2^{\oplus 4}\times M_*$ is not naturally defined. It is obtained using the fact that $\Gamma_*\hookrightarrow J_*(0)$ is a degree $16$ cover of $M_*$,
 with $16$ mutually disjoint sections given by $\{\mo_{\C_*}(\sigma_i-\sigma_j)\}_{1\leq j< i\leq 6}$ and the identity section of $J_*(0)
 \ri M_*$. 
	\end{rem}
	\noindent For the following remark and thereafter, we denote the universal rigidified Picard stack of degree $d$ over smooth genus $g$ curves by $J^d_g$. 
	\begin{rem}\label{j*-s6}
The $S_6$-action on $\C_*$, equivariance of $\pi$ and universal property of $J_*(d)$ induce an $S_6$-action on $J_*(d)$ for all $d\in \mathbb{Z}$.
Furthermore, the map $J_*(1)\ri J_{2}^1$ is invariant under this action.\\
It is worth noting that the induced map $J_*(d)/S_6\ri J^d_{2}$ is finite and generically a double cover due to the involution on $\C_*$. \\
	The natural embedding $\C_*\hr J_*(d)$ for $d$ odd is invariant under this action and so is the image of the map $\C_*\ri J_*(2)$ given by the divisor of double points. 
	In fact, they are pullbacks of cycles in $J_{2}^d$ under the map $J_*(d)\ri J_{2}^d$.\\
	Furthermore, the $S_6$-action on $J_*(1)$ further induces an action on $L_{\T_*}^2$ and consequently on $|2\T_*|_\pi$.
	\end{rem}	
	\noindent
\subsection{The Universal Extension Space}\label{univ-ext-even}
We begin by extending the classical \cite[Theorem 3]{nar-ram} to the family $\pi$ and rephrasing the theorem in terms of the universal extension space as defined in \cite{shubham}.\\
In order to do that, we need to restate some definitions from \cite{nar-ram} in terms of the family $\pi$ using the diagram 
$$
\begin{tikzcd}
	J_*(1)&\C_*\times_{M_*}J_*(1)\arrow[r,"\tau"]\arrow[d,"p_C"]\arrow[l,"\pi_J"']&J_*(0)\times_{M_*}J_*(1)\arrow[d,"\mu"]\arrow[r,"\s"]&J_*(0)\times_{M_*}J_*(1)\\
&\C_*\arrow[r,"\T_*",hook]&J_*(1),&
\end{tikzcd}$$
where $\tau := (\pi_1-\pi_2, \pi_2)$ is defined using the canonical embedding $\T_*:\C_*\hr J_*(1)$, 
$\pi_J$ is the projection of $\C_*\times_{M_*}J_*(1)$ onto $J_*(1)$, $\s: J_*(0)\times_{M_*}J_*(1)\ri J_*(0)\times_{M_*}J_*(1)$ is the pullback of $[-1]:J_*(0)\ri J_*(0)$ to $J_*(1)$, 
and $\mu$ is the twisting morphism.\\
 Let $T:= Im(\tau)$ and $L_T$ denote the line bundle associated to $T$ on $J_*(0)\times_{M_*}J_*(1)$, and let $\xi_1$ be the Poincar\'e bundle on $\C_J:= \C_*\times_{M_*}J_*(1)$ rigidified along the section $\pi^*\s_1:J_*(1)\ri\C_J$.
\begin{lemma}\label{c_1}
	There is a line bundle $L_J$ on $J_*(1)$ such that $\tau^*\s^*L_T = \xi_1^{\otimes 2}\otimes \pi_J^*L_J$ with $c_1(L_J) = 4\T_*$.
\end{lemma}
\begin{proof}
	Existence of $L_J$ follows from the observation that $\tau^*\s^*L_T, \xi_1^{\otimes 2}$ have identical restrictions on fibers over $J_*(1)$ by \cite[Lemma 6.4]{nar-ram}.
	 To show $c_1(L_J) = 4\T_*$, we calculate $c_1(\tau^*\s^*L_T|_{\s_1\times_{M_*}J_*(1)})\in A^1(J_*(1))$.
	We have $$\tau^*\s^*L_T|_{\s_1\times_{M_*}J_*(1)} = [2]^*\T_*\implies c_1(\tau^*\s^*L_T|_{\s_1\times_{M_*}J_*(1)}) = 4\T_*\in A^1(J_*(1))$$ finishing the proof.
\end{proof}
	\begin{theorem}\label{isom}\normalfont
Let $\tilde{U}(2,\mathcal{O},2)$ be the moduli space of semistable rank $2$ bundles with trivial relative determinant for the family $\pi$, and let
$\mathcal{E} := R^1\pi_{J*}\xi_1^{-\otimes 2}$ be the universal extension bundle. We put $\p(\mathcal{E}) := Proj_{J_*(1)}(\Sym \mathcal{E}^\vee)$ to be its projectivization, and put $JU_*:=J_*(1)\times_{M_*}\tilde{U}(2,\mathcal{O},2)$ to get the diagram:\begin{equation}\label{diag_1}
		\begin{tikzcd}
			& JU_*\arrow[rd, "p_U", near start]\arrow[dd,"p_J"' , bend right = 60]& & \\
			\xi\arrow[d]&\p(\mathcal{E})\arrow[d]\arrow[u, hook]\arrow[r, "s"]\arrow[rr,"f" ,bend left =30]&\tilde{U}(2,\mathcal{O},2)\arrow[d, "\lambda"]\arrow[r, "D"] &\mid 2\Theta_*\mid_\pi \arrow[ld, "\lambda_\pi"]\\
			\C_J\arrow[r,"\pi_J"]&J_*(1)\arrow[r,"\pi"]&M_*&
		\end{tikzcd}\end{equation}
		where the morphisms $s,f,D$ are such that their restrictions to fibers over $M_*$ coincide with {\cite[Theorem 2]{nar-ram}}. Moreover, $D$ is an isomorphism.
	\end{theorem}
	\begin{proof}\normalfont
		We start with constructing each of these morphisms. We have the Poincar\'e bundle $\mathcal{P}$ on $\C_*\times_{M_*}\p(\mathcal{E})$ given by \cite[Remark 1]{shubham} applied to 
		$$\xi_1^{\otimes 2}\rightarrow\C_*\times_{M_*}J_*(1).$$ 
		The bundle $\mathcal{P}\otimes (1\times p_J)^*\xi_1^{-1}$ defines the map $s: \p(\mathcal{E})\rightarrow \tilde{U}(2,\mathcal{O},2)$. 
		Put $\mathcal{F}:= \pi_{J*}(\tau^*L_T\otimes \tau^*\s^*L_T)\otimes L_{\T_*}^{-2}$. We use the following exact sequence in \cite[Proposition 6.1]{nar-ram} to define $f, D$:
		$$0\rightarrow \mathcal{F}^\vee\rightarrow\pi^* \pi_* L_{\Theta_*}^2\rightarrow L_{\Theta_*}^2\rightarrow 0.$$
 Using Lemma \ref{c_1}, we have 
		$$\mathcal{F}^\vee\simeq L_{\T_*}^2\otimes R^1\pi_{J*}(\tau^*\s^*L_T)^\vee\simeq L_{\T_*}^2\otimes L_J^\vee\otimes \mathcal{E}.$$
		The injection defines an embedding $\p(\mathcal{E})\simeq \p(\mathcal{F}^\vee)\hookrightarrow \p_{J_*(1)}(\pi^* \pi_* L_{\Theta_*}^2) = |2\Theta_*|_\pi\times_{M_*}J_*(1)$, which in turn defines $f$.\\
		The image of the induced morphism $\p(\mathcal{E})\hookrightarrow JU_*$ is a Cartier divisor, let $\mathscr{L} := \mathcal{O}_{JU_*}(\p(\mathcal{E}))=\mathcal{I}_{\p(\mathcal{E})/JU_*}^\vee$ be the line bundle determined by this divisor.
		From the proof of \cite[Theorem 2]{nar-ram}, the restrictions of $\mathscr{L}\otimes p_J^*L_{\Theta_*}^{-2}$ to the fibers of $p_U$ are trivial. 
		So, $\mathscr{L}\otimes p_J^*L_{\Theta_*}^{-2}\simeq p_U^*L$ for some $L\in Pic(\tilde{U}(2,\mathcal{O},2))$. 
		We will show that there is an injection $L^\vee\hookrightarrow\lambda^*(\pi_*L_{\Theta_*}^2)$ and use this injection to define $D$.\\
		We have the inclusion $\mathcal{O}_{JU_*}\hookrightarrow\mathscr{L}$ induced by the inclusion $\mathcal{I}_{P(\mathcal{E})/JU_*}\subset \mathcal{O}$. 
		Pushing forward this injection via $p_U$, we get $\mathcal{O}\hookrightarrow p_{U*}\mathcal{L}\simeq L\otimes p_{U*}p_J^*L_{\Theta_*}^{2}$. 
		Twisting this by $L^\vee$, we get the desired injection $L^\vee\hookrightarrow p_{U*}p_J^*L_{\Theta_*}^{2}\simeq \lambda^*(\pi_*L_{\Theta_*}^2) $.\\
		By \cite[Theorem 2]{nar-ram}, we know that $D$ is a bijection. Therefore, $D$ must be an isomorphism by Theorem \ref{zmt}.
	\end{proof}
	\begin{rem}
	A version of the above theorem can be found in \cite{hitchin}, in arbitrary characteristic, without proof. 
	We've added a proof here for completeness and in order to define the explicit construction of the morphisms for the family $\pi$.
	\end{rem}
\noindent Using $A^*(M_*)\simeq \Q$, we have $A^*(|2\Theta_*|_\pi)\simeq \Q[\zeta]/(\zeta^4)$ where $\zeta$ is the relative hyperplane divisor in $|2\Theta_*|_\pi$.\\
Results on the Picard group of $U(2,\mo,2)$ in \cite{fringuelli} show that $\zeta$ can be identified with the class of the generalized $\Theta$-divisor on fibers over $M_*$. 
Since $Pic(M_*)=0$ , $[\T_{gen}]=\zeta$ in $A^1(\tilde{U}(2,\mo,2))$ where $\T_{gen}\hr \tilde{U}(2,\mo,2)$ is the pullback of the universal generalized $\T$-divisor, $\T_{2, gen}$, via the map $\tilde{U}(2,\mo,2)\xrightarrow{\otimes \mo(\s_1)}\tilde{U}(2,2,2)$.\\
We provide another proof below using \cite[Theorem 2]{nar-ram}.
\begin{theorem}\label{t-gen}
	Let $\T_{gen}^K$ be the universal generalized $\T$-divisor restrcited to $\tilde{U}(2,\mathcal{K}_\pi,2)\hr \tilde{U}(2,2,2)$. 
	The image of $\T_{gen}^K$ is a hyperplane bundle under the morphism $D\otimes\mo(-\s_1): \tilde{U}(2,\mathcal{K}_\pi,2)\xrightarrow{\otimes \mo(-\s_1)} \tilde{U}(2,\mo,2)\ri |2\T_*|_\pi$.
\end{theorem}
\begin{proof}
	We have $[E]\in \T_{gen}^K$ if and only if $h^0(E)>0$ for $[E]\in U(2,\mathcal{K}_\pi,2)$. Let $\T_{gen}^K\otimes \mo(-\s_1)$ be the divisor in $\tilde{U}(2,\mo,2)$ obtained by twisting bundles in the locus $\T_{gen}^K$ by $\mo(-\s_1)$. 
	It is easy to see that $\T_{gen} = \T_{gen}^K\otimes \mo(-\s_1)$. We show that $\T_{gen}^K\otimes \mo(-\s_1)$ maps to a hyperplane on every fibre over $M_*$ under $D$.\\
	Let $p\in M_*$ be a point and $[W]\in U(2,\mo,2)_p$. For $[W]\in \T_{gen}^K\otimes \mo(-\s_1)$, we must have 
	$h^0(\C_p, W\otimes\mo(\s_1)_p)>0$.
	That is, $$[W]\in \T_{gen}^K\otimes \mo(-\s_1)\iff (\s_1)_p\in C_W.$$
	The locus of divisors in $|2\T_{\C_p}|$ that contain the point $(\s_1)_p$ is a hyperplane. Therefore, $D_p$ maps $\T_{gen,p}^K\otimes \mo(-\s_1)_p$ onto a hyperplane in $|2\T_{\C_p}|$.\\
	Using $Pic(M_*)=0$, we can conclude that $[\T_{gen}\otimes \mo(\s_1)] = \zeta$ in $A^1(\tilde{U}(2,\mo,2))$.
\end{proof}

\subsection{Chow rings of $J_*(d)$}\label{uni-jac-chow}
 For arbitrary integers $d,d'$, twisting by $(d'-d)\sigma_1$ induces the isomorphisms $\tau_{d,d'}: J_*(d)\xrightarrow{\simeq} J_*(d')$. \\
We calculate $A^*(J_*(2))$ and using $\tau_{d,2}$, we would then have $A^*(J_*(d))$ for any $d\in\mathbb{Z}$. Furthermore, we can define an Abelian scheme structure on $J_*(2)$ over $M_*$ using the map on $J_*(0)$ and $\tau_{2,0}$. 
	We shall provide an explicit stratification to get generators for every codimension using the following lemma from \cite[Lemma 1.6]{verra}:
	\begin{lemma}\label{4.2}
		Let $C$ be a curve of genus $g$. Let $n_1,\cdots, n_g\in \Z$ be non-zero integers such that $n_1+\cdots+n_g = d$, then the map $a_{n_1,\cdots,n_g}: C^g\rightarrow Pic^d(C)$ given by $(x_1,\cdots, x_g)\mapsto n_1x_1+\cdots+n_gx_g$ is surjective.
	\end{lemma}\noindent
	We need to determine $A^*(\C_*)$ to obtain generators for the different strata in our stratification.
	\subsubsection{$\mathbf{A^*(\mathbf{\mathscr{C}}_*)}$}\label{C*}
	Let  $\mathcal{W}_*:= \coprod_{1\leq j\leq 6}\sigma_j\hookrightarrow\C_*$ be the locus of Weierstrass points. \\
	We have $2[\sigma_i]_t = 2[\sigma_j]_t\forall 1\leq i,j \leq 6,\forall t\in M_*$. That is, the restriction of $2([\sigma_i]-[\sigma_j])$ to any fiber over $M_*$ is trivial and, therefore, $2([\sigma_i]-[\sigma_j]) = \pi^*L_{ij}$ for some $L_{ij}\in Pic(M_*)$.  
	Since $M_*\subset\mathbb{A}^3$, we have that 
	$$[\s_i]=[\s_j]\in A^1(\C_*) \forall 1\leq i,j\leq 6.$$ 
	Additionally, there are no non-trivial cycles inside any of the $\sigma_i$.\\
	Using the stratification $\C_* = (\C_*\setminus \mathcal{W}_*)\coprod_i\sigma_i$, we get the excision sequence 
	$$\Q.\mathcal{K}\rightarrow A^*(\C_*)\rightarrow A^*(\C_*\setminus\mathcal{W}_*)\rightarrow 0$$  
	where $\mathcal{K} := [K_{\C_*/M_*}] = \dfrac{1}{3}\Sigma_{1\leq i\leq 6} [\sigma_i]= 2[\sigma_i]\forall 1\leq i\leq 6$.
	\begin{lemma}\label{c1}\normalfont
		The Chow ring of $\C_*$ is generated by pushforward of cycles in $\mathcal{W}_*$. That is,
		$$A^*(\C_*\setminus \mathcal{W}_*) = \Q.$$
	\end{lemma}
	\begin{proof}
		Let $(C,x)$ be a pointed curve of genus $2$ such that $x$ is not a Weierstrass point of $C$. The map $\phi_x$ given by 
		$$C\xrightarrow[|4x|]{}\p^2$$ 
		maps $C$ into the plane as a degree $4$ curve with a single point of singularity that is either nodal or cuspidal, and a $4$-flex at the image of $x$. Let $\tilde{y}$ be the conjugate of $y$ under the hyperelliptic involution for any $ y\in C$.\\
		To rigidify this map, let's fix the node to be $N = [1:0:0]$, the image of the point $x$ to be $ M =  [0:1:0]$, and the inflection line to be $l:x_1 = 0$ where the coordinates are given by $[x_1:x_2:x_3]$. 
		\\
		The equation of the image $\phi_x(C)\hookrightarrow \p^2$ then has the form $\lambda_1 x_1^2x_2^2+\lambda_2 x_1^2x_2x_3+\la_3x_3^2x_1^2 +\la_4x_1x_2^3+\la_5 x_1x_2^2x_3+\la_6 x_1x_2x_3^2+\la_7 x_1x_3^3+\la_8 x_3^4$ for some $\la_i\in \mathbb{C}$.\\
		The projective transformations fixing $N, \phi_x(x), l$ are given by the group $ \left\{ \left(  \begin{matrix}
			a&0&0\\ 0&b&c\\0&0&d
		\end{matrix} \right) \right\} \subset GL_3(\K)$ as described in \cite[\S 2]{casnati}. \\
		Since $N$ is the node, for any $y\in C$, the points $\{\phi_x(y), \phi_x(\tilde{y}), N\}\subset \mathbb{\p}^2$ are collinear.
	Intersection with the line $x_3=0$ is given by the ideal $(x_3, x_2^2x_1(\la_1x_1+\la_4x_2))$. 
		Since $x$ is not a Weierstrass point, we have  $\la_4\neq 0$.
		Furthermore, $\la_8\neq 0$ since the image of $\phi_x$ is integral.\\
		 The transformation $x_1\mapsto \la_4x_1, x_2\mapsto x_2, x_3\mapsto \la_4x_3+\dfrac{1}{3}\la_5x_2$ is an element of this group.
		 An application of this kills the coefficient of $x_1x_2^2x_3$.
		Additionally, we can normalize $\la_4, \la_8$ to $1$.\\
		The equation of $Im(\phi_x)$ then has the form $E_{C,x}:\A_1 x_1^2x_2^2+\A_2 x_1^2x_2x_3+\A_3x_3^2x_1^2 + x_1x_2^3+\A_4 x_1x_2x_3^2+\A_5 x_1x_3^3+ x_3^4$. \\
		To find images of Weierstrass points on $E_{C,x}$, we intersect lines passing through $N$. Equations of these lines are given by $\{x_2-cx_3 = 0 \mid c\in \K\}$ since the line $x_3 = 0$ does not correspond to Weierstrass points.\\
		Intersection with the line $x_2-cx_3=0$ is given by the ideal $(x_2-cx_3, x_3^2((c^2\A_1+c\A_2+\A_3)x_1^2+x_3^2+x_1x_3(\A_5+c\A_4+c^3)))$. To get a Weierstrass point, we must have \begin{equation}\label{weier}(\A_5+c\A_4+c^3)^2 = 4(c^2\A_1+c\A_2+\A_3)\end{equation} 
		The space of ordered roots to this polynomial is the set of points in $\mathbb{A}^6$, with mutually distinct co-ordinates adding up to $0$.\\
		Let $f:\mathbb{A}^5\rightarrow\mathbb{A}^6$ be the morphism given by $(\B_i)\mapsto (\B_i, -\Sigma_i \B_i)$ and $U\subset \mathbb{A}^6$ be the set of points with mutually distinct co-ordinates. \\
		Let $\tilde{U} = f^{-1}(U)$. We have the morphism $f: \tilde{U}\rightarrow U$ and two points $p,q\in \tilde{U}$ determine projectively related $E_{C,x}$ if and only if $\exists \la\in \K^*$ such that $p = \la. q$.
		The induced $\K^*$-action on the coefficients $(\A_i)_{1\leq i\leq 5}\in \mathbb{A}^5$ is given by $\la.(\A_1,\A_2,\A_3,\A_4,\A_5)= (\la^4\A_1, \la^5\A_2, \la^6.\A_3, \la^{2}.\A_4, \la^3.\A_5)$.\\
		Shrinking $\tilde{U}$ if necessary, we have that the coordinates of any point in $\tilde{U}$ are roots of some suitable polynomial as in (\ref{weier}) which determine the coefficients of some $E_{C,x}$. Let $E_U$ be the associated family of plane curves given by the $\{\A_i\}$.\\
		Let $\tau_M:= M\times \tilde{U}, \tau_N:= N\times \tilde{U}$ be the sections corresponding to the marked point and node respectively. 
		We consider the desingularization of $E_U$ at the singularity $\tau_N$ to obtain $bl:\C_U\ri E_U$, let $\tilde{\p}_U$ be the space obtained from the corresponding blowups of $\p^2\times \tilde{U}$.
		We get the following diagram:
		\begin{equation}\label{sections}
		\begin{tikzcd}
			\tilde{\p}_U\hookleftarrow \C_{U}\arrow[r, "bl"]\arrow[rd, "\tilde{\pi}"']&E_{U}\arrow[d]\arrow[r, hook]& \p^2\times \tilde{U}\\
			& \tilde{U}\arrow[u, "N\times \tilde{U}(=:\tau_N)"', bend right = 60]\arrow[u, "\tau_M", bend left = 40]\arrow[ul, "\overline{\tau_1}{,}\cdots{,}\overline{\tau_6}", bend left = 30]\arrow[ul, bend left = 70]\arrow[rd]\arrow[d]&\\
			M_*& \C_*\arrow[l,"\pi"]& \tilde{U}//\K^*\arrow[l, "\tilde{\tau}",dashed].
		\end{tikzcd}
	\end{equation}
		The map $\tilde{\pi}$ defines a family of smooth genus $2$ curves. For any $1\leq i\leq 6$, we have a family of lines parameterized over $\tilde{U}$ given by $\tau_i (p) = Z(x_2-f(p)_ix_3)\forall p\in \mathbb{A}^5$. \\
		The section $\tau_M$ lifts to $\C_U$ since $\C_U\setminus bl^{-1}(\tau_N) \simeq E_{U}\setminus \tau_N$.\\ 
		For every $\tau_i$, the strict transform $\tilde{\tau_i}\hookrightarrow \tilde{\p}_U$ intersects $\C_U$ on a locus that maps bijectively onto $\tilde{U}$ via $\tilde{\pi}$. 
		Hence, the locus $(\tilde{\tau_i}\cap \tilde{\p}_U)_{red}$ is isomorphic to $\tilde{U}$ by Theorem \ref{zmt}. Let the sections induced by these isomorphisms be $\overline{\tau_i}$.
The sections $\{\overline{\tau}_i\}$ provide markings for the Weierstrass points in the family $\tilde{\pi}$.
		\\By the universality of $M_*$, we have the induced morphism $t:\tilde{U}\ri M_*$. Additionally, $t$ must lift to a morphism of families
		$ t_U: \C_U\rightarrow \C_*$ such that the square below is Cartesian. 
		  \begin{equation*}
			\begin{tikzcd}\label{sec}
				\tilde{U}\arrow[r,"\tau_M"] &\C_U\arrow[d,"\tilde{\pi}"]\arrow[dr, phantom, "\square"]\arrow[r, "t_U"]&\C_*\arrow[d,"\pi"]\\
				&\tilde{U}\arrow[r, "t"]&M_*
		\end{tikzcd} \end{equation*}
		The map $t_U\circ\tau_M$ is preserved under the $\K^*$-action and thus factors via $\tilde{\tau}:\tilde{U}//\K^*\rightarrow \C_*$. The map $\tilde{\tau}$ is a bijection between $\tilde{U}//\K^*$ and $\C_*\setminus \mathcal{W}_*$. Hence, we have $\tilde{U}//\K^*\simeq \C_*\setminus \mathcal{W}_*$. Additionally,
		$\tilde{U}//\K^*\subset \mathbb{A}^4$ is open since $\tilde{U}\subset \mathbb{A}^5$ consists of points with distinct coordinates, establishing the claimed statement.
	\end{proof}\noindent
	Hence, we have that $A^i(\C_*) = 0\forall i>1$ and $A^1(\C_*)=\Q.\mathcal{K}$. 	Restriction of the divisor associated to $\mathcal{K}$, to a fiber of $\pi: \C_*\rightarrow M_*$ yields $\mathcal{K}\neq 0$ in $A^1(\C_*)$. Using $\mathcal{K}  = 2\sigma_j\forall j$, we have $\dim A^1(\C_*) = 1$ and the following corollary.
	\begin{cor}
 The excision sequence for $\C_*$ yields
	$$A^*(\C_*) = \Q[\mathcal{K}]/(\mathcal{K}^2) = \Q[\sigma_j]/(\sigma_j^2)\forall 1\leq j\leq 6.$$		
	\end{cor}
	\subsubsection{$\mathbf{A^*(J_*(2))}$}\label{j-2}
	We have the morphism $\C_*\times_{M_*}\C_*\xrightarrow{\phi} J_*(2)$, given by the sequence of maps $\C_*\times_{M_*}\C_*\hr J_*(1)\times_{M_*}J_*(1)\ri J_*(2)$ where the last map is the twisting morphism. 
	Using Lemma \ref{4.2}, we have that this map is surjective and proper over $M_*$. We also have the map $i_{[2]}:\C_*\rightarrow J_*(2)$ 
	given by the sequence of morphisms $\C_*\hr J_*(1)\xrightarrow{\Delta}J_*(1)\times_{M_*}J_*(1)\ri J_*(2)$.\\
	We make a similar argument as presented in Lemma \ref{c1} to prove the following lemma.
	\begin{lemma}{\label{chow-2}}
		The Chow ring of $J_*(2)$ is generated by cycles in $Z_\C, \phi(\mathcal{W}_*\times_{M_*}\C_*)$ where $Z_\C := Im(i_{[2]})$. That is,
	$$A^*(J_*(2)\setminus (Z_\C\cup\phi(\mathcal{W}_*\times_{M_*}\C_*)))\simeq \Q.$$
	\end{lemma}
	\begin{proof}
	Let $(C,x,y)$ be a curve with two distinct marked points such that $x$ is not a Weierstrass point of $C$ and $y\neq \tilde{x}$.
	Using similar arguments as in Lemma \ref{c1}, the map given by 
	$$C\xrightarrow[\phi_{x,y}]{|3x+y|}\p^2,$$ embeds $C$ into the plane as a degree $4$ curve with a single point of singularity that is either nodal or cuspidal. We rigidify this map by setting the node to be $N, \phi_{x,y}(x)=M$ and $\phi_{x,y}(y)= P :=[0:0:1]$ (the inflection line at $M$ still has the same equation as in Lemma \ref{c1}).\\
	The equation of $\phi_{x,y}(C)\hookrightarrow \p^2$ then has the form $\lambda_1 x_1^2x_2^2+\lambda_2 x_1^2x_2x_3+\la_3x_3^2x_1^2 +\la_4x_1x_2^3+\la_5 x_1x_2^2x_3+\la_6 x_1x_2x_3^2+\la_7 x_1x_3^3+\la_8 x_2x_3^3$ for some $\{\la_i|1\leq i\leq 8\}\subset \mathbb{F}$.\\
	The projective transformations fixing $N, \phi_{x,y}(x), \phi_{x,y}(y)$ are given by the group $ \left\{ \left(  \begin{matrix}
		a&0&0\\ 0&b&0\\0&0&c
	\end{matrix} \right) \right\} \subset GL_3$ as described in \cite[\S 2]{casnati}. \\
	Since $N$ is the singular point, we have  $\{\phi_{x,y}(a), \phi_{x,y}(\tilde{a}), N\}$ are collinear for any $ a\in C$. Intersection with the line $x_3=0$ is the scheme given by the ideal $(x_3, x_2^2x_1(\la_1x_1+\la_4x_2))$. 
	Therefore, $\la_4\neq 0$ since $x$ is not a Weierstrass point, and $\la_8\neq 0$ since $Im(\phi_{x,y})$ is irreducible.\\
	We normalize $\la_4,\la_8 = 1$.  The equation of $Im(\phi_x)$ then has the form $E_{C,x}:\A_1 x_1^2x_2^2+\A_2 x_1^2x_2x_3+\A_3x_3^2x_1^2 + x_1x_2^3+\A_4 x_1x_2x_3^2+\A_5 x_1x_2^2x_3+\A_6 x_1x_3^3+ x_2x_3^3$.\\
	The intersection with the line $Z(x_2-cx_3)$ is given by the ideal $(x_2-cx_3, x_3^2((c^2\A_1+c\A_2+\A_3)x_1^2+cx_3^2+x_1x_3(c^2\A_5+c\A_4+c^3+\A_6)))$.
	 To get a Weierstrass point, we must have \begin{equation}\label{weier2}(c^2\A_5+c\A_4+c^3+\A_6)^2 = 4c(c^2\A_1+c\A_2+\A_3).\end{equation} 
	Let $U_0$ be the (ordered)space of roots to all such polynomials in $\mathbb{A}^6$, with mutually distinct co-ordinates and $U'=\{(p_i)_{1\leq i\leq 6}|(p_i^2)_{1\leq i\leq 6}\in U_0\}$.\\
	Two points $p,q\in U_0$ determine projectively related $E_{C,x, y}$ if and only if $\exists \la\in \K^*$ such that $p = \la. q$. The corresponding $\K^*$-action on $(\A_i)$ is given by 
	$$\la.(\A_1,\A_2,\A_3,\A_4,\A_5, \A_6)= (\la^{3}\A_1, \la^{4}\A_2, \la^{5}.\A_3, \la^{2}.\A_4, \la.\A_5,\la^3.\A_6 ).$$
	For a point $p=(p_i)_{1\leq i\leq 6}\in {U}'$, we have that $\{p_i^2\}$ are roots of some $E_{C,x,y}$ as described in (\ref{weier2}). 
	We can recover the coefficients $\{\A_i\}$ from $p$ by setting $\A_6$ to be $\Pi_{1\leq i\leq 6} p_i$.\\
	Following the notation established in (\ref{sections}) along with the section $\tau_P:= P\times U'$, and desingularizing $E_{U'}$ along $\tau_N$, we get a family $\C_{U'}$ of smooth genus $2$ curves. 
	Let $\p_{U'}$ be the space obtained from the corresponding blowups of $\p^2\times U'$.
	We have the diagram:
	$$\begin{tikzcd}
		\p_{U'}\hookleftarrow \C_{U'}\arrow[r, "bl"]\arrow[rd, "\tilde{\pi}"']&E_{U'}\arrow[d]\arrow[r, hook]& \p^2\times U'\\
		& {U}'\arrow[u, "N\times {U}'(=:\tau_N)"', bend right = 60]\arrow[u, "\tau_M{,}\tau_P", bend left = 40]\arrow[ul, "\overline{\tau_1}{,}\cdots{,}\overline{\tau_6}", bend left = 30]\arrow[ul, bend left = 70]\arrow[rd]\arrow[d,"\tau"]&\\
		M_*& \C_*\times_{M_*}\C_*\arrow[l,"\pi"]& {U'}//\K^*\arrow[l, "\tilde{\tau}",dashed].
	\end{tikzcd}$$ 
Furthermore, we have lines parametrized over ${U}'$ given by $\tau_i (p) = Z(x_2-p_i^2x_3)$. \\
	The sections $\tau_M,\tau_P$ lift to $\C_{U'}$ since $\C_{U'}\setminus bl^{-1}(\tau_N) \simeq E_{U'}\setminus \tau_N$.\\ 
	For every $\tau_i$, the strict transform $\tilde{\tau_i}\hookrightarrow \p_{U'}$ intersects $\C_{U'}$ so that the intersection maps bijectively onto ${U}'$ via $\tilde{\pi}$. Hence, the locus $(\tilde{\tau_i}\cap \p_{U'})_{red}$ is isomorphic to ${U}'$ by Theorem \ref{zmt}. Let the sections induced by these isomorphisms be $\overline{\tau_i}$.
	\\ Hence, we have a morphism ${U}'\xrightarrow{t} M_*$. The sections $\tau_M, \tau_P$  then define the maps $\sigma_M,\sigma_P:{U}'\xrightarrow{} \C_*$. These maps define $\tau: U'\rightarrow\C_*\times_{M_*}\C_*$ which is invariant under the $\K^*$-action and thus factors via $\tilde{\tau}:{U}'//\K^*\rightarrow \C_*\times_{M_*}\C_*$.\\ 
	Let $\tilde{\Delta}$ be the image of an embedding of $\C_*$ into $\C_*\times_{M_*}\C_*$ with projection onto first factor given by the identity map and projection onto second factor given by the hyperelliptic involution.\\ 
	The map $\tilde{\tau}$ is finite, surjective onto $\C_*\times_{M_*}\C_*\setminus (\Delta\cup \tilde{\Delta}\cup(\mathcal{W}_*\times_{M_*}\C_* ))$.\\
	Additionally, ${U}'//\K^*\subset \mathbb{A}^5$ since ${U}'\subset \mathbb{A}^6$ consists of points with distinct coordinates. Therefore, $$A^*(\C_*\times_{M_*}\C_*\setminus (\Delta\cup \tilde{\Delta}\cup(\mathcal{W}_*\times_{M_*}\C_* )))\simeq \Q$$\label{c_2} 
	\noindent by Remark \ref{rem1}. Consequently, we have
	\begin{multline*}\phi^{-1}(Z_\C\cup \phi(\mathcal{W}_*\times_{M_*}\C_*)) = \Delta\cup \tilde{\Delta}\cup (\mathcal{W}_*\times_{M_*}\C_*)\cup (\C_*\times_{M_*}\mathcal{W}_*)\implies  \phi^{-1}(J_*(2)\setminus (Z_\C\cup\phi(\mathcal{W}_*\times_{M_*}\C_*)))\\
		 \subset \C_*\times_{M_*}\C_*\setminus (\Delta\cup \tilde{\Delta}\cup(\mathcal{W}_*\times_{M_*}\C_* ))\implies A^*(\phi^{-1}(J_*(2)\setminus (Z_\C\cup\phi(\mathcal{W}_*\times_{M_*}\C_*))))=\Q.\end{multline*}
	The restriction of $\phi$ to $\phi^{-1}(J_*(2)\setminus (Z_\C\cup\phi(\mathcal{W}_*\times_{M_*}\C_*)))\rightarrow J_*(2)\setminus (Z_\C\cup\phi(\mathcal{W}_*\times_{M_*}\C_*))$ is a finite, surjective map between smooth varieties. The lemma now follows from Remark \ref{rem1}.
\end{proof}
	\begin{theorem}\label{t-s6}
The Chow ring of $J_*(d)$ is given by 
$$A^*(J_*(d))\simeq \Q[\T_d]/(\T_d^3)$$ for all $d\in \mathbb{Z}$. Consequently, the Chow rings of $J_{2}^d$ are given by 
$$A^*(J_{2}^1)\simeq \Q[\T]/(\T^3), A^*(J_{2}^2)\simeq \Q[Z]/(Z^3)$$\label{chow-j-2} where $\Theta$ is the theta-divisor in $J_{2}^1$ and $Z\hr J_{2}^2$ is the image of $Z_\C$ under the map $J_*(2)\ri J_{2}^2$.
	\end{theorem}
\begin{proof}
The deleted locus from $J_*(2)$ in Lemma \ref{chow-2} is $$Z_\C\cup \phi(\mathcal{W}_*\times_{M_*}\C_*) = i_{[2]}(\C_*)\cup_i \phi(\sigma_i\times_{M_*}\C_*)  = i_{[2]}(\C_*)\cup_j \phi_j(\C)$$ where $\phi_j := \C_*\simeq \sigma_j\times_{M_*}\C_*\xrightarrow{\phi} J_*(2) \forall 1\leq j\leq 6$. 
The maps $i_{[2]}, \phi_j$ are finite and $i_{[2]} = [2]\circ \phi_j\forall 1\leq j\leq 6$ where $[2]:J_*(2)\ri J_*(2)$ is given by $[(L,C,s_1,\cdots,s_6)]\mapsto [(L^{\otimes 2}\otimes K_C^\vee,C,s_1,\cdots,s_6)]\forall [(L,C,s_1,\cdots,s_6)]\in J_*(2)$.\\
	We have $\phi_{j*}(A_*(\C_*)) = \Q.[\phi_j(\C_*)]\oplus \Q.e\quad\forall 1\leq j\leq 6$, where $e$ is the identity section of the Abelian scheme $J_*(2)/M_*$ since $\phi_{j*}(\mathcal{K}) = \phi_{j*}(2\sigma_j) = 2\phi_{j*}(\sigma_j) = 2e\forall j$ and $i_{[2]*}(\mathcal{K}) = [2]_*(e)= e$.
	\\Using \cite[Theorem 2.19]{deninger}, $A^1(J_*(2)) = \oplus_{-1\leq s\leq 2}A^1_s(J_*(2))$. Let $\phi_j(\C_*) = \Sigma_{-1\leq s\leq 2} a_{j,s}\forall j$.  We have 
	\begin{gather*}[2]_*(\phi_j(\C_*)) = i_{[2]}(\C_*)\forall j\implies \oplus_{-1\leq s\leq 2} 2^{2+s}a_{j,s} = \oplus_{-1\leq s\leq 2} 2^{2+s}a_{k,s}\forall 1\leq j,k\leq 6\implies \\
		2^{2+s}a_{j,s}=2^{2+s}a_{k,s}\forall -1\leq s\leq 2\implies a_{j,s}=a_{k,s}\forall s, j,k\implies [\phi_j(\C_*)] = [\phi_k(\C_*)]\in A^1(J_*(2))\forall 1\leq j,k\leq 6.\end{gather*}
	Using the excision sequence on $J_*(2)  = (J_*(2)\setminus Z_\C\cup \phi(\mathcal{W}_*\times_{M_*}\C_*))\coprod (Z_\C\cup \phi(\mathcal{W}_*\times_{M_*}\C_*))$, we have
	 $$\dim A^i(J_*(2))=0\forall i>2, \dim A^2(J_*(2)) = \dim A^2_{0}(J_*(2))\leq 1, \dim A^1(J_*(2))\leq 2.$$ 
	We show 
	$[\phi_1(\C_*)]\in A^1_{0}(J_*(2))$. This would further imply
	$$i_{[2]*}(\C_*) = [2]_* ( [\phi_1(\C_*)] ) = 4[\phi_1(\C_*)].$$ 
	Consequently, we shall have $\dim A^1(J_*(2))\leq 1$.\\
	We have that $J_*(2)$ is a principally polarized Abelian scheme via the $\Theta_*$ divisor on $J_*(1)$. 
	Hence, the Fourier transform $\mathcal{F}:A_s^p(J_*(2))\simeq A^{2-p+s}_s(J_*(2))$ is defined for all $p,s$, and induces the isomorphisms  
	\begin{gather}\label{beauville-decomp}A^1_{-1}(J_*(2)) \simeq A^0_{-1}(J_*(2)) = 0, A^1_{1}(J_*(2)) \simeq A^2_1(J_*(2)) = 0 ,\\
		 A^1_2(J_*(2)) \simeq A^3_2(J_*(2)) = 0\implies A^1(J_*(2)) = A^1_0(J_*(2))\label{beauville-component}\end{gather} 
		 as claimed and $[Z_\C] = [2]_*([\phi_1(\C_*)])= 4[\phi_1(\C_*)]$.
	We have $$[\phi_1(\C_*)] =  \tau_{2,1}^*\Theta_* \in A^1(J_*(2)).$$
	 Restricting to a fiber of $J_*(1)$ over $M_*$, shows that $\Theta_*, \Theta_*^2\neq 0$ in $A^*(J_*(1))$ giving us the claimed form for $A^*(J_*(d))$ for $d\in\Z$.\\ 
	The claimed form for $A^*(J_{2}^1)$ follows from the fact that the divisor $\Theta_*\hookrightarrow J_*(1)$ is the pullback of the universal $\T$-divisor. Similarly, 
	\begin{equation*}A^*(J_*(2))= \Q[Z_\C]/(Z_\C^3)\implies A^*(J_{2}^2)\simeq \Q[Z]/(Z^3)\end{equation*} 
	since the cycle $Z_\C(= 4\T_2)$ is the pullback of $Z\in A^1(J_{2}^2)$.
\end{proof}
\noindent We thus have $A^*(J_{2}^d)$ for all values of $d$. These generators and relations are identical to the ones obtained in \cite{larson-24}.
	\subsection{\textbf{Final Steps}}\label{final-steps}
	Returning to the original problem, we use that $|2\Theta_*|_\pi\times_{M_*}J_*(d)$ has a finite map to $U(2,2d,2)$.
	More conceretely, we have a sequence of maps
	$$|2\T_*|_\pi\times_{M_*}J_*(d)\xrightarrow{q_\Gamma}\tilde{U}(2,2d,2)\xrightarrow[]{q}U(2,2d,2)$$
	where $q$ is the map induced by the universality of $U(2,2d,2)$. A straightforward application of the projective bundle formula gives us the following lemma.
	\begin{lemma}\label{prod}
		The product $|2\Theta_*|_\pi\times_{M_*}J_*(d)$ is a projective-bundle over $J_*(d)$. Therefore, 
		\begin{equation*}A^*(|2\Theta_*|_\pi\times_{M_*}J_*(d))= \Q[A, B]/(A^4,B^3)\end{equation*}
		where $ \pi_\T:|2\Theta_*|_\pi\times_{M_*}J_*(d)\ri |2\T_*|_\pi, \pi_d: |2\Theta_*|_\pi\times_{M_*}J_*(d)\ri J_*(d)$ are the projection maps and $A:= \pi_\T^*\zeta, B:= \pi_d^*\Theta_d, \Theta_d=\tau_{d,1}^*\Theta_*$.
	\end{lemma}
	\begin{theorem}\label{u-even} 
		Let $\nu\in A^1(\tilde{U}(2,2d,2))$ be the divisor which pulls back to $\pi_\T^*\zeta$ via $q_\Gamma^*$ and $\tilde{Z}_d := det_d^*\tau_{2d,2}^*Z$. We have 
	$$A^*({U}(2,2d,2))\simeq \Q[\nu, \tilde{Z}_{d}]/(\nu^4,\tilde{Z}_{d}^3)$$
	for all $d\in\mathbb{Z}$. Furthermore, $\nu+\dfrac{1}{8}\tilde{Z}_d$ can be identified with the generalized $\T$-divisor if $d$ is odd.
	\end{theorem}
 \begin{proof}
	Applying the projective bundle formula, we use $A^*(J_*(d))$ to determine $A^*(U(2,2d,2))$.
	We consider the action of $\Gamma$ on $|2\Theta_*|_\pi\times_{M_*}J_*(d) $.\\ 
	Let $\alpha\in \Gamma$ be an arbitrary element, we show that $\alpha.\zeta = \zeta$.\\
	We have $A^*(\tilde{U}(2,\mathcal{O},2))=\Q[\zeta]/(\zeta^4)$ and therefore, $\alpha.\zeta = l.\zeta$ for some $l\in\Q$. Furthermore
	$$\alpha^{2} = id \implies l^{2} = 1\implies l =\pm 1.$$
	It is enough to show that $l\neq -1$. Suppose $l=-1$, then 
	$$\alpha.\zeta = -\zeta\implies \alpha.\eta =-\eta$$ where $\eta$ is the restriction of $\zeta$ to a fiber over $p = [(C,(p_i))]\in M_*$. Positivity of $\eta$ gives the desired contradiction and we have that $l=1$ .
	Hence, $\zeta$ is invariant under the $\Gamma$-action.\\
	 Let $ t_\lb(W)$ be the (relative)translation of a divisor $W\ri J_*(d)$ by $\lb\in Pic(\C_*)$ for the abelian scheme $J_*(d)\ri M_*$. Since
	$$[2]_*(\T_d) = [2]_*(t_{\mo(\s_i-\s_j)}(\T_d))\quad\forall 1\leq i,j\leq 6, d\in \mathbb{Z},$$ 
invariance of $\Theta_d$ under the $\Gamma$-action is an immediate consequence of (\ref{beauville-component}).\\
	The image of the $\Gamma$-quotient of the pullback of $\T_d\hr J_*(d)$ to $|2\T_*|_\pi\times_{M_*}J_*(d)$ is $\det_d^{-1}(\tilde{Z}_d)\hr\tilde{U}(2,2d,2)$ as explained in the diagram below
	$$\begin{tikzcd}
		\mid 2\T_*\mid_\pi \times_{M_*} J_*(d)\arrow[d,"\pi_d"]\arrow[r,"q_\Gamma"] & \tilde{U}(2,2d,2)\arrow[d, "det_d"]\\
		J_*(d)&J_*(2d).
	\end{tikzcd}$$
	That is, \begin{equation}\label{t-Z}
		\pi_d^*\T_d = \dfrac{1}{16}q_\Gamma^*\tilde{Z}_d\end{equation}
		 and $q_\Gamma^*$ is an isomorphism by Lemma \ref{prod}.
	In order to calculate $A^*(U(2,2d,2))$, we show that $\nu, \tilde{Z_d}$ are in the image of $q^*$.\\
We have $\dim A^1(U(2,2d,2)) = 2$ by \cite[Theorem A.2]{fringuelli}. 
Therefore, $q^*: A^1(U(2,2d,2))\hr A^1(\tilde{U}(2,2d,2))$ is an isomorphism since $q$ is finite.\\
Lastly, we show that, for $d$ odd,
$$q_\Gamma^*(\nu+\dfrac{1}{8}\tilde{Z}_d)-q_\Gamma^*\circ q^*(\T_{gen}) = 0.$$
This is equivalent to showing that 
$$q_\Gamma^*\circ q^*(\T_{gen}) = \pi_\T^*\zeta+2\pi_d^*\T_d. $$
Applying Theorem \ref{t-gen} and Lemma \ref{prod}, $q_\Gamma^*\circ q^*(\T_{gen}), \pi_\T^*\zeta$ have identical restrictions over $\tau_{d,0}(e)\hr J_*(d)$ via $\pi_d$.
Hence, $$q_\Gamma^*\circ q^*(\T_{gen}) = \pi_\T^*\zeta+a\pi_d^*\T_d $$ 
for some $a\in \Q$. Restricting to a fiber over $p = [C,(s_i)_{1\leq i\leq 6}]\in M_*$, let $q_{C}: U_C(2,\mo)\times J^1(C)\ri U_C(2,2)$ be the restriction of the quotient map $q_\Gamma$ over $p$ and $h\in A^1(U_C(2,\mo))$ be the hyperplane divisor.
We show that 
$$q_{C}^*\T_{gen, C} = \pi_U^*h+2\pi_{JC}^*\T_C$$
 where $\pi_U,\pi_{JC}$ are projections onto $U_C(2,\mo), J^1(C)$, respectively, and $\T_{gen, C}$ is the restriction of $\T_{gen}$ over $p$. In order to do this, we show that 
\begin{equation}\label{cohom}\pi_{JC*}(\pi_U^*h^3.q_{C}^*\T_{gen, C}) = 2\T_C \in H^2(J^1(C), \Q)\end{equation}
via the cycle-class map $A^1(J^1(C))\ri H^2(J^1(C), \Q)$.\\
Let $p\in C\setminus \{s_i\}_{1\leq i\leq 6}$ be a general point of $C$. Let $[E]\in SU_C(2)$ be the bundle defined by a non-trivial extension 
$$0\ri \mo(p-s_1)\ri E\ri \mo(s_1-p)\ri 0.$$
We have $\pi_U^*h^3 = [E]\times J^1(C)$. The transverse intersection of $[E]\times J^1(C), q_C^*\T_{gen, C}$ is given by the locus $\{(E,L)| h^0(E\otimes L)>0\}_{L\in J^1(C)}$. Twisting by $L$, we get 
\begin{gather*}
0\ri L(p-s_1)\ri E\otimes L\ri L(s_1-p)\ri 0.\end{gather*}
\text{Therefore, }$h^0(E\otimes L)>0\iff h^0(L(p-s_1))>0 \text{ or } h^0(L(s_1-p))>0$.
The cycle $\pi_{JC*}(\pi_U^*h^3.q_{C}^*\T_{gen, C})$, is hence the union of the loci $\T_C(p-s_1), \T_C(s_1-p)$ which are translations of $\T_C$ by the divisors $\mo(p-s_1), \mo(s_1-p)$ respectively.\\
The claimed equation (\ref{cohom}) follows consequently, showing that $a=2$. 
 \end{proof}
	\noindent We shall provide explicit expressions of these generators in terms of tautological classes in \S\ref{tautological}. 
	This provides an alternate method to establish $\dim A^1(U(2,2d,2)) = 2$ and the fact that $q^*$ is an isomorphism.
\subsection{The stable locus}\label{stable-locus}
We determine $A^*(U^s(2,2d,2))$ using excision and Theorem \ref{u-even} where $U^s(2,2d,2)\subset U(2,2d,2)$ is the moduli space of stable bundles of rank $2$, and degree $2d$ over genus $2$ curves.
Let $i^{ss}:U^{sss}(2,2d,2)\hookrightarrow U(2,2d,2)$ be the locus of strictly semi-stable bundles in $U(2,2d,2)$. The excision sequence is given by 
$$A^{*-1}(U^{sss}(2,2d,2))\xrightarrow[]{i_{*}^{ss}} A^*(U(2,2d,2))\ri A^*(U^s(2,2d,2))\ri 0.$$
We shall show that the image of $i_*^{ss}:A^{*-1}(U^{sss}(2,2d,2))\ri A^*(U(2,2d,2))$ is given by the ideal $(\nu)$ which would prove the following theorem.
\begin{theorem}\label{u-stable}
	Following the notation in Theorem \ref{u-even}, we have
	$$A^*(U^s(2,2d,2))\simeq \Q[\tilde{Z}_d]/(\tilde{Z}_d^3)$$
	for all $d\in\mathbb{Z}$.
\end{theorem}
\begin{proof}
	We use that $q_\Gamma^*,q^*$ are isomorphisms as established in Theorem \ref{u-even} to get
$$q_\Gamma^* \circ q^*i_{*}^{ss}(A^*({U}^{sss}(2,2d,2))) = i_{J*}^{ss}A^*(\tilde{U}^{sss}(2,\mo,2)\times_{M_*}J_*(d))$$
where $i_J^{ss}$ is the inclusion of $\tilde{U}^{sss}(2,\mo,2)\times_{M_*}J_*(d)\hr \tilde{U}(2,\mo,2)\times_{M_*}J_*(d)$.\\
The inclusion $i_J^{ss}$ is the pullback of the inclusion $i_U^{ss}:\tilde{U}^{sss}(2,\mo,2)\hr \tilde{U}(2,\mo,2)$ via projection.\\
Using \cite[Theorem 4]{nar-ram}, we know that $(i_U^{ss})_p$ maps onto a quartic surface in $|2\T_{\C_p}|$ over every point $p\in M_*$.
Hence, $[\tilde{U}^{ss}(2,\mo,2)] = 4\zeta$ in $A^1(\tilde{U}(2,\mo,2))$ and $\nu \in Im(i_*^{ss})$.\\
Therefore, it is enough to show that $\pi_d^*\T_d^2\neq 0$ in $A^*(\tilde{U}^s(2,\mo,2)\times_{M_*}J_*(d))$ for all $d\in \mathbb{Z}$.\\
Using the isomorphisms $\{\tau_{d,1}\}_{d\in \mathbb{Z}}$, it is enough to show this inequality just for $d=1$.\\
We prove by restricting to a fiber over $M_*$.\\
Fixing a $p_0 = [C, (s_i)_{1\leq i\leq 6}]\in M_*$, it is enough to show that $\pi_J^*\T_C^2\neq 0$ in $A^2(U_C(2,\mo)\times J^1(C))$ which follows from the Künneth formula applied to $H^4(U_C(2,\mo)\times J^1(C))$ via the cycle-class map.  
\end{proof}

	\section{Odd degree}\label{odd-degree}\noindent
	We calculate $A^*(U(2,l,2))$ for $l$ odd in this section.
	 Suitable twisting by the relative canonical divisor and taking dual if necessary shows that $U(2,l,2)\simeq U(2,3,2)$ for all odd integers $l$.
	  It is therefore enough to calculate $A^*(U(2,3,2))$.
	\subsection{Bertram's construction}\label{bertram-construct} 
	Let $C$ be a curve of genus $2$ and $L$ be a line bundle on $C$ of odd degree. Following \cite{bertram}, we use the following notation:
	\begin{defi*}
	Let $U_C(2,L)$ be the moduli space of stable bundles of rank $2$ and determinant $L$, and let $\p_L:= \p(Ext^1(L,\mathcal{O})^\vee)$ be the projectivized space of extension classes.
	\end{defi*}\noindent
	We have a rational map $\phi_L: \p_L\DashedArrow U_C(2,L)$ given by the Poincar\'e bundle on $C\times\p_L$, defined on the subset of extensions $[0\rightarrow \mathcal{O}\rightarrow E\rightarrow L\rightarrow 0]$ where $E$ is stable. 
	\begin{defi*}
	Let $B_L := \{E| h^0(C, E)\geq  2\}\hookrightarrow U_C(2,L)$ be the Brill-Noether locus in $U_C(2,L)$ of bundles with at least $2$ independent sections.
		We put $E_L := \Phi_L^{-1}(B_L)$ and $Q_L$ to be the image of $E_L$ under $\s$.
	\end{defi*}\noindent
Let $\mathcal{W}_{2,3}^k\hr U_C(2,3)$ be the Brill-Noether locus of stable bundles with rank $2$ and degree $3$ having atleast $k+1$ independent sections.
Using \cite[Theorem 4.3]{brill}, we have the following:
	\begin{itemize}
		\item $\mathcal{W}_{2,3}^2$ is empty 
		\item $\mathcal{W}_{2,3}^1$ is irreducible, smooth and has dimension $3$.
	\end{itemize}
	Therefore, $h^0(C,E)=2\forall [E]\in B_L$.
	\begin{theorem}[\cite{bertram}, Theorem 1 and Lemma 4.1]\label{main-bertram}
		The blowup of $\p_L$ along $C$, for $deg(L)=3$, is such that the following diagram commutes:
		$$
		\begin{tikzcd}
			\tilde{\p}_L\arrow[d,"\s"]\arrow[rd,"\Phi_L"]&\\
			\p_L\arrow[r,"\phi_L", dashed]& U_C(2,L)
		\end{tikzcd}
		$$ 
		where the embedding of $C$ into $\p_L$ is given by $C\xrightarrow{|K_C+L|}\p_L$. We have the following:
		\begin{enumerate}
			\item If $x\in C$, then there is a natural isomorphism $\sigma^{-1}(x)\simeq {\p}_{L(-2x)}$, and, when restricted to $\sigma^{-1}(x)$, 
			$\Phi_L$ coincides with the map $$\Phi_{L(-2x)}\otimes \mathcal{O}(x):{\p}_{L(-2x)}\rightarrow U_C(2,L)$$
			\item The map $\Phi_L$ resolves the rational map $\phi_L$ and is defined via a Poincar\'e bundle on $C\times\tilde{\p}_L$. 
\item The exceptional locus of $\Phi_L$ is $\Phi_L^{-1}(B_L)$
	 and we have 
	$$\Phi_L^{-1}([E])\simeq \p^1\quad\forall [E]\in B_L.$$
		\end{enumerate}
	\end{theorem}
	\begin{lemma}\label{surject}
 The restriction of $\Phi_L$ to $E_L\cap \sigma^{-1}(C)$ maps surjectively onto $B_L$.
	\end{lemma}
	\begin{proof}
		Let  $\tilde{\delta}\in E_L\setminus \s^{-1}(C)$ be such that $\Phi_L(\tilde{\delta}) = [E]\in B_L$, and let $\delta:= \s(\tilde{\delta})$. 
		We show that $\exists p\in C$ and $\gamma \in \s^{-1}(p)$ such that $\Phi_L(\gamma) = [E]$.
		By definition, $E$ is an extension 
		$$0\ri \mathcal{O}\xrightarrow{\theta} E\xrightarrow{\la} L\ri 0 $$
		 such that $h^0(C,E)\geq 2$. 
		Hence, $\exists\alpha\in H^0(C,L)\setminus 0$ which lifts via $\lambda$. The section $\alpha $ of $L$ must factor via an inclusion $\mo\ri \mo (p)\xrightarrow{i}L$
		for some $p\in C$. With slight abuse of notation, we denote the lift by $\alpha$.
		Hence, we have the following diagram $$\begin{tikzcd}
			& & & &\mathcal{O}\arrow[ld,"\alpha" ,dashed]\arrow[r,  "", hook]&\mathcal{O}(p)\arrow[ld, "", ]\arrow[ld, "i"']	\\
			\delta : & 0\arrow[r] &\mathcal{O}\arrow[r,"\theta"]&E\arrow[r]&L\arrow[r]&0.
		\end{tikzcd}
		$$  We must have $\delta \in ker(H^1(C, L^\vee)\ri H^1(C, \mo))$ since $i|_\mo$ lifts to $ E$. 
		But $H^1(C, L^\vee)\ri H^1(\mo)$ factors as $H^1(C, L^\vee)\ri H^1(C, \mo(-p))\simeq H^1(C, \mo)$ implying a lift $\mo(p)\ri E$.\\
		Since $E$ is stable, the inclusion $\mo(p)\hookrightarrow E$ defines a maximal subbundle. Hence, $[E]\in \Phi_L(\s^{-1}(p))$ as claimed.
\end{proof}
	\begin{lemma}\label{conjugate}
		Let $x,y\in \sigma^{-1}(C)$ be distinct points such that $\Phi_L(x) = \Phi_L(y)$. Then $h^0(C, L(-p-q))>0$ where $p = \sigma(x),q = \sigma(y)$.
		 Conversely, let $p,q\in C$ be distinct points such that $h^0(C, L(-p-q))>0$ then $\exists ! (x, y)\in \sigma^{-1}(p)\times\sigma^{-1}(q)$ such that $\Phi_L(x) = \Phi_L(y)$.
	\end{lemma}
	\begin{proof}
		For distinct points $x,y\in \sigma^{-1}(C)$ such that $\Phi_L(x)=\Phi_L(y)$, we have the following diagram $$\begin{tikzcd}
			& & \mathcal{O}(q)\arrow[d, "", hook]\arrow[d, "i"']& &	\\
			0\arrow[r] &\mathcal{O}(p)\arrow[r]&E\arrow[r, "s"]&L(-p)\arrow[r]&0.
		\end{tikzcd}
		$$ The composition $s\circ i$ is non-zero since $x\neq y$. Therefore, $\mo(q)$ is a subbundle of $L(-p)$ and we have $h^0(C, L(-p-q))>0$. \\ 
		Suppose $p,q\in C$ are distinct points such that $h^0(C,L(-p-q))>0$. Let $x := ker(H^1(C, L^\vee(2p))\ri H^1(C, \mo(p-q)))$ where  $H^1(C, L^\vee(2p))\ri H^1(C, \mo(p-q))$ is given by twisting $L^\vee(2p)$ with $L(-p-q)$.
		 Based on this definition of $x$, we have the following diagram $$\begin{tikzcd}
			& & & &\mathcal{O}(q)\arrow[d, " ", hook]\arrow[d, "i"']\arrow[ld, ,dashed]&	\\
			 & 0\arrow[r] &\mathcal{O}(p)\arrow[r]&E\arrow[r]&L(-p)\arrow[r]&0.
		\end{tikzcd}$$ 
		The inclusion $i$ lifts to a morphism $\mo(q)\ri E$. Using a similar argument as in proof of Lemma \ref{surject}, we have $[E]\in \Phi_L(\sigma^{-1}(q))$.
		 Let $y\in \sigma^{-1}(q)$ be such that $\Phi_L(y) = [E]$. 
		 Uniqueness of $x$, and $y$ follow from \cite[Lemma 3.1]{nar-ram}, and \cite[Corollary 4.2]{bertram}.
	\end{proof}
	\begin{rem}\label{lin-sys}
		Let $p,q,r\in C$ be distinct points such that $L\simeq \mo(p+q+r)$. The bundles $E_{p,q}$ defined using extension classes $x,y$ for $p,q,$ and $E_{p,r}$ defined using extension classes $x',y'$ for $p,r$ as in Lemma \ref{conjugate} satisfy $E_{p,q}\simeq E_{p,r}$. That is, $x=x'$ and $\Phi_L(x) = \Phi_L(y)=\Phi_L(y')$. \\
		Furthermore, $\Phi_L(\tilde{L}_{p,q}) = \Phi_L(x)$ where $\tilde{L}_{p,q}$ is the proper transform of the line $L_{p,q}$ associated to the blowup $\s$.
	\end{rem}
	\begin{lemma}\label{sig-iso}
		The restriction of $\s$ to $E_L\cap \sigma^{-1}(C)$ maps it isomorphically onto $C$ if and only if $|L|$ is base-point free.
		 If $L$ has a basepoint $p\in C$, then $\sigma^{-1}(p)$ maps isomorphically onto $B_L$.
	\end{lemma}
	\begin{proof}
		Let's assume that $L$ is basepoint free. It is enough to show that we have a bijection of points between $E_L\cap \sigma^{-1}(C)$ and $C$ via $\sigma$ by Theorem \ref{zmt}. \\
	We show that $\sigma^{-1}(p)\cap E_L$ consists of a unique point if and only if $p$ is not a base-point of $|L|$.
		Points in $\sigma^{-1}(p)$ are extension classes $0\rightarrow \mathcal{O}(p)\rightarrow E\rightarrow L(-p)\rightarrow 0$. For $\Phi_L(E)\in B_L$, we must have the following diagram $$\begin{tikzcd}
			& & &\mathcal{O}\arrow[d, "", hook]\arrow[d, "i"']\arrow[ld, ,dashed]&	\\
			0\arrow[r] &\mathcal{O}(p)\arrow[r]&E\arrow[r]&L(-p)\arrow[r]&0.
		\end{tikzcd}
		$$The map $i$ lifts as shown if and only if $\delta(E)\in ker(H^1(C, L^\vee(2p))\ri H^1(C, \mo(p)))$ where $H^1(C, L^\vee(2p))\ri H^1(C, \mo(p))$ is given by a non-zero section of $L(-p)$.\\
		In the case when $p$ is not a base-point of $|L|$, $L(-p)$ has a unique nonzero section up to scaling and $dim(ker(H^1(C, L^\vee(2p))\ri H^1(C, \mo(p)))) = 1$. This defines the unique extension class $\delta\in \sigma^{-1}(p)$ such that $\delta\in E_L$.\\
		On the other hand, suppose $|L|$ has a base-point $p\in C$. For any $0\neq \tau\in H^0(C, L(-p))$ with $Z(\tau) = q_\tau+\tilde{q}_\tau$ for some $q_\tau, \tilde{q}_\tau\in C$, we can define the extension class $\delta_\tau\in ker(H^1(C, L^\vee(2p))\ri H^1(C, \mo(p-q_\tau)))$ such that the associated bundle $E_\tau$ satisfies the following diagram $$\begin{tikzcd}
			&	& & &\mathcal{O}\arrow[r, "", hook]\arrow[ld, ,dashed]&\mo(q_\tau)\arrow[ld,"\tau"]	\\
			\delta_\tau :&	0\arrow[r] &\mathcal{O}(p)\arrow[r]&E_\tau\arrow[r]&L(-p)\arrow[r]&0.
		\end{tikzcd}
		$$ 
		For $\tau,\tau'\in H^0(C,L(-p))\setminus 0$ such that $\tau\not\in <\tau'>$, the defined extension classes $\delta_\tau,\delta_{\tau'}\in H^1(C,L^\vee (2p))$ are distinct.
		Thus, for every non-zero section $\tau$ of $L(-p)$, we have a unique extension class $\delta_\tau := ker(H^1(C, L^\vee(2p))\ri H^1(C, \mo(p)))$ defining a bundle $E_\tau\in B_L$ via $\Phi_L$.\\
		 Hence, we have $\sigma^{-1}(p)\hookrightarrow E_L$ if $p$ is a basepoint of $L$.
		On the other hand, let $E\in B_L$. By Lemma \ref{surject}, $q\in \s(\Phi_L^{-1}([E]))$ for some $q\in C$.
		If $q\neq p$, then $h^0(C, L(-p-q))>0$ and since $q$ is not a basepoint of $L$, $\Phi_L(E_L\cap \sigma^{-1}(q)) = [E]$. By Lemma \ref{conjugate}, we have that $[E]\in \Phi_L(\sigma^{-1}(p))$. Therefore, $\Phi_L$ induces a bijection between $\sigma^{-1}(p)$ and $B_L$. Using Theorem \ref{zmt}, we have the claimed isomorphism.
	\end{proof}
	
	\begin{lemma}\label{rc}
		The Brill-Noether locus, $B_L$, is a smooth rational curve.
	\end{lemma}
	\begin{proof}
		If $L$ has a basepoint $p\in C$, then $B_L\simeq \sigma^{-1}(p)\simeq \p^1$ by Lemma \ref{sig-iso}.\\
		If $L$ is basepoint free, we have $C\simeq E_L\cap \s^{-1}(C)$ by Lemma \ref{sig-iso}. 
		Thus, $C\simeq E_L\cap \sigma^{-1}(C)\xrightarrow{\Phi_L} B_L$ is surjective and non-constant by Lemma \ref{surject}, and Lemma \ref{conjugate}.
		 Hence, $B_L$ is irreducible and has dimension $1$. We identify the restriction of $\Phi_L$ to $E_L\cap \sigma^{-1}(C)$ with the induced map $C\ri B_L$.\\
		Let $\tilde{B}_L$ be the normalization of $B_L$. The restricted morphism, $\Phi_L|_C$, then factors via $C\ri \tilde{B}_L\ri B_L$. 
		For $p,q,r\in C$ distinct such that $L\sim \mo(p+q+r)$, we have $\Phi_L(p) = \Phi_L(q)=\Phi_L(r)$ by Remark \ref{lin-sys}.\\
		 Since $\tilde{B}_L\ri B_L$ is birational, $C\ri \tilde{B}_L$ is a degree $3$ map between smooth curves and $\tilde{B}_L$ has genus at most $1$ by the Riemann-Hurwitz Theorem.\\
		 The map $\Phi_L|_C$ has identical fibers to $C\xrightarrow{|L|}\p^1$ given by the linear system of $L$ by Lemma \ref{conjugate}, which has at least $4$ branch points. That is, there are at least $4$ points in $B_L$ whose fibers have less than $3$ points.
		 Therefore, $g(\tilde{B}_L)=0$ and $\tilde{B}_L\simeq \p^1$. 
		A general fiber of $\tilde{\Phi}_L$ coincides with a general fiber of $\Phi_L|_C$ by Lemma \ref{conjugate}. 
		Thus, $\tilde{\Phi}_L$ is given by the linear system $|L|$.\\
This shows that distinct points $a, b\in \tilde{B}_L$ must map to distinct points in $B_L$ by Lemma \ref{conjugate}. Therefore, $\tilde{B}_L\simeq B_L$ and $B_L$ is a smooth rational curve.
	\end{proof}
	\begin{rem}\label{cone-ruling}
		If $L$ has a basepoint $p$ then $Q_L\setminus p \simeq E_L\setminus \sigma^{-1}(p)$ via $\s$ by Lemma \ref{sig-iso}. 
		For the induced morphism $\phi_L: C\setminus p \hr Q_L\setminus p \xrightarrow{\Phi_L} B_L$, we have $\phi_L(q) = \phi_L(r)$ if and only if $K_C\sim q+r$ by Lemma \ref{conjugate}. Hence, $\phi_L = \alpha|_{C\setminus p}$ and the extension of $\phi_L$ to $C$ is therefore given by $\alpha$ where $\alpha$ is the $g^1_2$-morphism of $C$.
	\end{rem}
	\begin{theorem}\label{quadric}
		If $L$ has a basepoint $p\in C$, then $Q_L$ is a quadric cone in $\p_L$ with vertex at $p$. On the other hand if $L$ is basepoint free, $E_L$ maps isomorphically onto $Q_L$ via $\sigma$ and $Q_L$ is a smooth quadric surface in $\p_L$.
	\end{theorem}
	\begin{proof}
		By \cite[Definition-Claim  4.6]{bertram}, $Q_L$ is a quadric hypersurface in $\p_L$.
		In the case when $L$ is basepoint free, $E_L$ maps isomorphically onto $Q_L$ via $\sigma$ by Lemma \ref{sig-iso}. 
		Since $E_L\ri B_L$ is a $\p^1$-fibration by Theorem \ref{main-bertram}, $E_L$ is a Hirzebruch surface by Lemma \ref{rc}. 
	Thus, it must be a smooth quadric surface.\\
		On the other hand, let $p\in C$ be a basepoint of $L$ and $q\in C\setminus p$. 
		We have $ p,q,\overline{q}$ are collinear in $\p_L$ for any $q\in C$  where $\overline{q}$ is the hyperelliptic conjugate of $q$ by Lemma \ref{conjugate},.
		Therefore, $Q_L$ consists of lines joining triples of points $\{(p,q,\overline{q})|q\in C\}$. Consequently, $Q_L = C_p(C)$ where $C_p(C)$ is the cone obtained by lines joining points in $C$ with $p$.
	\end{proof}
	\begin{rem}\label{rulings}
	For any divisor $D$ on $C$, put $\p^D := \p H^0(C,D)^\vee$. For $L$ basepoint free, the Segre embedding $s_L: \p^K\times \p^L\hr \p_L$ is given by the isomorphism $H^0(C,K_C)\otimes H^0(C,L)\xrightarrow{\sim} H^0(C,K_C\otimes L)$.
	The maps $C\xrightarrow{|K_C|}\p^K, C\xrightarrow{|L|}\p^L$ induce the morphism $C\hr \p^K\times\p^L\xrightarrow{s}\p_L$ which coincides with the embedding of $C$ into $\p_L$. Hence, $Q_L = Im(s)$ is the unique quadric containing $C$ in $\p_L$.
		Furthermore, the induced map $C\ri  B_L$ is given by the linear system $|L|$ and therefore, $\Phi_L|_{E_L}:E_L\simeq Q_L\ri B_L$ can be identified with the projection onto $\p^L$.\\
		If $L$ has a basepoint $p$, the map $E_L\ri B_L$ is still a $\p^1$-fibration over a smooth rational curve. Therefore, $E_L$ is a Hirzebruch surface. 
		The map $E_L\ri Q_L$ is a birational map which contracts $\sigma^{-1}(p)\simeq \p^1$. Hence, $E_L = bl_{C}Q_L\hr \tilde{\p}_L$ is the strict transform of $Q_L$. We further have the morphism $E_L\ri bl_{p}Q_L$ by the universal property of blowups, which is a bijection on closed points. 
	Therefore, $E_L\simeq bl_p Q_L$
	\end{rem}
	\begin{rem}\label{blowup-blowdown}
	Irrespective of the basepoint freeness of $L$, we have the induced morphism $bl_C(\p_L)\ri bl_{B_L}(U_C(2,L))$ which induces a bijection on closed points. Therefore, we have $$bl_C(\p_L)\simeq bl_{B_L}(U_C(2,L)).$$
	\end{rem}
	\subsection{Generalizing to $U(2,3,2)$}\label{gen-U}
With slight abuse of notation for this subsection, let $
		\lb\ri\C\xrightarrow{\pi}S$ be a family of smooth genus $2$ curves with a line bundle $\lb$ of relative degree $3$.\\
The analogue of the extension space is $\p_\lb:= \underline{Proj}_S \Sym^{\bullet} (R^1\pi_* \lb^\vee)^\vee$ with $\pi_\lb: \p_\lb\ri S$ the induced morphism.\\
 By Serre duality for the family $\pi$, we have $(R^1\pi_*\lb^\vee)^\vee \simeq \pi_*(K_\pi\otimes \lb)$.\\
  The surjection $\pi^*\pi_* (K_\pi\otimes \lb)\twoheadrightarrow (K_\pi\otimes \lb) $ defines an embedding $\C\hr \p_\lb$. We shall denote the image of this embedding by $B^0(\lb)$.\\
	Let $\s_\pi:\tilde{\p}_\lb\ri\p_\lb$ be the blowup of $\p_\lb$ along $B^0(\lb)$ and $\B: \C\times_S B^0(\lb)\ri S$ be the composed morphism. We define a Poincar\'e bundle on $\C\times_S \tilde{\p}_\lb$ that restricts to the Poincar\'e bundle used in Theorem \ref{main-bertram} over each fiber of $\B$.\\
	Let $\mathscr{E}_\lb\ri \C\times_S \p_\lb$ be the Poincar\'e bundle given by $\lb, \pi$ using \cite[Remark 1]{shubham}.
	\begin{lemma}
		We have the following unique lift $f_\lb$ over $\C\times_S B^0(\lb)$  $$\begin{tikzcd}
			0\arrow[r]& \pi_{B^0(\lb)}^*(\mo (1))\arrow[r]\arrow[d, " \otimes\Delta"', hook]& \mathscr{E}_\lb|_{\C\times_S B^0(\lb)} \arrow[r]\arrow[ld, "f_\lb", dashed]& \pi_\C^*\lb \arrow[r]& 0\\
			&\mathscr{L}_\lb^0& & &
		\end{tikzcd}
		$$ where we identify $\mo_{\p_\lb}(1)$ with its restriction to $B^0(\lb)$, $\mathscr{L}_\lb^0 := \pi_{B^0(\lb)}^*(\mo (1))\otimes \mo (\Delta) $ and $\Delta \hr \C\times_S B^0(\lb)$ is the diagonal.
	\end{lemma}
	\begin{proof}
		We show the existence of this lift $f_\lb$ and then prove uniqueness. Let 
		$$\mathcal{F}:= \pi_\C^*\lb^\vee\otimes \pi_{B^0(\lb)}^*(\mo(1)), \mathcal{G} :=  \pi_\C^*\lb^\vee\otimes \pi_{B^0(\lb)}^*(\mo(1))\otimes \mo(\Delta) $$
		 and let $\partial\in H^1(\C\times_S B^0(\lb), \mathcal{F}) $ be the extension class of $ \mathscr{E}_\lb|_{\C\times_S B^0(\lb)}$.\\ 
		Let $U$ be an arbitrary affine open in $S$, and let $CB_U := \B^{-1}(U)$. We show that $\partial|_{CB_U}\in ker(H^1(CB_U, \mathcal{F}|_{CB_U})\ri H^1(CB_U, \mathcal{G}|_{CB_U}))$.\\
		Since $U$ is affine, $H^1(CB_U, \mathcal{F}|_{CB_U})\simeq \Gamma(U,R^1\B_{U*}(\mathcal{F}|_{CB_U})), H^1(CB_U, \mathcal{G}|_{CB_U})\simeq \Gamma(U,R^1\B_{U*}(\mathcal{G}|_{CB_U}))$. By Cohomology and Base Change, we have 
		$$H^1(CB_U, \mathcal{F}|_{CB_U})\otimes k(s)\simeq H^1(C_s\times B^0(L_s), \mathcal{F}_s), H^1(CB_U, \mathcal{G}|_{CB_U})\otimes k(s)\simeq H^1(C_s\times B^0(L_s), \mathcal{G}_s)\forall s\in S$$
		 As explained in \cite[Definition-Claim 3.4]{bertram}, 
		\begin{gather*}
\partial_s \in ker(H^1(C_s\times B^0(L_s), \mathcal{F}_s)\ri H^1(C_s\times B^0(L_s), \mathcal{G}_s)) \forall s \in U  \implies\\
\partial\otimes k(s)\in ker(H^1(CB_U, \mathcal{F}|_{CB_U})\otimes k(s)
\ri H^1(CB_U, \mathcal{G}|_{CB_U})\otimes k(s))\forall s\in U\implies\\
 \partial\in ker(H^1(CB_U, \mathcal{F}|_{CB_U})\ri H^1(CB_U, \mathcal{G}|_{CB_U})).
\end{gather*}
		Therefore, we have local lifts $f_{\lb, U}:\mathscr{E}_\lb|_{CB_U}\ri \mathscr{L}_\lb^0|_{CB_U}$ for the morphism $\pi_{B^0(\lb|_{\C_U})}^*(\mo (1)) \xrightarrow{\otimes\Delta_U} \mathscr{L}_\lb^0|_{CB_U}$ over every open affine $U\subset S$. Next we show that these lifts $f_{\lb, U}$ are unique.\\
		Suppose there are lifts $f_{\lb, U}^1, f_{\lb, U}^2: \mathscr{E}_\lb|_{CB_U}\ri \mathscr{L}_\lb^0|_{CB_U}$. Then $f_{\lb, U}^1- f_{\lb, U}^2$ lifts to a morphism $\pi_{\C_U}^*\lb|_{\C_U}\ri \mathscr{L}_\lb^0|_{CB_U}$.\\ 
		We show that $\Hom (\pi_{\C_U}^*\lb|_{\C_U}, \mathscr{L}_\lb^0|_{CB_U} ) = 0$ to conclude that these local lifts $f_{\lb, U}$ are unique over every open affine $U$. It is enough to show that $H^0(CB_U, \pi_{\C_U}^*\lb^\vee|_{\C_U}\otimes \mathscr{L}_\lb^0|_{CB_U} ) = 0$ .\\
		Let $t\in H^0(CB_U, \pi_{\C_U}^*\lb^\vee|_{\C_U}\otimes \mathscr{L}_\lb^0|_{CB_U})$ be a section. For a point $b\in B^0(\lb|_{\C_U})$ with $\pi_\lb(b) = s\in U$, $t|_{\C_s\times b}\in H^0(\C_s, \lb|_{\C_s}^\vee) = 0$. Therefore, 
		$t|_{(CB_U)_b} = 0\forall b\in B^0(\lb|_{\C_U})$. In other words, $t = 0$ and
		consequently, we have $H^0(CB_U, \pi_{\C_U}^*\lb^\vee|_{\C_U}\otimes \mathscr{L}_\lb^0|_{CB_U} ) = 0$ proving the uniqueness of the lift $f_{\lb,U}$ for all open affines $U\subset S$.\\
		By the uniqueness of the lifts, $f_{\lb,U}|_{CB_V} = f_{\lb,V}$ for arbitrary affine opens $V\subset U\subset S$. Hence, the lifts $\{f_{\lb,U}\}_U$ patch together to define a lift $\mathscr{E}_\lb|_{\C\times_S B^0(\lb)}\xrightarrow{f_\lb} \mathscr{L}_\lb^0$ as claimed. 
		The uniqueness of $f_\lb$ follows from its construction.
	\end{proof}
	\begin{theorem}\label{poincare-bertram}
There is a bundle $\mathscr{E}^1$ on $\C\times_S\tilde{\p}_\lb$ which restricts to the Poincar\'e bundle on $\C_s\times\tilde{\p}_{\lb_s}$, as defined in Theorem \ref{main-bertram}, for all $s\in S$. That is, the bundle $\mathscr{E}^1$ defines a map $\tilde{\p}_\lb\xrightarrow{\Phi_\lb} {U}(2,3,2)$. 
	\end{theorem}
	\begin{proof}
	Let $\mathscr{E}^0:= (1, \sigma_\pi)^*\mathscr{E}_\lb$ be the pullback of $\mathscr{E}_\lb$ over $1\times\s_\pi: \C\times_S \tilde{\p}_\lb\ri \C\times_S \p_\lb$ and let $E_\s$ be the exceptional divisor for the blowup $\sigma_\pi$. The kernel of the composition 
	$$\mathscr{E}^1:= ker(\mathscr{E}^0\ri \mathscr{E}^0|_{\C\times_S E_\s}\xrightarrowdbl{f_\lb}(1,\sigma_\pi)^*\mathscr{L}_\lb^0)$$ 
	restricts to the Poincar\'e bundle used to construct the morphism $\Phi_{\lb_s}$ for $\lb_s\ri \C_s$ in \cite[Theorem 1]{bertram}, for all $s\in S$. 
	\end{proof}\noindent
	Let $B^2_{2,3}$ be the universal Brill-Noether locus of bundles $\{[E\ri C]| rk E = 2, \deg E = 3, g(C)=2, h^0(C,E)= 2\}$, put $E_\lb := \Phi_\lb^{-1}(B^2_{2,3})$, and put $Q_\lb := \sigma_\pi (E_\lb)$. \\
	We have $\C\hr Q_\lb$ by Lemma \ref{sig-iso} and $Q_\lb\hr \p_\lb$ is a family of degree $2$ hypersurfaces parametrized by $S$ by Theorem \ref{quadric}.\\
	We replace the family $\begin{tikzcd}
		\lb\arrow[d]&\\ \C\arrow[r,"\pi"]&S
	\end{tikzcd}$ with the family $\begin{tikzcd}
		\xi\arrow[d]&\\ \C_J\arrow[r,"\pi_J"]&J_*(3)
	\end{tikzcd}$ 
	where $\C_J := \C_*\times_{M_*}J_*(3)$, and $\xi$ is the Poincar\'e bundle for the relative Jacobian $J_*(3)$ rigidified along the section $\s_1$.\\
	Following the notation in Theorem \ref{larson-main}, let $C_J\ri J_2^3, C_\J\ri \J_{3,2}$ be the universal curves over $J_2^3, \J_{3,2}$ respectively and $\p_\xi,\tilde{\p}_\xi$ be the analogues of $\p_\lb,\tilde{\p}_\lb$ for the family $\pi_J$.
	\begin{defi*}
		The Poincar\'e bundle $\lb_\J$ on $C_\J$ defines a stack $\p_\J$ over $\J_{3,2}$, which is defined using the functoriality of $\p_\lb$. We denote the $\mathbb{G}_m$-rigidification of this stack by $\p_J:= \p_\J\hollowslash\mathbb{G}_m$. There is a natural map $\p_J\ri J_2^3$.
	\end{defi*}
	\noindent More conceretely, for any scheme $S$, we have $\p_J(S) = \{(\C_S,\lb_S, \mathcal{F})\}$ such that $\lb_S$ is a line bundle on a family of smooth genus $2$ curves $\C_S\xrightarrow{\pi_S}S$; and $\mathcal{F}\subset \pi_{S*}(K_{\pi_S}\otimes\lb_S)$ is a subsheaf such that the quotient $\pi_{S*}(K_{\pi_S}\otimes\lb_S)/\mathcal{F}$  is a line bundle on $S$.\\
	Morphisms between objects $(\C_T,\lb_T, \mathcal{F}_T)\ri(\C_S,\lb_S, \mathcal{F}_S)$ over a map $f:T\ri S$ are given by a map $g:\C_T\ri \C_S$ such that the following square is Cartesian
	$$\begin{tikzcd}
		\C_T\arrow[r,"g"]\arrow[d,"\pi_T"]&\C_S\arrow[d,"\pi_S"]\arrow[ld,"\square",phantom]\\
		T\arrow[r,"f"]&S,
	\end{tikzcd}$$
and there exists $L_T\in Pic(T)$ such that $\lb_S \simeq g^*\lb_T\otimes\pi_T^*L_T$ and $\mathcal{F}_T = f^*\mathcal{F}_S\otimes L_T$.
\begin{rem}\label{d-m-pj}
The rigidified stack $\p_J$ is a Deligne-Mumford stack of finite type over $J_2^3$. 
\end{rem}
\noindent By the universal property of $J_2^3$, we have maps $q_\p:\p_\xi\ri \p_J, q_J: J_*(3)\ri J_2^3, \text{ and } q_C:\C_J\ri C_J$ such that every quadrilateral in the following diagram is Cartesian
	$$
	\begin{tikzcd}
		&\p_\xi\arrow[r,"q_\p"]\arrow[ldd]&\p_J\arrow[ldd]\\
		\C_J\arrow{r}[near start]{q_C}\arrow[rd,"\square",phantom]\arrow[d,"\pi_J"']\arrow[ur, hook]&C_J\arrow[d]\arrow[ur, hook]&\\
		J_*(3)\arrow[r,"q_J"]&J_2^3.&
	\end{tikzcd}
	$$ 
	Let $\tilde{U}(2,3,2)$ be the moduli space of stable bundles of rank $2$ and relative degree $3$ for the family $\C_*\xrightarrow{\pi} M_*$ and $B_J\hr \mathcal{U}(2,3,2)$ be the universal Brill-Noether locus of bundles in $\mathcal{U}(2,3,2)$ parametrizing bundles with $2$ independent sections.\\
It follows that $B_{2,3}^2$ is the coarse moduli space of $B_J$.
	The morphism $\Phi_\xi:\tilde{\p}_\xi\ri U(2,3,2)$, defined by the Poincar\'e bundle on $\C_J\times_{J_*(3)}\tilde{\p}_\xi$ lifts to $\tilde{U}(2,3,2)$ via the natural map $q:\tilde{U}(2,3,2)\ri U(2,3,2)$. Additionally, put $B_*:= B^2_{2,3}\times_{U(2,3,2)}\tilde{U}(2,3,2)$, and let $q_B: B_*\ri B_J$ be the induced map:
	\begin{equation}\label{universal}
		\begin{tikzcd}
			&& & & B_*\arrow[d, "i_B", hook]\arrow[r]&B_{2,3}^2\arrow[d,hook]&\\
			&&E_\xi\arrow[d]\arrow[r,"i_\Phi"' , hook]\arrow[rru,"\Phi_E" , bend left = 15] & \tilde{\p}_\xi\arrow[dd, "\sigma_\pi"]\arrow[r, "\Phi_\xi"]& \tilde{U}(2,3,2)\arrow[r, "q"]\arrow[ddd,"\det_\xi"]& U(2,3,2) \\
			&\xi\arrow[d]	&Q_\xi\arrow[dr, "i_\phi", hook]& &  \\
			&\C_J\arrow[ru, "", hook]\arrow[rr, "j", hook]\arrow[d, "\pi_J"]& &\p_\xi\arrow[lld, "\pi_\xi"] & &\\
			\C_*\arrow[r,"\T_*", hook]&J_*(3)\arrow[u,"\pi^*\sigma_i",bend left= 50]\arrow[rrr]& & &J_{2}^3 &\T_3.\arrow[l,""',hook']
	\end{tikzcd}\end{equation}
We calculate the image of $q^*$, denoted by $S^*(\tilde{U}(2,3,2))$ from hereon to obtain $A^*(U(2,3,2))$. We shall stratify $\tilde{U}(2,3,2)$ and use excision arguments to find generators of desired Chow groups. The excision of $\tilde{U}(2,3,2)$ is a two step process:
\begin{flalign*}\label{excision-U}\tag{\textdagger}
		\bullet\quad & \tilde{U}(2,3,2) = (\tilde{U}(2,3,2)\setminus B_*)\coprod B_*&\\
	\bullet\quad & B_*:= B_V\coprod B_\T \text{ where }V:= J_*(3)\setminus \T_* \text{ and } B_V:= V\times_{J_*(3)}B_*, B_\T:= \T_*\times_{J_*(3)}B_* 
\end{flalign*}
\begin{rem}\label{blowdown}
The induced morphism $\tilde{\p}_\xi\ri bl_{B_*}(\tilde{U}(2,3,2))$ defined by the universal property of blowups is a bijection on closed points since it restricts to an isomorphism for every fiber over $J_*(3)$ by Remark \ref{blowup-blowdown}. By Theorem \ref{zmt}, we have the isomorphism $\tilde{\p}_\xi\simeq bl_{B_*}(\tilde{U}(2,3,2))$.
\end{rem}
\begin{lemma}\label{pj-pxi}
	The induced pullback maps $q_\p^*:A^*(\p_J)\ri A^*(\p_\xi), q_B^*: A^*(B_J)\ri A^*(B_*)$ are isomorphisms.
\end{lemma}
\begin{proof}
Let $\p_{3,2}$ be the coarse moduli space of $\p_J$. The pullback maps $A^*(B_{2,3}^2)\ri A^*(B_J), A^*(\p_{3,2})\ri A^*(\p_J)$ are isomorphisms by Remark \ref{d-m-pj}. 
The morphisms $B_*\ri B_{2,3}^2, \p_\xi\ri \p_{3,2}$ are finite. Consequently, $q_\p^*,q_B^*$ are injective by Remark \ref{rem1}. 
Applying excision on $B_*$, it is enough to show that $q_\p^*, q_B|_{J_2^3\setminus\T_3}^*,q_B|_{\T_3}^*$ are surjective.\\
By Remark \ref{rulings}, $\p_\xi,B_*|_V,B_*|_{\T_*}$ are projective bundles over $J_*(3), V, \T_*$ respectively. 
It is enough to show that the respective relative hyperplane divisors are in the images of the corresponding pullbacks in codimension $1$ by the projective bundle formula.
We prove the claim for $q_\p^*$. Other pullbacks shall follow similarly.\\
Applying \cite[Proposition 6.6]{melo} and \cite[\S 6.4]{larson-24}, let $\J ac_{3,2}\ri J_2^3$ be the $\mathbb{G}_m$-gerbe given by twice the cohomology class of the $\mathbb{G}_m$-gerbe $\J_{3,2}\ri J_2^3$. Consequently, the induced map $\J_{3,2}\ri \J ac_{3,2}$ is a $\mu_2$-gerbe.
We have a trivialization of $\J ac_{3,2}$ over $J_2^3$ due to the $2$-torsion of the Brauer class of $\J_{3,2}$ over $J_2^3$. That is, we have a section $\s_\J:J_2^3\ri \J ac_{3,2}$.\\
Let $\p_{\J ac}$ be the pullback of $\p_J$ over $\J ac_{3,2}$. The section $\s_\J$ pulls back to a section $\s_{\p J}:\p_J\ri \p_{\J ac}$ of $t_\J:\p_{\J ac}\ri \p_J$ giving us the following diagram where all squares are Cartesian
$$
\begin{tikzcd}
	\p_\xi\arrow[d]\arrow[r]&\p_\J\arrow[ld,"\square",phantom]\arrow[d]\arrow[r,]&\p_{\J ac}\arrow[ld,"\square",phantom]\arrow[r,"t_\J"']\arrow[d]&\p_J\arrow[d]\arrow[l,bend right = 30,"\s_{\p J}"']\arrow[ld,"\square",phantom]\\
	J_*(3)\arrow[r]&\J_{3,2}\arrow[r]&\J ac_{3,2}\arrow[r]&J_2^3\arrow[l,bend left=30, "\s_\J"]\\
	&&J_2^3.\times B\mathbb{G}_m\arrow[u,sloped,"\simeq"] &
\end{tikzcd}
$$
The section $\s_{\p J}$ defines a trivialization of the $\mathbb{G}_m$-gerbe $\p_{\J ac}/\p_J$. 
The induced map $\p_\J\ri \p_{\J ac}$ is a $\mu_2$-gerbe. 
Therefore, the pullback map $q_\p^*$ factors as:
$$A^*(\p_J)\xleftrightarrow[t_\J^*]{\s_{\p J}^*}A^*(\p_{\J ac})\simeq A^*(\p_\J)\ri A^*(\p_\xi)$$
We show that the pullback map $q_\p^*$ is an isomorphism in codimension $1$ by showing that $\dim A^1(\p_\xi) = \dim A^1(\p_J)$.
 By the projective bundle formula, we have 
$$\dim A^1(\p_{\J ac}) = \dim A^1(\p_\J) = \dim A^1(\J_{3,2})+1 = 3, \dim A^1(\p_\xi) = \dim A^1(J_*(3))+1 = 2$$
Using $\p_{\J ac}\simeq \p_J\times B\mathbb{G}_m$ and $A^*(B\mathbb{G}_m) = \Q[c_1]$, we get 
$$ A^1(\p_{\J ac})\simeq A^1(\p_J)\oplus A^1(B\mathbb{G}_m)\implies \dim A^1(\p_{J}) = \dim A^1(\p_{\J ac})-\dim A^1(B\mathbb{G}_m) = 3-1=2$$
Thus, $q_\p^*$ is an isomorphism in codimension $1$ as claimed.
\end{proof}
\subsection{\textbf{Chow rings of }${\C_J,\p_\xi \text{ and }\tilde{\p}_\xi}$}\label{chow-blowup-ss} 
We start with $A^*(C_J)$, and then use Grothendieck-Riemann-Roch Theorem, Theorem \ref{blowup-chow} to calculate $A^*(\p_\xi)$ and $A^*(\tilde{\p}_\xi)$. Let $\mathcal{C}_2$ be the universal curve over $\mathcal{M}_2$.
\begin{lemma}\label{cxc}
 Let $\tilde{\psi}_i:=\pi_i^*\mathcal{K} = 2\pi_i^*\sigma_1$ for $i=1,2$ where $\pi_i:\C_*\times_{M_*}\C_*\ri \C_*$ is the projection onto the $i$-th factor; and let $\Delta$ be the diagonal. We have 
 $$A^*(\C_*\times_{M_*}\C_*) = \Q[\tilde{\psi}_1,\tilde{\psi}_2, \Delta]/(\tilde{\psi}_1^2,\tilde{\psi}_2^2, 2\Delta\tilde{\psi}_1-\tilde{\psi}_1\tilde{\psi}_2, \Delta(\tilde{\psi}_1-\tilde{\psi}_2), 2\Delta^2+\tilde{\psi}_1\tilde{\psi}_2).$$
\end{lemma}
\begin{proof}
	Let $Z_W:= \W_*\times_{M_*}\C_*\hr \C_*\times_{M_*}\C_*$ be the locus of Weierstrass points in $\C_*\times_{M_*}\C_*$ with respect to the projection onto the second factor and let $\tilde{\Delta}$ be the image of the diagonal under the hyperelliptic involution with respect to the projection onto the first factor.\\
	We use the stratification $\C_*\times_{M_*}\C_* = (\C_*\times_{M_*}\C_*\setminus (\Delta\cup\tilde{\Delta}\cup Z_W))\cup\Delta\cup\tilde{\Delta}\cup Z_W$  along with 
	$$A^*(\C_*\times_{M_*}\C_*\setminus (\Delta\cup\tilde{\Delta}\cup \W_*\times_{M_*}\C_*))=\Q$$ 
	as shown in Lemma \ref{chow-2}.
	We have the following pushforwards of total Chow groups.
	\begin{itemize}
		\item $i_{W*}(A_*(Z_W)) = \Q.\tilde{\psi}_1\oplus \Q.\tilde{\psi}_1\tilde{\psi}_2$ for the embedding $i_W:Z_W\hr \C_*\times_{M_*}\C_*$.
		\item $i_{\Delta *}(A_*(\C_*)) = \Q.[\Delta]\oplus \Q.\tilde{\psi}_1\tilde{\psi}_2 $ for the embedding $i_\Delta:\Delta\hr \C_*\times_{M_*}\C_*$.
		\item $i_{\tilde{\Delta} *}(A_*(\C_*)) = \Q.[\tilde{\Delta}]\oplus \Q.\tilde{\psi}_1\tilde{\psi}_2 $ for the embedding $i_{\tilde{\Delta}}:\tilde{\Delta}\hr \C_*\times_{M_*}\C_*$
	\end{itemize}
	Thus, $A^2(\C_*\times_{M_*}\C_*)=<\tilde{\psi}_1\tilde{\psi}_2>$. Since restriction to every fiber over $\pi_2$ is trivial, we have $\Delta+\tilde{\Delta} - \tilde{\psi}_1\in \pi_2^*(A^1(\C_*))$. Thus, 
	$$\exists a\in \Q :\Delta+\tilde{\Delta} = \tilde{\psi}_1+a\tilde{\psi}_2.$$
	The $\mu_2$-action of transposing coordinates in $\C_*\times_{M_*}\C_*$ leaves $\Delta,\tilde{\Delta}$ invariant and transposes $(\tilde{\psi}_1,\tilde{\psi}_2)$. It follows that $a=1$. \\
	Additionally, $\tilde{\Delta}.\tilde{\psi}_i = \Delta.\tilde{\psi}_i=\dfrac{1}{2}\tilde{\psi}_1\tilde{\psi}_2$ since $\tilde{\Delta}.\pi_i^*\sigma_1 = \Delta.\pi_i^*\sigma_1 = [\sigma_1\times_{M_*}\sigma_1]\hr \C_*\times_{M_*}\C_*$ for $i=1,2$ and $$\tilde{\Delta}.\Delta = i_{\Delta *}(\W_*)=6[\sigma_1\times_{M_*}\sigma_1] = \dfrac{6}{4
	}\tilde{\psi}_1\tilde{\psi}_2\implies \Delta^2=-\dfrac{1}{2}\tilde{\psi}_1\tilde{\psi}_2.$$
	  Restricting to fibers over $M_*$, we get 
	  $$\Delta,\tilde{\psi}_1\tilde{\psi}_2\neq 0 \in A^*(\C_*\times_{M_*}\C_*).$$ 
	 Let $U_\Delta := \C_*\times_{M_*}\C_*\setminus \Delta$. We have the finite map $\beta:U_\Delta\ri M_{2,2}$ and $\psi_1,\psi_2\in A^1(M_{2,2})$ so that $\beta^*\psi_i=\tilde{\psi}_i|_{U_\Delta}$ for $i=1,2$.
	The excision sequence for $\Delta\hr \C_*\times_{M_*}\C_*$ gives 
	$$\Q.\Delta\ri A^*(\C_*\times_{M_*}\C_*)\ri A^*(U_\Delta)\ri 0,$$ 
	which implies that $\Delta,\tilde{\psi}_1,\tilde{\psi}_2$ are linearly independent in $A^1(\C_*\times_{M_*}\C_*)$ since $\beta^*$ is injective.\\
	 Therefore, $\{[\C_*\times_{M_*}\C_*], \tilde{\psi}_1,\tilde{\psi}_2, \Delta, \tilde{\psi}_1\tilde{\psi}_2\}$ form a linear basis of $A^*(\C_*\times_{M_*}\C_*)$. 
\end{proof}
\begin{cor}\label{chow-d}
	Let $\psi_{i\Delta}, \psi_{i\tilde{\Delta}}$ be restrictions of $\tilde{\psi}_i$ to $\C_*\times_{M_*}\C_*\setminus\Delta,\C_*\times_{M_*}\C_*\setminus\tilde{\Delta}$ respectively for $i=1,2$. The excision sequence on $\Delta\hr \C_*\times_{M_*}\C_*$ yields 
	$$A^*(\C_*\times_{M_*}\C_*\setminus\Delta)=\Q[\psi_{1\Delta},\psi_{2\Delta}]/(\psi_{1\Delta},\psi_{2\Delta})^2, A^*(\C_*\times_{M_*}\C_*\setminus\tilde{\Delta})=\Q[\psi_{1\tilde{\Delta}},\psi_{2\tilde{\Delta}}]/(\psi_{1\tilde{\Delta}},\psi_{2\tilde{\Delta}})^2.$$ 
\end{cor}
\begin{defi*}
We shall refer to $\tilde{\Delta}\hr \C_*\times_{M_*}\C_*$, and its image in $\mathcal{C}_2\times_{\mathcal{M}_2}\mathcal{C}_2$, as the \textit{conjugate diagonal} from hereon.
\end{defi*} 
\begin{lemma}\label{chow-3}
For any $i,j\in \{1,2,3\}$, let $\Delta_{ij}\hr \times_{\mathcal{M}_2}^3\mathcal{C}_2$ be the $i,j$-diagonal and $\tilde{\Delta}_{ij}\hr \times_{M_*}^3\mathcal{C}_2$ be the pullback of the $i,j$-conjugate diagonal under the projection $\pi_{ij}:\times^3_{M_*}\mathcal{C}_2\ri \mathcal{C}_2\times_{\mathcal{M}_2}\mathcal{C}_2$ onto the $i,j$ factors. We have $$A^*(\times^3_{\mathcal{M}_2}\mathcal{C}_2\setminus (\cup_{1\leq i,j\leq 3} \tilde{\Delta}_{ij}\cup_{1\leq i<j\leq 3}\Delta_{ij})) = \Q.$$
\end{lemma}
\begin{proof}
Using \cite[\S 2.4]{casnati}, a genus $2$ curve with $3$ distinct marked points $(C,p_1,p_2,p_3)$, with none of the $p_i$'s Weierstrass or conjugates under the hyperelliptic involution can be embedded as a degree $4$ plane curve with a single nodal/cuspidal singularity.\\
We rigidify the embedding further by setting the images of the node $N\mapsto [1:0:0] , p_1\mapsto [0:1:0], p_2\mapsto [0:0:1], p_3\mapsto [0:1:1]$. 
The equation of the image is then an element of the space $V := \{\la_1x_1^2x_2^2+\la_2x_1^2x_2x_3+\la_3x_1^2x_3^2+x_1x_2^3+\la_4x_1x_2^2x_3+\la_5x_1x_2x_3^2+\la_6x_1x_3^3+ x_2(x_2-x_3)x_3^2\}$.\\
 With the blowup construction as described in Lemma \ref{chow-2}, we get a map $U'\xrightarrow{\tau}\times_{\mathcal{M}_2}^3\mathcal{C}_2$. Shrinking $U'$ if necessary, we can ensure that $Im(\tau) = \times^3_{\mathcal{M}_2}\mathcal{C}_2\setminus (\cup_{1\leq i,j\leq 3} \tilde{\Delta}_{ij}\cup_{1\leq i<j\leq 3}\Delta_{ij})$. Therefore, $\tau$ is a finite, surjective map from $U'\subset\mathbb{A}^6$ onto $Im(\tau)$. The lemma now follows from Remark \ref{rem1}.
\end{proof}
\begin{defi*}
		Let $\mathcal{K}_1\hr J_*(3)$ be the cycle given by composing the embeddings $\sigma_1\hr\C_*\xrightarrow{\T_*}J_*(3)$. 
\end{defi*}
\begin{theorem}\label{chow-curve}
Let $S^*(\C_J)$ be the image of the pullback $q_C^*:A^*(C_J)\hr A^*(\C_J)$. We have 
$$S^*(\C_J)\simeq \Q[K_{\pi_J}, \pi_J^*\Theta_*,\xi]/(\pi_J^*\Theta_*^3, K_{\pi_J}^2,\xi^2 +K_{\pi_J}\pi_J^*\T_*,\xi K_{\pi_J}, (\xi-\dfrac{3}{2}K_{\pi_J})\pi_J^*\T_*^2).$$
\end{theorem}
\begin{proof}
	We start by showing that $\xi,K_{\pi_J},\pi_J^*\T_*$ are contained in $S^*(\C_J)$. The divisors $K_{\pi_J},\pi_J^*\T_*$ are contained in the image of $q_C^*$ since $q_J^*$ is an isomorphism by Theorem \ref{t-s6}. \\
	Following the notation in \cite[\S 6]{melo}, the map $\C_J\ri C_J$ lifts to $\C_J\ri \mathcal{J}ac_{3,2,1}$ (relating our notation to \cite{melo}, we have $C_J = \mathcal{J}_{3,2,1}$) and $\xi^{\otimes 2}\in q_C^*(Pic(C_J))$ by \cite[Proposition 6.6]{melo}. 
	Thus, $ \xi\in S^1(\C_J)$.
	\\We have sections $J_*(3)\xrightarrow{\pi^*\sigma_1}\C_J, \C_*\xrightarrow{\sigma_J}\C_J$ given by pullbacks of the section $\sigma_1: M_*\ri \C_*$. Therefore $\pi_J^*\T_*^i K_{\pi_J}^j=\pi_J^*\T_*^i\pi_\C^*K_{\pi}^j\neq 0\forall 0\leq i\leq 2, 0\leq j\leq 1$ and are linearly independent in $A^*(\C_J)$.\\
	The rigidification of $\xi$ along $\pi^*\sigma_1$ gives 
	$$\xi.\pi^*\sigma_1 = 0\implies\xi. K_{\pi_J} = 0.$$
The Poincar\'e property of $\xi$ implies that the restrictions of $(\xi-K_{\pi_J}-\pi^*\sigma_1)$ are trivial on fibers over $\mathcal{K}_1\simeq M_*$. Since $\T_*^2 = 2[\pi^*\sigma_1]$, we have
$$(\xi-K_{\pi_J}-\pi^*\sigma_1).\mathcal{K}_1 = 0\implies (\xi-\dfrac{3}{2}K_{\pi_J})\pi_J^*\T_*^2=0 .$$ 
		We define the map $\Sigma:\C_*\times_{M_*}\C_*\setminus\tilde{\Delta}\ri \C_J\setminus\sigma_J$ as $\pi_1\times (\phi\otimes\sigma_1)$, where $\phi:\C_*\times_{M_*}\C_*\ri J_*(3)$ was defined in \S\ref{j-2} and $\sigma_J:= \pi_J^{-1}(\mathcal{K}_1)$. 
		The Poincar\'e property of $\xi$ shows that $[\Sigma(\C_*\times_{M_*}\C_*\setminus\tilde{\Delta})], \xi-\pi^*\sigma_1\in A^1(\C_J\setminus \sigma_J)$ restrict to identical divisors on fibers over $J_*(3)\setminus \mathcal{K}_1$. Thus, 
		$$\exists a\in \Q: [\Sigma(\C_*\times_{M_*}\C_*\setminus\tilde{\Delta})]- \xi+\pi^*\sigma_1=a\pi_J^*\T_*\implies[\Sigma(\C_*\times_{M_*}\C_*\setminus\tilde{\Delta})]\pi^*\sigma_1 =a\pi^*\sigma_1\pi_J^*\T_*.$$ 
		Transversality of the loci $\pi^*\s_1, \Sigma(\C_*\times_{M_*}\C_*\setminus\tilde{\Delta})$ in $\C_J\setminus\s_J$ shows that 
		$$[\Sigma(\C_*\times_{M_*}\C_*\setminus\tilde{\Delta})]\pi^*\sigma_1 =\pi^*\sigma_1\pi_J^*\T_*\implies a=1.$$ 
		 Therefore, \begin{equation}\label{chern}[\Sigma(\C_*\times_{M_*}\C_*\setminus\tilde{\Delta})] = \xi-\dfrac{1}{2}K_{\pi_J} +\pi_J^*\T_*.\end{equation} 
 We calculate $S^*(\C_J\setminus \sigma_J)$ and then use the excision sequence
 $$A^{*-2}(\sigma_J)\xrightarrow{i_{J*}} S^*(\C_J)\ri S^*(\C_J\setminus \sigma_J)\ri 0.$$
	The image of the pushforward $i_{J*}A^*(\sigma_J)$, is spanned by $\pi_J^*\T_*^2, \pi_J^*\T_*^2K_{\pi_J}$. \\
	We define the morphism $\Phi_\C: \times^3_{M_*}\C_*\ri \C_J = \C_*\times_{M_*}J_*(3)$ as $\pi_1\times ((\phi\circ\pi_{23})\otimes \s_1)$, where $\pi_{23}:\times^3_{M_*}\C_*\ri \C_*\times_{M_*}\C_*$ is the projection map onto the second and third factors. \\
	The restriction $\phi_\C:\times^3_{M_*}\C_*\setminus \tilde{\Delta}_{23}\ri \C_J\setminus\sigma_J$ is then a finite, flat map between smooth varieties.	Hence, $S^*(\C_J\setminus \sigma_J) $ is linearly generated by $[\C_J\setminus \sigma_J], \phi_{\C *}(A_*(\Delta_{12}\setminus \tilde{\Delta}_{23})), \phi_{\C *}(A_*(\Delta_{23}\setminus \tilde{\Delta}_{23})), \phi_{\C *}(A_*(\tilde{\Delta}_{12}\setminus \tilde{\Delta}_{23})), \phi_{\C *}(A_*(\tilde{\Delta}_{22}\setminus \tilde{\Delta}_{23})),\text{ and }\phi_{\C *}(A_*(\tilde{\Delta}_{11}\setminus \tilde{\Delta}_{23}))$ by Lemma \ref{chow-3} and Lemma \ref{cxc}.\\
	Using (\ref{chern}) and Corollary \ref{chow-d}, we have the following.
	\begin{itemize}
	\item \underline{$\phi_{\C *}(A^*(\Delta_{12}\setminus \tilde{\Delta}_{23}))$}: Is the image of the map $\phi_\C\circ\Delta_{12}:\C_*\times_{M_*}\C_*\setminus \tilde{\Delta}\ri \C_J\setminus\sigma_J$.\\
	 Therefore, $\phi_{\C *}(A^*(\Delta_{12}\setminus \tilde{\Delta}_{23}))$ is spanned by $[\phi_\C(\Delta_{12}\setminus \tilde{\Delta}_{23})]=\xi-\pi^*\sigma_1+\pi_J^*\T_*=\xi-\dfrac{1}{2}K_{\pi_J}+\pi_J^*\T_*, [\phi_{\C *}\circ\Delta_{12*}(\tilde{\psi}_1)] = K_{\pi_J}\pi_J^*\T_*$ and $[\phi_{\C *}\circ\Delta_{12*}(\tilde{\psi}_2)] = 2(\xi-K_{\pi_J}+\pi_J^*\T_*)\pi_J^*\T_*$. 
	\item \underline{$\phi_{\C *} (A^*(\Delta_{23}\setminus \tilde{\Delta}_{23}))$}: Is the image of the map $\phi_{\C}\circ\Delta_{23}:\C_*\times_{M_*}\C_*\setminus \C_*\times_{M_*}\W_*\ri \C_J\setminus\sigma_J$.\\
	 Therefore, $\phi_{\C *}(A^*(\Delta_{23}\setminus \tilde{\Delta}_{23}))$ is spanned by $[\phi_\C(\Delta_{23}\setminus\tilde{\Delta}_{23})]=4\pi_J^*\T_*, \phi_{\C *}\circ\Delta_{23*}(\Delta) = (2\xi-K_{\pi_J}+2\pi_J^*\T_*)\pi_J^*\T_*$ and $\phi_{\C *}\circ\Delta_{23*}(\tilde{\psi}_1) = 4K_{\pi_J}\pi_J^*\T_*$.    
	\item \underline{$\phi_{\C *}(A^*(\tilde{\Delta}_{12}\setminus \tilde{\Delta}_{23})) $}: Is the image of the map $\phi_\C\circ\tilde{\Delta}_{12}:\C_*\times_{M_*}\C_*\setminus \tilde{\Delta}\ri \C_J\setminus\sigma_J$. \\
	Therefore, $\phi_{\C *}(A^*(\tilde{\Delta}_{12}\setminus \tilde{\Delta}_{23}))$ is spanned by $[\phi_\C(\tilde{\Delta}_{12}\setminus \tilde{\Delta}_{23})] = 2K_{\pi_J}+\pi^*\sigma_1-\xi -\pi_J^*\T_*= \dfrac{5}{2}K_{\pi_J}-\xi-\pi_J^*\T_*, [ \phi_{\C *}\circ\Delta_{23*}(\tilde{\psi}_1)] = K_{\pi_J}\pi_J^*\T_*$ and $ [ \phi_{\C *}\circ\Delta_{23*}(\tilde{\psi}_2)] = 2(2K_{\pi_J}-\xi-\pi_J^*\T_*)\pi_J^*\T_*$.    
	\item \underline{$\phi_{\C *}(A^*(\tilde{\Delta}_{22}\setminus \tilde{\Delta}_{23}) )$}: 
	  Is the image of the map $\phi_\C\circ\tilde{\Delta}_{22}:\coprod_{1\leq i\leq 6}\C_*\times_{M_*}\C_*\setminus \C_*\times_{M_*}\sigma_i\ri \C_J\setminus\sigma_J$.
	   Let $\prescript{}{i}{\tilde{\psi}_j}:= \tilde{\psi}_j\in A^1(\C_*\times_{M_*}\C_*\setminus \C_*\times_{M_*}\sigma_i)$ be the $\tilde{\psi}_j$ class as defined in Lemma \ref{cxc} for $j=1,2$, in the $i$-th element of this disjoint union.
	   \\ Therefore, $\phi_{\C *}(A^*(\tilde{\Delta}_{22}\setminus \tilde{\Delta}_{23}))$ is spanned by $\phi_\C(\tilde{\Delta}_{22}\setminus \tilde{\Delta}_{23})=6\pi_J^*\T_*$ and $\phi_{\C *}\circ \tilde{\Delta}_{22*}(\prescript{}{i}{\tilde{\psi}_1})= K_{\pi_J}\pi_J^*\T_*, \phi_{\C *}\circ \tilde{\Delta}_{22*}(\Delta)=(\xi-K_{\pi_J}+\pi_J^*\T_*)\pi_J^*\T_* \forall 1\leq i\leq 6$. 
\item \underline{$\phi_{\C *}(A^*(\tilde{\Delta}_{11}\setminus \tilde{\Delta}_{23})) $}: Is the image of the map $\phi_\C|_{\pi^*\sigma_1}:\C_*\times_{M_*}\C_*\setminus\tilde{\Delta}\ri \C_J$. \\
Therefore, $\phi_{\C *}(A^*(\tilde{\Delta}_{11}\setminus \tilde{\Delta}_{23}))$ is spanned by $\pi^*\sigma_1 = \dfrac{1}{2}K_{\pi_J}, \phi_{\C *}(\sigma_1\times_{M_*}\sigma_1\times_{M_*}(\C_*\setminus\sigma_1) = \pi^*\sigma_1\pi_J^*\T_* = \dfrac{1}{2}K_{\pi_J}\pi_J^*\T_* $.
\end{itemize}
Therefore, $S^1(\C_J)$ is spanned by $\{\xi,K_{\pi_J},\pi_J^*\T_*\}$, $S^2(\C_J)$ is spanned by $\{\pi_J^*\T_*^2, \xi\pi_J^*\T_*, K_{\pi_J}\pi_J^*\T_*\}$, $S^3(\C_J)$ is spanned by $\{K_{\pi_J}\pi_J^*\T_*^2\}$, and $S^i(\C_J) = 0$ for all $i\geq 4$.\\
We show that $\{\xi,K_{\pi_J},\pi_J^*\T_*\}\in A^1(\C_J)$ are linearly independent. It has been shown that $K_{\pi_J},\pi_J^*\T_*$ are linearly independent. Hence, it is enough to show that $\xi\not\in <K_{\pi_J},\pi_J^*\T_*>$. \\
Assuming the contrary, we must have $\xi\in\Q.K_{\pi_J}$ since $\xi.K_{\pi_J}=K_{\pi_J}^2=0$ and $K_{\pi_J}.\pi_J^*\T_*\neq 0$. Consequently, we have 
$$(\xi-\dfrac{3}{2}K_{\pi_J}).\pi_J^*\T_*^2=0\implies \xi=\dfrac{3}{2}K_{\pi_J}.$$
 Restriction to a general fiber over $J_*(3)$ produces the desired contradiction and we have that $\{\xi,K_{\pi_J},\pi_J^*\T_*\}\in A^1(\C_J)$ are linearly independent.
\\ Linear independence of $\{\pi_J^*\T_*^2, \xi\pi_J^*\T_*, K_{\pi_J}\pi_J^*\T_*\}\in A^2(\C_J)$ follows similarly.\\
\\To describe the intersection product on $S^*(\C_J)$, it is enough to express $\xi^2$ in terms of $\pi_J^*\T_*^2, \xi\pi_J^*\T_*, K_{\pi_J}\pi_J^*\T_*$. \\
We must have $\xi^2\in <\xi\pi_J^*\T_*, K_{\pi_J}\pi_J^*\T_*>$ since $\xi.K_{\pi_J}=0$. Let $\overline{\Sigma}:\C_*\times_{M_*}\C_*\ri\C_J$ be the natural extension of $\Sigma$. \\
Let $h$ be the hyperelliptic involution on the family $\pi_J$ and $Z:= \overline{\Sigma}(\C_*\times_{M_*}\C_*)= \overline{\Sigma(\C_*\times_{M_*}\C_*\setminus\tilde{\Delta})}$.\\
 We have $[Z]=\xi-\dfrac{1}{2}K_{\pi_J}+\pi_J^*\T_*$ by (\ref{chern}). The divisors $h^*(\xi)+\xi, 3K_{\pi_J}$ have identical restrictions on every fiber of $\pi_J$. Thus,
  $$h^*(\xi)+\xi - 3K_{\pi_J}= b\pi_J^*\T_*\text{ for some }b\in\Q.$$ 
  Using $h^*(\xi).K_{\pi_J}= h^*(\xi.h^*(K_{\pi_J})) = h^*(\xi.K_{\pi_J})= 0$, we have \begin{align*}
 	bK_{\pi_J}\pi_J^*\T_*=0\implies b =0\implies 
 	h^*([Z]) = \dfrac{5}{2}K_{\pi_J}-\xi+\pi_J^*\T_*\implies [Z]h^*([Z]) = -\xi^2+2K_{\pi_J}\pi_J^*\T_*+\pi_J^*\T_*^2.
 \end{align*}
 Consequently, we have
 $$[Z]h^*[Z] = h^*([Z]h^*[Z]) = -h^*(\xi^2)+2K_{\pi_J}\pi_J^*\T_*+\pi_J^*\T_*^2\implies h^*(\xi^2)=\xi^2.$$ 
 Therefore, $\xi^2\in \Q.K_{\pi_J}\pi_J^*\T_*$ since $\xi.\pi_J^*\T_*$ is not invariant under $h^*$. Let $\xi^2 = cK_{\pi_J}\pi_J^*\T_*$, then 
 $$[Z]h^*([Z]) = (2-c)K_{\pi_J}\pi_J^*\T_*+\pi_J^*\T_*^2\text{ and }
 [Z]h^*([Z])\pi_J^*\T_*=(2-c)K_{\pi_J}\pi_J^*\T_*^2.$$
  We have $[Z]h^*([Z])\pi_J^*\T_*= i_*(i^*[Z]i^*(h^*[Z]))$ where $i:\C_*\times_{M_*}\C_*\simeq {\C_*\times_{M_*}\T_*}\hr \C_J$ is the natural inclusion obtained from $\T_*\hr J_*(1)$. Using $i$ and restricting $[Z]h^*([Z])$, we get
   \begin{multline*}i^*[Z] = \Delta + \sigma_1\times_{M_*}\C_* = \Delta+\dfrac{1}{2}(\tilde{\psi}_1+\tilde{\psi}_2), i^*(h^*([Z]))=\tilde{\Delta}+\dfrac{1}{2}(\tilde{\psi}_1+\tilde{\psi}_2)\implies\\ i^*[Z]i^*(h^*[Z])= 3\tilde{\psi}_1\tilde{\psi}_2
  \implies [Z]h^*([Z])\pi_J^*\T_* = i_*(3\tilde{\psi}_1\tilde{\psi}_2) = 12i_*(\sigma_1\times_{M_*}\sigma_1) = 3K_{\pi_J}\pi_J^*\T_*^2. \end{multline*}
  Therefore, $c=-1$ and $\xi^2 = -K_{\pi_J}\pi_J^*\T_*$ giving all the generating relations in $S^*(\C_J)$.
\end{proof}
\begin{rem}
	The arguments in Lemma \ref{chow-3} do not extend over to the corresponding open in $\times_{M_*}^3\C_*$ directly. This further limits the scope of computations in Theorem \ref{chow-curve} to $C_J$. 
	Consequently, the arguments for $A^*(\C_J)$ cannot be obtained plainly from the ones in the aforementioned theorem.\\
\end{rem}
\noindent From the definition of $\p_\xi$, we have $\p_\xi = {Proj}_{J_*(3)}\Sym^\bullet (\pi_{J*}(K_{\pi_J}\otimes\xi))^\vee$. \\
Applying the Grothendieck-Riemann-Roch to the morphism $\pi_J$ on $\xi\ri \C_J$, we have \begin{equation*}ch(\pi_{J*}(K_{\pi_J}\otimes\xi))= \pi_{J*}(ch(K_{\pi_J}\otimes\xi)Td(K_{\pi_J}^\vee)) =\pi_{J*}((1+K_{\pi_J}+\xi-\dfrac{1}{2}K_{\pi_J}\pi_J^*\T_*)(1-\dfrac{1}{2}K_{\pi_J}))= 4-\T_*.\end{equation*}
Therefore, $c_1(\pi_{J*}(K_{\pi_J}\otimes\xi)) = -\T_*, c_2(\pi_{J*}(K_{\pi_J}\otimes\xi)) = \dfrac{1}{2}\T_*^2, c_i(\pi_{J*}(K_{\pi_J}\otimes\xi))=0\forall i\geq 3$. 
\begin{cor}\label{chow-p-xi}
Let $H_J:= c_1(\mo_{\p_\xi} (1))\in A^1(\p_\xi)$ and $\T_\xi:= \pi_\xi^*\T_*$ be the pullback of the theta divisor on $J_*(3)$. We have
$$A^*(\p_\xi) = \Q[\T_\xi, H_J]/(\T_\xi^3,H_J^4+H_J^3\T_\xi+\dfrac{1}{2}H_J^2\T_\xi^2).$$
\end{cor}
\begin{lemma}\label{normal-cj}
	The total chern class of the normal bundle $\mathcal{N}_{\C_J/\p_\xi}$ is given by
	$$c(\mathcal{N}_{\C_J/\p_\xi}) = 1 + (5K_{\pi_J}+4\xi+\pi_J^*\T_* )+ \pi_J^*\T_*(3\xi-2K_{\pi_J})+\dfrac{1}{2}\pi_J^*\T_*^2\in A^*(\C_J).$$
\end{lemma}
\begin{proof}
	We have the Euler sequence on the projective bundle $\p_\xi\xrightarrow{\pi_\xi}J_*(3)$ 
	$$0\ri \Omega^1_{\p_\xi/J_*(3)}\ri \pi_\xi^*\pi_{J*}(K_{\pi_J}\otimes\xi)\otimes \mo_{\p_\xi}(-1)\ri \mo_{\p_\xi}\ri 0$$ 
	which restricts to $0\ri \Omega^1_{\p_\xi/J_*(3)}|_{\C_J}\ri \pi_J^*\pi_{J*}(K_{\pi_J}\otimes\xi)\otimes \underbrace{\mo_{\p_\xi}(-1)|_{\C_J}}_{(K_{\pi_J}\otimes\xi)^\vee}\ri \mo_{\C_J}\ri 0$ over $\C_J$. Additionally, the conormal bundle sequence 
	$$0\ri \mathcal{N}_{\C_J/\p_\xi}^\vee\ri \Omega^1_{\p_\xi/J_*(3)}|_{\C_J}\ri K_{\pi_J}\ri 0\text{ gives }$$
	\begin{align*}c(\mathcal{N}_{\C_J/\p_\xi}^\vee)= c(\Omega^1_{\p_\xi/J_*(3)}|_{\C_J})(1-K_{\pi_J}) = c(\pi_J^*\pi_{J*}(K_{\pi_J}\otimes\xi)\otimes (K_{\pi_J}\otimes\xi)^\vee)(1-K_{\pi_J})\\=1-(5K_{\pi_J}+4\xi+\pi_J^*\T_*)+(3\xi-2K_{\pi_J})\pi_J^*\T_*+\dfrac{1}{2}\pi_J^*\T_*^2\end{align*}
	\text{ since }$c(\pi_{J*}(K_{\pi_J}\otimes\xi))=1-\T_*+\dfrac{1}{2}\T_*^2\in A^*(J_*(3))$. Therefore, $c(\mathcal{N}_{\C_J/\p_\xi})= 1 + (5K_{\pi_J}+4\xi+\pi_J^*\T_* )+ \pi_J^*\T_*(3\xi-2K_{\pi_J})+\dfrac{1}{2}\T_*^2$ as claimed.
\end{proof}
\begin{lemma}\label{expressions}
We have the following expressions for cycles in terms of generators in $A^*(\p_\xi)$.
 $$[\C_J]= 5H_J^2+3H_J\T_\xi+\dfrac{1}{2}\T_\xi^2,
[K_{\pi_J}]=2H_J^3+2H_J^2\T_\xi+H_J\T_\xi^2, [\xi] = 3H_J^3+H_J^2\T_\xi-\dfrac{1}{2}H_J\T_\xi^2$$
where we identify $K_{\pi_J},\xi \in A^1(\C_J)$ with their pushforwards in $A^3(\p_\xi)$. Consequently, $[\pi^*\s_i] = \dfrac{1}{2}[K_{\pi_J}] = H_J^3+H_J^2\T_\xi+\dfrac{1}{2}H_J\T_\xi^2\forall 1\leq i\leq 6$.
\end{lemma}
\begin{proof}
We start with proving the claimed expressions for $K_{\pi_J},\xi$. By Corollary \ref{chow-p-xi}, $A^3(\p_\xi)$ is freely spanned by $H_J^3, H_J^2\T_\xi, H_J\T_\xi^2$.
Let $\bullet$ be the intersection product in $A^*(\C_J)$. We have
$$H_J.[K_{\pi_J}] = j_{ *}((K_{\pi_J}+\xi)\bullet K_{\pi_J}) = 0\implies [K_{\pi_J}] = 2H_J^3+2H_J^2\T_\xi+H_J\T_\xi^2 $$
 since $[K_{\pi_J}],2H_J^3$ have identical restrictions on fibers of $\pi_\xi$. \\
Using a similar argument for $[\xi],3H_J^3$, we have 
\begin{align*}H_J.[\xi] = j_{ *}((K_{\pi_J}+\xi)\bullet \xi) = -j_{ *}(K_{\pi_J}\pi_J^*\T_*) = -j_{ *}(K_{\pi_J})\T_\xi = -[K_{\pi_J}].\T_\xi = -(2H_J^3\T_\xi+2H_J^2\T_\xi^2)\implies
	\\ [\xi] = 3H_J^3+H_J^2\T_\xi-\dfrac{1}{2}H_J\T_\xi^2\implies
H_J.[\C_J] = j_*(H_J|_{\C_J}) = j_*(K_{\pi}+\xi) = 5H_J^3+3H_J^2\T_\xi+\dfrac{1}{2}H_J\T_\xi^2.\end{align*}
 Using Corollary \ref{chow-p-xi}, it follows that $A^2(\p_\xi)$ is freely spanned by $\{H_J^2, H_J\T_\xi, H_J\T_\xi^2\}$. Therefore, we have 
 $$[\C_J]=5H_J^2+3H_J\T_\xi+\dfrac{1}{2}\T_\xi^2.$$
\end{proof}\noindent
Let $E_{\sigma}$ be the exceptional divisor of $\sigma_\pi$, and $i_\sigma:E_\sigma\hr \tilde{\p}_\xi$ be the inclusion map.\\
Following the notational setup for $\C_J$, let $S^*(\tilde{\p}_\xi)=Im(\tilde{q}^*)$ where $\tilde{q}:\tilde{\p}_\xi\ri Bl_{C_J}(\p_J)$ is the map induced by $q_C,q_J$.
 Applying Theorem \ref{blowup-chow} to Theorem \ref{chow-curve} and Corollary \ref{chow-p-xi} we get the following theorem.
\begin{theorem}
Let $\xi_\s:= i_{\sigma *}(\sigma_\pi|_{E_\sigma}^*\xi),\text{ and } K_\s:= i_{\sigma *}(\sigma_\pi|_{E_\sigma}^*K_{\pi_J})$. We have
\begin{gather*}S^*(\tilde{\p}_\xi)=\Q[\s_\pi^*H_J, E_\sigma, \xi_\s, \sigma_\pi^*\T_\xi, K_\s]/(\s_\pi^*(H_J^4+H_J^3\T_\xi+\dfrac{1}{2}H_J^2\T_\xi^2), \sigma_\pi^*\T_\xi^3, \xi_\s^2+E_\s K_\s \T_\xi, E_\s K_\s-K_\s\s_\pi^*\T_\xi+\\
	\s_\pi^*(2H_J^3+2H_J^2\T_\xi+H_J\T_\xi^2),
  (\xi_\s-\dfrac{3}{2}K_\s)\s_\pi^*\T_\xi^2,K_\s\s_\pi^*H_J, E_\s\s_\pi^*H_J-(K_\s+\xi_\s), \xi_\s\s_\pi^*H_J+K_\s\s_\pi^*\T_\xi,\\
  E_\s^2-(5K_\s+4\xi_\s+E_\s\s_\pi^*\T_\xi)+\s_\pi^*(5H_J^2+3H_J\T_\xi+\dfrac{1}{2}\T_\xi^2), K_\s^2, \xi_\s K_\s,
 E_\s\xi_\s-\xi_\s\s_\pi^*\T_\xi+4K_\s\s_\pi^*\T_\xi+\\
 \s_\pi^*(3H_J^3+H_J^2\T_\xi-\dfrac{1}{2}H_J\T_\xi^2))).\end{gather*}
\end{theorem}
\begin{proof}
Following the notational setup we have developed thus far, let $S^*(E_\s)$ by the image of the pullback map $A^*(\tilde{E})\ri A^*(E_\s)$ where $\tilde{E}$ is the exceptional divisor of the blowup of $\p_J$ along $C_J$. By Lemma \ref{normal-cj} and the projective bundle formula, we have 
\begin{multline}\label{E_s} A^*(\p\mathcal{N}_{\C_J/\p_\xi})\supset S^*(E_\s)= \Q[\sigma^*K_{\pi_J}, \sigma^*\pi_J^*\Theta_*,\sigma^*\xi,\zeta]/(\sigma^*K_{\pi_J}^2,\sigma^*\xi^2 + \s^*(K_{\pi_J}\pi_J^*\T_*), \sigma^*(\xi K_{\pi_J}),\\
\sigma^*((\xi-\dfrac{3}{2}K_{\pi_J})\pi_J^*\T_*^2), \sigma^*\T_\xi^3,\zeta^2+\zeta\sigma^*(5K_{\pi_J}+4\xi+\pi_J^*\T_*)+\s^*((3\xi-2K_{\pi_J})\pi_J^*\T_*+\dfrac{1}{2}\pi_J^*\T_*^2))\end{multline}
where $\s^*Z:= \s_\pi|_{E_\s}^*Z\quad\forall Z\in S^*(\C_J)$, and $\zeta$ is the relative hyperplane divisor in $E_\s \simeq Proj_{\C_J}\Sym^\bullet(\mathcal{N}_{\C_J/\p_\xi})^\vee$. \\
Furthermore, for $Z\in S^*(\C_J)$ we have $$i_{\s *}\sigma_\pi|_{E_\sigma}^*(Z\pi_J^*\T_*) = i_{\s *}(\s_\pi|_{E_\s}^*Z)\s_\pi^*\T_\xi.$$
 Let $\s_*:= \s_\pi|_{E_\s *}$. Using \cite[Proposition 0.1.3]{beauville}, we have 
  $$z = \s_\pi^*\s_{\pi *}z-i_{\s *}\s^*\s_*i_{\s}^*z \forall z\in A^*(\tilde{\p}_\xi).$$
Applying this formula to the case where $z = i_{\s *}(t)$ for $t\in A^*(E_\s)$ gives
$$z = \s_\pi^*\s_{\pi *}(i_{\s *} t)-i_{\s *}\s^*\s_*i_\s^*i_{\s *}(t) = \s_\pi^*j_*\s_* t+i_{\s *}\s^*\s_*(\zeta t)\quad\forall t\in A^*(E_\s).$$ 
Using $E_\s^2=-i_{\s*}(\zeta), E_\s K_\s = -i_{\s *}(\zeta \s^* K_{\pi_J}), E_\s\xi_\s =-i_{\s *}(\zeta \s^*\xi)$, (\ref{E_s}), and Lemma \ref{expressions}, we get the following elements in the ideal of relations of $S^*(\tilde{\p}_\xi)$:
\begin{multline*} E_\s^2-(5K_\s+4\xi_\s+E_\s\s_\pi^*\T_\xi)+\s_\pi^*(5H_J^2+3H_J\T_\xi+\dfrac{1}{2}\T_\xi^2), E_\s K_\s-K_\s\s_\pi^*\T_\xi+\s_\pi^*(2H_J^3+2H_J^2\T_\xi+H_J\T_\xi^2),\\
	E_\s\xi_\s-\xi_\s\s_\pi^*\T_\xi+4K_\s\s_\pi^*\T_\xi+\s_\pi^*(3H_J^3+H_J^2\T_\xi-\dfrac{1}{2}H_J\T_\xi^2) ,\xi_\s^2+E_\s K_\s\s_\pi^*\T_\xi, K_\s^2, \xi_\s K_\s,(\xi_\s-\dfrac{3}{2}K_\s)\s_\pi^*\T_\xi^2.\end{multline*}
With the above relations in mind, the following equations describe $S^*(\tilde{\p}_\xi)$
 \begin{align*}& i_{\s *}(Y)\s_\pi^*H_J^a\T_\xi^b = i_{\s *}(Y\s_{\pi}|_{E_\s}^*(K_{\pi_J}+\xi)^a\pi_J^*\T_*^b)\forall a,b\geq 0, Y\in A^*(E_\s)\implies\\ &  K_\s\s_\pi^*H_J = 0,\xi_\s\s_\pi^*H_J = -K_\s\s_\pi^*\T_\xi, E_\s\s_\pi^*H_J = K_\s+\xi_\s.\end{align*} 
 Since \begin{align*}&\{[\tilde{\p}_\xi], \s_\pi^*H_J, \s_\pi^*\T_\xi, E_\s, \s_\pi^*H_J^2, \s_\pi^*\T_\xi^2, \s_\pi^*(H_J\T_\xi), K_\s, \xi_\s, E_\s\s_\pi^*\T_\xi, \s_\pi^*H_J^3, \s_\pi^*(H_J^2\T_\xi),\\& \s_\pi^*(H_J\T_\xi^2), E_\s \s_\pi^*\T_\xi^2, K_\s\s_\pi^*\T_\xi,\xi_\s\s_\pi^*\T_\xi, \s_\pi^*(H_J^3\T_\xi), \s_\pi^*(H_J^2\T_\xi)^2, K_\s\s_\pi^*\T_\xi^2, \s_\pi^*(H_J^3\T_\xi^2)\}\end{align*}
  are independent linear generators of $S^*(\tilde{\p}_\xi)$ by \cite[Proposition 0.1.3]{beauville}.
\end{proof}
\subsection{Some loci in the blowup}\label{loci-blowup}
We define some loci in $\p_\xi$ and consider their proper transforms in $\tilde{\p}_\xi$.\\
Let $\pi_\phi:\p'\ri \C_*\times_{M_*}\C_*$ be the pullback of $\pi_\xi:\p_\xi\ri J_*(3)$ over $\phi\otimes\s_1: \C_*\times_{M_*}\C_*\ri J_*(3)$, let $\overline{\s}_j$ be the pullback of the section $\pi^*\s_j$ over $\phi\otimes\s_1$ for $ 1\leq j\leq 6$, and let $\p_{\phi(\s_1)}:\p'\ri \p_\xi$ be the induced map.
Put $\{\Delta_{i}\}_{i=1,2}$ as the diagonal sections for the morphism $\pi_{12}:\times^3_{M_*}\C_*\ri\C_*\times_{M_*}\C_*$.\\
$$\begin{tikzcd}
	& &\p_\xi\arrow[dd, "\pi_\xi", bend left=70]\\
	\p'\arrow[rru, "\p_{\phi(\s_1)}", bend left = 20]\arrow[rd, "\pi_\phi"',bend right=20 ]&\times^3_{M_*}\C_*\arrow[l,hook']\arrow[r, "\Phi_\C"]\arrow[d, "\pi_{12}"]\arrow[dr, phantom, "\square"]&\C_J\arrow[u,hook]\arrow[d, "\pi_J",bend left=20]\\
&\C_*\times_{M_*}\C_*\arrow[r,"\phi\otimes\s_1"']\arrow[u, "\Delta_i{,} \overline{\s}_j"', bend left=70]& J_*(3).\arrow[u, "\pi^* \s_j", bend left=20]
\end{tikzcd}
$$
We construct $L_1$ to be the $\p^1$-bundle over $\C_*\times_{M_*}\C_*$, the fibers of which are given by lines joining the sections $\Delta_1,\Delta_2$ over $\pi_\phi$.\\
More concretely, $L_1:= \p_{\C_*\times_{M_*}\C_*}(\pi_{\phi *}(\mathcal{F}|_{\Delta_1\cup \Delta_2}))^\vee\hr \p'$ where $\mathcal{F}:= \Phi_\C^*(K_{\pi_J}\otimes\xi)$. 
That is, $\p' = \p_{\C_*\times_{M_*}\C_*}(\pi_{\phi *}(\mathcal{F}))^\vee$.
The embedding $L_1\hr \p'$ follows from the exact sequence 
\begin{equation*} 0\ri\mathcal{F}(-\Delta_1-\Delta_2)\ri\mathcal{F}\ri \mathcal{F}|_{\Delta_1\cup \Delta_2}\ri 0,\end{equation*}
and $R^1 \pi_{12*}(\mathcal{F}(-\Delta_1-\Delta_2))=0$.
We use the exact sequence 
$$0\ri \mathcal{F}(-\Delta_1-\Delta_2-\overline{\s}_2)\ri \mathcal{F}\ri\mathcal{F}|_{\Delta_1\cup\Delta_2\cup\overline{\s}_2}\ri 0 $$ 
to define a hyperplane bundle $H_{2}:= \p_{\C_*\times_{M_*}\C_*}(\pi_{\phi *}(\mathcal{F}|_{\Delta_1\cup\Delta_2\cup\overline{\s}_2}))^\vee\hr \p'$.\\
To see that $H_2$ is a hyperplane bundle, we note that 
$$R^1\pi_{12*}(\mathcal{F}(-\Delta_1-\Delta_2-\overline{\s}_2))=0.$$
This is due to the cohomology and base change theorem along with the observation that $\mathcal{F}(-\Delta_1-\Delta_2-\overline{\s}_2), \mo(\overline{\s}_1+\overline{\s}_2)$ restrict to identical divisors on ever fiber of $\pi_{12}$.\\
Equivalently, $H_{2}\hr \p'$ is the hyperplane-bundle, the fiber of which over $q\in \C_*\times_{M_*}\C_*$, is given by the hyperplane which contains the tangent line to $(\times_{M_*}^3\C_*)_q$ at $(\overline{\s}_2)_q$ and $\{(\Delta_i)_q\}_{i=1,2}$.
\begin{defi*}
Let $L_{\xi_1}:= \p_{\phi(\s_1)}(L_1), H_{\xi_2}:= \p_{\phi(\s_1)}(H_2) $ be the images of these loci in $\p_\xi$. 
We define $\{H_{\xi_i}\}_{2\leq i\leq 6}$ using $\overline{\s}_i$ instead of $\overline{\s}_2$ and $\{L_{\xi_i}\}_{2\leq i\leq 6}$ analogously using the maps $\phi\otimes\s_i$ instead of $\phi\otimes\s_1$.
\end{defi*}
\noindent The next locus we define is a $\p^1$-bundle over $J_*(3)$ named $L_{2\s_1}$, the fiber of which over $p\in J_*(3)$ is given by line tangent to $(\C_J)_p$ at the point $(\pi^*\s_1)_p$.
\begin{defi*}
	Put $L_{2\s_1}:= \p_{J_*(3)}\pi_{\xi *}((K_{\pi_J}\otimes \xi)|_{2\s_1})$ where the ideal of the closed subscheme $2\s_1 \hr \C_J$ is given by $\mathcal{I}_{2\s_1/\p_\xi} := \mathcal{I}_{\pi^*\s_1/\p_\xi}^2$.
The embedding $L_{2\s_1}\hr \p_\xi$ is defined by the natural surjection $K_{\pi_J}\otimes \xi\twoheadrightarrow (K_{\pi_J}\otimes \xi)|_{2\s_1} $.\\
We shall also use the loci $\{L_{2\s_i}\}_{2\leq i\leq 6}$ and the loci $\{L_{\s_j\s_k}\}_{1\leq j<k\leq 6}$ , which are defined analogously: $\p_{J_*(3)}\pi_{\xi *}((K_{\pi_J}\otimes \xi)|_{X})$ with $\mathcal{I}_X = \mathcal{I}_{\pi^*\s_i/\p_\xi}^2 \text{ or }\mathcal{I}_{\pi^*\s_j/\p_\xi}\mathcal{I}_{\pi^*\s_k/\p_\xi}$ respectively.
\end{defi*}
\begin{rem}\label{h-2}
	It is worth noting that $L_{2\s_1}$ is a $\p^1$-bundle over $J_*(3)$ and $L_1$ is a $\p^1$-bundle over $\C_*\times_{M_*}\C_*$.
	 However, $L_{\xi_1}$ is not a $\p^1$-bundle over $J_*(3)$. More conceretely, $L_{\xi_1}$ is a $\p^1$-bundle bundle to $J_*(3)\setminus \mathcal{K}_1$ and $L_{\xi_1}, Q_\xi$ have identical restrictions over $\mathcal{K}_1$.\\
	 Similarly, $H_{\xi_2}$ is not a $\p^2$-bundle over $J_*(3)$. It is a $\p^2$-bundle over $J_*(3)\setminus \mathcal{K}_1$ and $H_{\xi_2}, \p_\xi$ have identical restrictions over $\mathcal{K}_1$.
\end{rem}
\noindent We can now calculate expressions for these loci and their proper transforms in $\p_\xi, \tilde{\p}_\xi$.
\begin{lemma}\label{expressions_below}
Fr all $1\leq i,j,k,l\leq 6$, we have the following expressions for the aforementioned loci in $A^*(\p_\xi)$:
	$$[H_{\xi_i}]= H_J+\T_\xi, [Q_\xi] = 2H_J+\T_\xi,  [L_{2\s_i}]=[L_{\s_j\s_k}] = H_J^2+H_J\T_\xi+\dfrac{1}{2}\T_\xi^2, [L_{\xi_l}] = H_J^2+2H_J\T_\xi+\dfrac{1}{2}\T_\xi^2.$$
\end{lemma}
\begin{proof}
	We start with $[L_{2\s_1}]$. The loci $\{L_{2\s_i}\}_{1\leq i\leq 6}$ are permuted under the $S_6$-action on $\p_\xi$ and thus have the same class in $A^*(\p_\xi)$ since the generators $H_J,\T_\xi$ are $S_6$-invariant. 
	Put $[L_{2\s_1}]=H_J^2+aH_J\T_\xi+b\T_\xi^2$ for some $a,b\in \Q$.\\
 Since $L_{2\s_1},L_{2\s_2}$ intersect transversally on the locus of vertices of cones in $Q_\xi$, and $L_{2\s_1},\pi^*\s_2$ intersect transversally on the vertex of the cones in $Q_\xi$ that are fibered over $\mathcal{K}_2$, we have $\pi_{\xi *}(L_{2\s_1}.L_{2\s_2})=\T_*, \pi_{\xi *}(L_{2\s_1}.\pi^*\s_2) = \mathcal{K}_2$. Hence, $2a-1=1\implies a=1, b =\dfrac{1}{2}$. Thus, $$[L_{2\s_1}] = H_J^2+H_J\T_\xi+\dfrac{1}{2}\T_\xi^2.$$
 Let $\mathcal{K}_{123}\hr J_*(3)$ be the locus of points in $J_*(3)$ over which $\xi, \mo(\s_1+\s_2+\s_3)\in Pic(\C_J)$ define identical divisors.
 It follows that $[\mathcal{K}_{123}]=[\mathcal{K}_1]=\dfrac{1}{2}\T_*^2$ in $A^*(J_*(3))$. We have
$$\pi_{\xi *}(L_{\s_1\s_2}.L_{\s_3\s_4})=\T_*, \pi_{\xi *}(L_{\s_1\s_2}.\pi^*\s_3) = \mathcal{K}_{123}.$$
Using similar arguments as above, we get 
$$[L_{\s_j\s_k}] = H_J^2+H_J\T_\xi+\dfrac{1}{2}\T_\xi^2\forall 1\leq j< k\leq 6.$$
	Let $[L_{\xi_1}]=H_J^2+cH_J\T_\xi+d\T_\xi^2$ for some $c,d\in\Q$.\\
The irreducible components of $L_{\xi_1}|_{\T_*}$ are: $L_{2\s_1}|_{\T_*}$ and $Q_\xi|_{\mathcal{K}_1}$. Each of them have multiplicity one when restricted to a general fiber over $\T_*$. Thus, 
$$[L_{\xi_1}]\T_\xi = [L_{2\s_1}]\T_\xi + \dfrac{1}{2}\T_\xi^2.[Q_\xi] = H_J^2\T_\xi+H_J\T_\xi^2+\dfrac{1}{2}\T_\xi^2.(2H_J) = H_J^2\T_\xi + 2H_J\T_\xi^2\implies c = 2.$$ 
For $1\leq i,j\leq 6$, let $\xi_{ij}\hr J_*(3)$ be the locus of points such that $\{H^0(\C_p, (\xi(-\s_i-\s_j))|_{\C_p})> 0\}_{p\in J_*(3)}$. 
We have that $L_{\xi_1}, L_{\s_2\s_3}$ coincide over the locus $\mathcal{K}_{123}$ and intersect transversally along $\pi^*\s_2,\pi^*\s_3$ over the loci $\xi_{12},\xi_{13}$ respectively. Using Theorem \ref{expressions}, we have
	\begin{gather*}[L_{\xi_1}][L_{\s_2\s_3}] = L_{\s_2\s_3}.\mathcal{K}_{123}+[\pi^*\s_2].\pi_\xi^*\xi_{12}+[\pi^*\s_3].\pi_\xi^*\xi_{13}  = \dfrac{1}{2}L_{\s_2\s_3}\T_\xi^2+[K_{\pi_J}]_{\p_\xi}\T_\xi\implies\\ 2H_J^3\T_\xi+(2+d)H_J^2\T_\xi^2 = 2H_J^3\T_\xi+\dfrac{5}{2}H_J^2\T_\xi^2
		\implies d=\dfrac{1}{2}\implies [L_{\xi_1}]=H_J^2+2H_J\T_\xi+\dfrac{1}{2}\T_\xi^2.\end{gather*}
	Moving to $H_{\xi_2}$, put $[H_{\xi_2}]=H_J+\alpha\T_\xi$ for some $\alpha\in \Q$.\\
	We restrict our attention to $J_*(3)\setminus \mathcal{K}_1$.\\
	The transverse intersection of $(H_{\xi_2})_{J_*(3)\setminus \mathcal{K}_1},(\pi^*\s_3)_{J_*(3)\setminus \mathcal{K}_1}$ shows that 
	$$\pi_{\xi *}([(H_{\xi_2})_{J_*(3)\setminus \mathcal{K}_1}\cap(\pi^*\s_3)_{J_*(3)\setminus \mathcal{K}_1}]) = \T_*,$$
	where we identify $\pi_\xi$ with it's restriction over $J_*(3)\setminus \mathcal{K}_1$, giving us $\alpha =1$.\\
	Lastly, we find $[Q_\xi] = 2H_J+e\T_*$ for some $e\in \Q$.\\
	Restricting to the open $T:= J_*(3)\setminus (\mathcal{K}_1\cup \mathcal{K}_2)\subset J_*(3)$, we have $$(Q_\xi.L_{\s_1\s_2})_T = (2H_J+e\T_*)(H_J^2+H_J\T_*)= 2H_J^3+(e+2)H_J^2\T_\xi.$$
	The intersection $(Q_\xi\cap L_{\s_1\s_2})_T$ is transverse over $(\pi^*\s_1)_T,(\pi^*\s_2)_T, L_{\s_1\s_2}|_{\xi_{12}}$. That is, 
	$$(Q_\xi.L_{\s_1\s_2})_T = [\pi^*\s_1]_T+[\pi^*\s_2]_T+([L_{\s_1\s_2}].\xi_{12})_T = 2H_J^3+3H_J^2\T_\xi\implies e=1\implies [Q_\xi]=2H_J+\T_\xi.$$
\end{proof}
\noindent We can calculate the proper transforms of these cycles in $A^*(\tilde{\p}_\xi)$. The proper transform of $Q_\xi$ is $E_\xi$. Let $\tilde{L}_{\alpha}, \tilde{H_\beta}$ be the proper transforms of the loci $L_\alpha, H_\beta$ for $\alpha\in \{\xi_i,2\s_i\}_{1\leq i\leq 6}, \beta\in \{\xi_j\}_{2\leq j\leq 6}$ respectively.
Additionally,
\begin{equation}\label{tilde-h-xi2}
[\tilde{H}_{\xi_2}] = \s_\pi^*([H_{\xi_2}]) = \s_\pi^*(H_J+\T_\xi) \in A^*(\tilde{\p}_\xi).
\end{equation}
\begin{theorem}\label{exp-blowup}
	We have the following expressions for the aforementioned loci in $A^*(\tilde{\p}_\xi)$:
	$$[E_\xi] = \s_\pi^*(2H_J+\T_\xi)-E_\s, [\tilde{L}_{2\s_1}]=\s_\pi^*(H_J^2+H_J\T_\xi+\dfrac{1}{2}\T_\xi^2)-K_\s, [\tilde{L}_{\xi_1}] = \s_\pi^*(H_J^2+2H_J\T_\xi+\dfrac{1}{2}\T_\xi^2)-\xi_\s-E_\s\s_\pi^*\T_\xi.$$
 Let $\tilde{\s}_i := \s_\pi^{-1}(\pi^*\s_i)\cap E_\xi\text{ for } 1\leq i\leq 6$, it follows that $[\tilde{\s}_i] = \s_\pi^*(H_J^3+H_J^2\T_\xi+\dfrac{1}{2}H_J\T_\xi^2)$.
\end{theorem}
\begin{proof}
	For a subvariety $V\hr \p_\xi$ of dimension $d$, let $\tilde{V}$ be its proper tranform. If $\dim (V\cap \C_J) \leq d-1$, by Theorem \ref{transform}, we have
	$$[\tilde{V}] = \s_\pi ^*(V)- i_{\s *}(\s_\pi|_{E_\s} ^*(V\cap\C_J)_{(d-1)}).$$ 
	Let $Z= Im(\overline{\Sigma})$ be as defined in the proof of Theorem \ref{chow-curve}, we have that 
$$[E_\xi] = \s_\pi^*([Q_\xi])-E_\s, [\tilde{L}_{2\s_1}]=\s_\pi^*([L_{2\s_1}])-K_\s, [\tilde{L}_{\xi_1}]= \s_\pi^*([L_{\xi_1}])- i_{\s *}(\s_\pi|_{E_\s} ^*(Z\cup \pi^*\s_1)).$$ 	
Applying (\ref{chern}), we have
	$$ [Z\cup \pi^*\s_1] = \xi + \pi_J^*\T_*\implies [\tilde{L}_{\xi_1}]= \s_\pi^*([L_{\xi_1}])-\xi_\s-E_\s\s_\pi^*\T_\xi.$$ 
The intersection of $\s_\pi^{-1}(\pi^*\s_1), E_\xi$ is generically transverse and the intersection maps isomorphically onto $\pi^*\s_1$ over $J_*(3)\setminus \mathcal{K}_1$ via $\s_\pi$.\\
Hence, $[\tilde{\s}_1] = i_{\s *}(\s^* \pi^*\s_1).E_\xi = \dfrac{1}{2}K_\s E_\xi$. 
Using the expression for $E_\xi$ and Theorem \ref{blowup-chow}, we have the claimed expression for $[\tilde{\s}_1]$
$$[\tilde{\s}_1] = \s_\pi^*(H_J^3+H_J^2\T_\xi+\dfrac{1}{2}H_J\T_\xi^2).$$ 
The claim for $\tilde{\s}_i$ for other values of $i$ follows from the $S_6$-action on $\tilde{\p}_\xi$ defined using the $S_6$-actions on $\C_J, \p_\xi$.
\end{proof}
\begin{defi*} 
Let $L_{2\s_1}|_{\T}, \tilde{L}_{2\s_1}|_{\T}, \tilde{\s}_1|_\T$ be the restriction of $L_{2\s_1}, \tilde{L}_{2\s_1}, \tilde{\s}_1$ to $\T_*$ respectively and 
$\p_{2\s_1}$ be the proper transform of $L_{2\s_1}|_{\T}$ under $\s_\pi$. We have 
$$\tilde{L}_{2\s_1}|_{\T} = \p_{2\s_1}\cup \s_\pi^{-1}(\Delta)$$
where $\Delta\hr \C_J$ is the image of $\Delta_1\cap (\C_*\times_{M_*}\C_*\times_{M_*}\s_1)$ by the map $\Phi_\C: \times_{M_*}^3\C_*\ri \C_J$.
Let $\tilde{\s}:= \overline{\tilde{\s}_1|_\T\setminus \s_{\pi}^{-1}(\Delta|_{\mathcal{K}_1})}$. The locus $\tilde{\s}_1|_\T$ is reducible with irreducible components $\tilde{\s},\s_\pi^{-1}(\Delta|_{\mathcal{K}_1})$. Hence 
$$\tilde{\s}_1|_\T = \tilde{\s}\cup \s_\pi^{-1}(\Delta|_{\mathcal{K}_1}).$$
\end{defi*}
\begin{theorem}\label{tilde-s-pi-delta}
We have the following relations in $A^*(\tilde{\p}_\xi)$:
$$[\s_\pi^{-1}(\Delta)]=(\xi_\s-K_\s)\s_\pi^*\T_\xi+\dfrac{1}{2}E_\s\s_\pi^*\T_\xi^2,\quad
[\tilde{\s}] = \s_\pi^*(H_J^3\T_\xi+H_J^2\T_\xi^2)-\dfrac{1}{4}K_\s\s_\pi^*\T_\xi^2.$$
\end{theorem}	
\begin{proof}
The restriction of $Z$ to $\T_*$, namely $Z|_{\T_*}$, is reducible with irreducible components $\Delta, \pi^*\s_1|_{\T_*}, \pi_J^{-1}(\mathcal{K}_1)$. 
Restriction to a general fiber over $\T_*$ shows that multiplicities of $\Delta, \pi^*\s_1|_{\T_*} $ are one. The multiplicity of $\pi_J^{-1}(\mathcal{K}_1)$ in $[Z|_{\T_*}]$ follows from restriction to $\pi^*\s_2|_{\T_*}$, giving us the following in $A^*(\C_J)$
	\begin{gather*}[Z|_{\T_*}] = [\Delta]+[\pi^*\s_1|_{\T_*}]+[\pi_J^{-1}(\mathcal{K}_1)]\implies (\xi-\dfrac{1}{2}K_{\pi_J}+\T_*)\T_* = [\Delta]+\dfrac{1}{2}K_{\pi_J}\T_*+\dfrac{1}{2}\T_*^2      \\
		\implies [\Delta] = (\xi-K_{\pi_J}+\dfrac{1}{2}\T_*)\T_* .\end{gather*}
	This translates to the following in $A^*(\tilde{\p}_\xi)$ 
	\begin{equation*}\label{s-pi-Delta}
		[\s_\pi^{-1}(\Delta)] =i_{\s *}\s^*([\Delta]_{A^*(\C_J)}) = (\xi_\s-K_\s)\s_\pi^*\T_\xi+\dfrac{1}{2}E_\s\s_\pi^*\T_\xi^2.\end{equation*} 
The definition of $\tilde{\s}$ provides the following relation in $A^*(\tilde{\p}_\xi)$
\begin{equation*}\label{tilde-s}
	[\tilde{\s}_1]\s_{\pi}^*\T_\xi = [\tilde{\s}]+[\s_{\pi}^{-1}(\Delta)]\dfrac{1}{2}\s_{\pi}^*\T_\xi \implies [\tilde{\s}] = \s_\pi^*(H_J^3\T_\xi+H_J^2\T_\xi^2)-\dfrac{1}{4}K_\s\s_\pi^*\T_\xi^2.\end{equation*}
	\end{proof}
\subsection{\textbf{Linear generators of }$\mathbf{A_*(B_*)}$}\label{gen-b-*}
	We stratify $B_* = B_V\cup B_\T$ to compute generators of $A^*(B_*)$. The excision sequence is then 
	\begin{equation}\label{excision-B}A_{*-1}(B_\T)\ri A_*(B_*)\ri A_*(B_V)\ri 0\tag{\textdagger\textdagger}.\end{equation} 
	We calculate Chow groups of $B_V, B_\T$ and then use this sequence to find generators of $A_*(B_*)$.
\subsubsection{$\mathbf{A^*(B_V)}$}\label{B-V}
 Let $E_V , Q_V, \p_V, \tilde{\p}_V  \text{ and } \C_V\xrightarrow{\pi_V}V$ be the restrictions of $E_\xi,Q_\xi,\p_\xi,\tilde{\p}_\xi$  and the family $\pi_J$ to $V$ respectively. 
 Put $\p^K_*:= P_{V}(\pi_{V*}K_{\pi_V})^\vee, \p^\xi_*:= P_{V}(\pi_{V*}\xi|_{\C_V})^\vee$. We have the following isomorphism.
\begin{lemma}\label{isom-KL}
There is a canonical isomorphism 
$$\pi_{V*}(K_{\pi_V}\otimes \xi|_{\C_V})\simeq \pi_{V*}(K_{\pi_V})\otimes \pi_{V*}(\xi|_{\C_V}).$$
\end{lemma}	
\begin{proof}
The following canonical morphisms are surjections by the cohomology and base change theorem 
$$\pi_{V}^*\pi_{V*}K_{\pi_V}\twoheadrightarrow K_{\pi_V}, \pi_{V}^*\pi_{V*}\xi|_{\C_V}\twoheadrightarrow \xi_{\C_V}.$$
Hence, we have the surjection $\pi_V^*(\pi_{V*}K_{\pi_V}\otimes \pi_{V*}\xi|_{\C_V})\twoheadrightarrow K_{\pi_V}\otimes \xi|_{\C_V}$. Using adjoint properties of $\pi_V^*, \pi_{V*}$; we have the morphism 
$$\pi_{V*}K_{\pi_V}\otimes \pi_{V*}\xi|_{\C_V}\xrightarrow{\otimes_{K\xi}} \pi_{V*}(K_{\pi_V}\otimes \xi|_{\C_V}).$$ 
Another application of cohomology and base change theorem shows that $(\otimes_{K\xi})_v$ is an isomorphism for all $v\in V$.
\end{proof}\noindent
	\begin{theorem}\label{Q_V}
 The isomorphism $Q_V\simeq \p^K_*\times_{V}\p^\xi_*$ defines rulings on $Q_V$ and the embedding $Q_V\hr \p_V$ coincides with the Segre embedding $\p^K_*\times_{V}\p^\xi_*\hr \p_V$ via this isomorphism. 
	\end{theorem}
	\begin{proof}
		We show that there is a universal analogue of the Segre embedding $\p^K_*\times_{V}\p^\xi_*\hr \p_V\simeq P_V(\pi_{V*}(K_{\pi_V}\otimes \xi|_{\C_V}))^\vee$, as outlined in Remark \ref{rulings}.\\
		Let $\pi^K: \p^K_*\ri V , \pi^\xi: \p^\xi_*\ri V, \pi^{K\xi}: \p^K_*\times_{V}\p^\xi_*\ri V$ be the respective canonical morphisms. We have the surjections $\pi^{K*}\pi_{V*}K_{\pi_V}\twoheadrightarrow \mo_{\p^K_*}(1), \pi^{\xi*}\pi_{V*}\xi|_{\C_V}\twoheadrightarrow \mo_{\p^\xi_*}(1)$, inducing the surjection $\pi^{K\xi *}(\pi_{V*}K_{\pi_V}\otimes \pi_{V*}\xi|_{\C_V})\twoheadrightarrow \pi_1^* \mo_{\p^K_*}(1)\otimes \pi_2^* \mo_{\p^\xi_*}(1)$ where $\pi_1,\pi_2$ are projections onto the first and second factors of $\p^K_*\times_{V}\p^\xi_* $.\\
		We have the surjection $\pi^{K\xi *}\pi_{V*}(K_{\pi_V}\otimes \xi|_{\C_V})\twoheadrightarrow \pi_1^* \mo_{\p^K_*}(1)\otimes \pi_2^* \mo_{\p^\xi_*}(1)$ by Lemma \ref{isom-KL}, inducing the morphism $\p^K_*\times_{V}\p^\xi_*\xrightarrow{i_{K\xi}} \p_V$. 
		The restriction to fibers over $V$ of $i_{K\xi}$ is the Segre embedding by Remark \ref{rulings}. Hence, $i_{K\xi}$ is an embedding and its image is a family of smooth quadric surfaces parametrized by $V$.
		\\The relative linear systems $\C_V\xrightarrow{|K_{\pi_V}|} \p^K_*, \C_V\xrightarrow{|\xi_{\C_V}|} \p^\xi_*$ induce an embedding $\C_V\hr \p^K_*\times_V \p^\xi_*$. This embedding, coupled with $\p^K_*\times_{V}\p^\xi_*\hr \p_V$ implies that $Q_V\simeq \p^K_*\times_{J_*(3)}\p^\xi_*$ since we have fiberwise equality of images over $V$ by Remark \ref{rulings}.
	\end{proof}\noindent
	Consequently, we can conclude the following about $A^*(E_V), A^*(B_V)$.
	\begin{theorem}\label{classes-E_V}
		The restrictions of the cycles $E_\xi, \tilde{L}_{2\s_1}, \tilde{L}_{\xi_1},\tilde{\s}_1$ to $V$ linearly span $A^*(E_V)$. 
		The restriction $(\tilde{L}_{2\s_1})_V$ of $\tilde{L}_{2\s_1}$, to $V$, maps isomorphically onto $B_V$ via $\Phi_\xi$.
	\end{theorem}
	\begin{proof}
		By Remark \ref{rulings}, $(\tilde{L}_{2\s_1})_V,(\tilde{L}_{\xi_1})_V$ are contained in $E_V$ and belong to the rulings $\p_*^\xi, \p_*^K$ respectively.\\
		The projections $E_V\xrightarrow{\pi_K} \p_*^K\ri V$ form a tower of $\p^1$-bundles with sections $\gamma_1: V\ri \p_*^K, \gamma_2: \p_*^K\ri E_V$ given by 
		the images of $(\pi^*\s_1)_V$ under the relative linear systems $|K_{\pi_V}|,\p_*^K\times_{V}|\xi_{\C_V}| $ respectively. 
		Hence, $A^*(E_V)$ is linearly spanned by 
		$$E_V, \pi_K^{-1}(Im(\gamma_1)) = \tilde{L}_{2\s_1}, Im(\gamma_2) = \tilde{L}_{\xi_1},\text{ and }Im(\gamma_2)\cap \pi_K^{-1}(Im(\gamma_1)) = \tilde{L}_{2\s_1}\cap \tilde{L}_{\xi_1} = \tilde{\s}_1$$ 
		since $A^*(V) = \mathbb{Q}$. 
		By Theorem \ref{quadric}, we have $E_V\simeq Q_V$  and the fibers of the map $E_V\ri B_V$ are identical to the fibers of the projection $\p^K_*\times_{V}\p^\xi_* \ri \p^\xi_*$ by Remark \ref{rulings}. \\
		The sections $\sigma_i^K: V\xrightarrow{\sigma_i|_V} \C_V\xrightarrow{|K_{\pi_V}|}\p^K_*$ for $1\leq i\leq 6$, induce isomorphisms $b_i:\p^\xi_*\ri B_V$ by Remark \ref{rulings} and Theorem \ref{zmt} where we identify the sections $\sigma_i$ with their pullbacks to $J_*(3)$.
		 In fact, the images of any closed point under $\{b_i\}_{1\leq i\leq 6}$ are equal via by Theorem \ref{quadric} and Theorem \ref{newstead}.\\
		This shows that the maps $b_i$ are equal for all $1\leq i\leq 6$. 
		The isomorphism $b_1$ corresponds to $\Phi_\xi$ restricted to $(\tilde{L}_{2\s_1})_V$ proving the last part of the theorem.
	\end{proof}
\begin{cor}\label{B_V-chow}
The image of $(\tilde{\s}_1)_V$ under $\Phi_\xi$ is a section of the $\p^1$-bundle $B_V$ over $V$. Therefore, $A_*(B_V)$ is linearly spanned by the cycles $[B_V], [\Phi_\xi(\tilde{\s}_1)_V]$.
\end{cor}
\subsubsection{$\mathbf{A^*(B_\Theta)}$}\label{B_T}
	 Let $E_\T, Q_\T, \p_\T, \tilde{\p}_\T$ and $\C_\T\xrightarrow{\pi_\T}\T_*$ be the pullbacks of $E_\xi, Q_\xi, \p_\xi, \tilde{\p}_\xi$  and the restriction of the family $\pi_J$, respectively, over $\T_*\hr J_*(3)$.\\
	The locus $Q_\T\hr \p_\T$ is a family of quadric cones with vertices parameterized by $\Delta$ by Theorem \ref{quadric}. Let $E_\Delta := \sigma_\pi^{-1}(\Delta)$ be the exceptional locus over $\Delta\hr \C_J$.\\
	The restriction of $\Phi_\xi$ to $E_\Delta$ is an isomorphism and $Q_\T\setminus\Delta\simeq E_\T\setminus E_\Delta$ by Lemma \ref{sig-iso}. 
  \begin{lemma}
  The blowup $bl_\Delta Q_\T$ embeds into $\tilde{\p}_\T$ with image $E_\T$.
  \end{lemma}
	\begin{proof} 
	We have $\tilde{\p}_\T = bl_{\C_\T} \p_\T$ since $\C_J, Q_\xi$ are flat over $J_*(3)$. Therefore, we have the embedding $bl_{\C_\T}Q_\T\hr \tilde{\p}_\T$ and $E_\T = bl_{\C_\T}Q_\T$ by Remark \ref{blowup-blowdown}.\\
We have an induced morphism $E_\T\ri bl_{\Delta}Q_\T$ which induces a bijection of closed points on each fiber over $\T_*$. 
Therefore, we have $E_\T = bl_{\C_\T}Q_\T \simeq bl_\Delta Q_\T$. 
\end{proof}\noindent
Additionally, $\Delta$ is a Cartier divisor on $\C_\T$ and $\C_\T\hr Q_\T$. Therefore, we have an induced embedding $i_\C:\C_\T\simeq bl_{\Delta}\C_\T\hr bl_{\Delta}Q_\T = E_\T$.
\begin{rem}\label{delta-sig}
By Lemma \ref{sig-iso}, we have $\s_\pi^{-1}(\Delta)$ maps isomorphically onto $B_\T$ via $\Phi_\xi$.
\end{rem}
\noindent Let $Y:= P_{\T_*}(\pi_{\T *}K_{\pi_\T})^\vee\xrightarrow{p_Y}\T_*$ be the $g^1_2$ space for $\C_\T\xrightarrow{\pi_\T}\T_*$. Following a similar line of reasoning as in Remark \ref{cone-ruling}, the composition $\phi_\xi: \C_\T \xrightarrow{i_\C} E_\T \xrightarrow{\Phi_\xi} B_\T$ is invariant under the hyperelliptic involution, namely $h_\T$, of $\C_\T$. Hence,
$\phi_\xi$ factors through $\tilde{\phi}_\xi:Y \xrightarrow{\sim} B_\T$ and is an isomorphism by Theorem \ref{zmt}. \\
 Let $\delta\in H^1(\T_*, \lb_\Delta^\vee\otimes \pi_{\T *}K_{\pi_\T}^2)$ be the extension class of the following exact sequence on $\T_*$ 
$$0\ri \pi_{\T *}K_{\pi_\T}^2\ri \pi_{\T *}(K_{\pi_\T}^2(\Delta))\ri {\lb_\Delta}\ri 0$$ 
where $\lb_\Delta:= \pi_{\T *}(\mo_\Delta(K_{\pi_\T}^2(\Delta)))$. The sequence of morphisms 
$$H^1(\T_*, \lb_\Delta^\vee\otimes \pi_{\T *}K_{\pi_\T}^2)\xrightarrow{\pi_Y^*}H^1(Y, \pi_Y^*\pi_{\T *}K_{\pi_\T}^2\otimes\pi_Y^*\lb_\Delta^\vee)\simeq H^1(Y, \pi_Y^*\pi_{Y*}\mo_Y(2)\otimes\pi_Y^*\lb_\Delta^\vee)\ri H^1(Y, \mo_Y(2)\otimes\pi_Y^*\lb_\Delta^\vee)$$
 maps $\delta$ to the class of an extension 
 $$0\ri \mo_Y(2)\ri\E\ri \pi_Y^*\lb_\Delta\ri 0$$ 
 on $Y$. Let $\p_Y:= P_Y(\E^\vee)\xrightarrow{p_\E}Y$.
 The surjection $\E\twoheadrightarrow \pi_Y^*\lb_\Delta$ defines a section $s: Y\ri \p_Y$.
\begin{theorem}\label{diagrams-pY}
    We have $\p_Y\simeq E_\T$ and the induced morphism $\p_Y\ri Q_\T$ maps the section $s:Y\ri\p_Y$ to the section $\Delta:\T_*\ri \C_\T\hr \p_\T$. 
\end{theorem}
\begin{proof}
In order to define the isomorphism between $\p_Y, E_\T$, we construct a blow-down morphism $\p_Y\ri \p_\T$.
The extension class of $\E$ induces the following commutative diagram:
$$
\begin{tikzcd}
	0\ar[r]&\pi_Y^*\pi_{Y*}\mo_Y(2)\simeq\pi_Y^*\pi_{\T *}K_{\pi_\T}^2\ar[r]\ar[d, twoheadrightarrow]&\pi_Y^*\pi_{\T *}K_{\pi_\T}^2(\Delta)\ar[r]\ar[d]& \pi_Y^*\lb_\Delta\ar[equal]{d}\ar[r]&0\\
	0\ar[r]&\mo_Y(2)\ar[r]&\E\ar[r]& \pi_Y^*\lb_\Delta\ar[r]&0.
\end{tikzcd}
$$
The left and right vertical arrows are surjective. By the snake lemma, the induced map $\pi_Y^*\pi_{\T *}(K_{\pi_\T}^2(\Delta))\ri \E $ is a surjection. This induces an embedding $i_Y:\p_Y\hr Y\times_{\T_*}\p_\Delta$.\\ 
We shall show that the restricted projection $\p_Y\xrightarrow{\pi_\E}\p_\Delta\simeq\p_\T$ is the desired blow-down morphism mapping onto $Q_\T$. The isomorphism between $\p_Y$ and $E_\T$ would then follow from the universal property of blowups.\\
Next, we show that the embedding $\C_\T\xrightarrow{j|_{\C_\T}}\p_\T$ factors through $\pi_\E$. \\
Let $\E':= |K_{\pi_\T}|^*(\E)$ be the pullback of $\E$ to $\C_\T$ via the $g^1_2$ map $|K_{\pi_\T}|:\C_\T\ri Y$ and, for any sheaf $G$ on $\C_\T$, put $\mathcal{F}_G:=ker(\pi_\T^*\pi_{\T *}G\ri G)$.\\
The following diagram shows that the surjection $\pi_\T^*\pi_{\T *}K_{\pi_\T}^2(\Delta)\twoheadrightarrow K_{\pi_\T}^2(\Delta)$ factors via $\pi_{\T}^*\pi_{\T *}K_{\pi_\T}^2(\Delta)\twoheadrightarrow \E'$
$$\begin{tikzcd}
	&\mathcal{F}_{K_{\pi_\T}^2}\arrow[d,hook]\arrow[r,hook]&\mathcal{F}_{K_{\pi_\T}^2(\Delta)}\arrow[d,hook] & &\\
	0\arrow[r]&\pi_\T^*\pi_{\T *}K_{\pi_\T}^2\arrow[r]\arrow[d,twoheadrightarrow]&\pi_\T^*\pi_{\T *}K_{\pi_\T}^2(\Delta)\arrow[r]\arrow[d, twoheadrightarrow]\arrow[dd, twoheadrightarrow, bend left = 40]& \pi_\T^*\lb_\Delta\arrow[equal]{d}\arrow[r]&0\\
	0\arrow[r]&K_{\pi_\T}^2\arrow[r]\arrow[rd]&\E'\arrow[d, dashed, twoheadrightarrow]\arrow[r]& \pi_\T^*\lb_\Delta\arrow[r]\arrow[d, twoheadrightarrow]&0\\
 & &K_{\pi_\T}^2(\Delta)\arrow[r,twoheadrightarrow]&K_{\pi_\T}^2(\Delta)|_{\Delta} \arrow[r]&0.
\end{tikzcd}$$
A straightforward diagram chase shows that $ker(\pi_\T^*\pi_{\T *}K_{\pi_\T}^2(\Delta)\ri \E') = \mathcal{F}_{K_{\pi_\T}^2}\hr \mathcal{F}_{K_{\pi_\T}^2(\Delta)}$ and the surjection $\pi_\T^*\pi_{\T *}K_{\pi_\T}^2(\Delta)\twoheadrightarrow K_{\pi_\T}^2(\Delta)$ factors via $\E'$, giving us the following diagram
\begin{equation}\label{F_2-uni}
\begin{tikzcd}
	& & \p_Y\arrow[rd, "\pi_\E"]\arrow[d, "p_\E"] \simeq E_\T&\\
	\Delta\arrow[r,hook]&\C_\T\arrow[ru,  bend left = 40]\arrow[r, "|K_{\pi_\T}|"']\arrow[rr, "j|_{\C_\T}", bend right = 50]&Y \arrow[u, "s",  bend left = 40]& \p_\T.	
\end{tikzcd}\end{equation}
The embedding $i_Y$ shows that fibers of $\p_Y$ over $Y$ map to lines in $\p_\Delta$ via $\pi_\E$. For $y\in Y$, the image of $p_\E^{-1}(y)$ embeds in $\p_\T$ as a line via $\pi_\E$, joining the two(possibly degenerate) points in $|K_{\pi_\T}|^{-1}(y)$.\\
Additionally, the image of the section $s$ over $Y$ via $\pi_\E$ is the section $\Delta$ of $\p_\T$ over $\T_*$, where we identify $\Delta$ with its image under $j$.\\
Therefore, $\p_Y$ maps onto $Q_\T$ via $\pi_\E$ using Theorem \ref{quadric}. The universal property of blowups and Remark \ref{rulings} then imply that $\p_Y\simeq E_\T$ as claimed.
\end{proof}
\begin{cor}\label{phi-s}
    The composition $\Phi_\xi\circ s: Y\ri \p_Y\simeq E_\T \ri B_\T$ maps $Y$ isomorphically onto $B_\T$ by Remark \ref{delta-sig}.
\end{cor}
\begin{theorem}\label{classes-E_T}
For any $\alpha\in A^*(E_\T)$, put  $\alpha|_{\mathcal{K}_1}:= (\s_\pi\circ\pi_\xi)^{-1}(\mathcal{K}_1)\cap\alpha$ and $E_{\mathcal{K}_1}:= E_\T|_{\mathcal{K}_1}$.
The following cycles linearly span $A^*(E_\T)$ 	
$$E_\T, E_{\mathcal{K}_1}, \s_\pi^{-1}(\Delta), \p_{2\s_1}, \p_{2\s_1}|_{\mathcal{K}_1}, \s_\pi^{-1}(\Delta)|_{\mathcal{K}_1}, \tilde{\s}, \tilde{\s}|_{\mathcal{K}_1}.$$ 
\end{theorem}
\begin{proof}
	The tower of $\p^1$-bundles $\p_Y\xrightarrow{p_\mathcal{E}} Y\xrightarrow{p_{Y}} \T_*$ with $Y\xrightarrow[\simeq]{p_Y} \p^1\times \T_*$  shows that 
	$$A^*(E_\T) = p_\mathcal{E}^*(A^*(Y))\oplus \s_\pi^{-1}(\Delta).p_\mathcal{E}^*(A^*(Y)), A^*(Y) \simeq \Q[\zeta_Y,p_Y^*(\mathcal{K}_1)]/(p_Y^*(\mathcal{K}_1)^2,\zeta_Y^2)$$
	where $\zeta_Y$ is the relative hyperplane class of $Y$ over $\T_*$. Put $\Gamma:= |K_{\pi_\T}|(\s_1)$. We have that $A^*(Y)$ is linearly spanned by 
	$$[Y],[\Gamma], [p_Y^*(\mathcal{K}_1)], [p_Y^*(\mathcal{K}_1)\cap \Gamma]$$ 
	 since $\Gamma$ is a section of $p_Y$. The locus $\p_{2\s_1}$ coincides with the pullback $p_\mathcal{E}^*(\Gamma)$, which intersects $\s_\pi^{-1}(\Delta)$ transversally. 
	Thus, $A^*(E_\T)$ is linearly spanned by 
	$$E_\T, E_{\mathcal{K}_1}, \s_\pi^{-1}(\Delta), \p_{2\s_1}, \p_{\mathcal{K}_1}, \s_\pi^{-1}(\Delta)|_{\mathcal{K}_1}, \p_{2\s_1}\cap \s_\pi^{-1}(\Delta), (\p_{2\s_1}\cap \s_\pi^{-1}(\Delta))|_{\mathcal{K}_1}$$ 
	as claimed, since $\p_{2\s_1},\s_{\pi}^{-1}(\Delta), (\p_{2\s_1}\cap\s_{\pi}^{-1}(\Delta)) $ are flat over $\T_*$.
	Additionally, $\tilde{\s}, \p_{2\s_1}\cap\s_\pi^{-1}(\Delta)$ map isomorphically onto $Y$ via $p_\E$. 
	Thus, $$<\s_\pi^{-1}(\Delta)|_{\mathcal{K}_1}, \p_{2\s_1}\cap \s_\pi^{-1}(\Delta)> = <\s_\pi^{-1}(\Delta)|_{\mathcal{K}_1},\tilde{\s}>,\quad [(\p_{2\s_1}\cap \s_\pi^{-1}(\Delta))|_{\mathcal{K}_1}]=[\tilde{\s}|_{\mathcal{K}_1}].$$
\end{proof}
\begin{cor}\label{B_T-chow}
Let $B_{\mathcal{K}_1}$ be the restriction of $B_*$ over $\mathcal{K}_1\hr J_*(3)$. The following cycles linearly span $A^*(B_\T)$ 
$$\{[B_\T], [B_{\mathcal{K}_1}], [\Phi_\xi(\tilde{\s}_1)|_{\T_*}], [\Phi_\xi(\tilde{\s}_1)|_{\mathcal{K}_2}]\}.$$
\end{cor}
\begin{proof}
Using Theorem \ref{diagrams-pY}, we have that $A^*(\s_{\pi}^{-1}(\Delta))\simeq A^*(Y)$ is spanned by 
$$\s_{\pi}^{-1}(\Delta), \s_{\pi}^{-1}(\Delta)|_{\mathcal{K}_1},  \p_{2\s_1}\cap \s_\pi^{-1}(\Delta), (\p_{2\s_1}\cap \s_\pi^{-1}(\Delta))|_{\mathcal{K}_1}.$$
The images of these cycles under $\Phi_\xi$ linearly span $A^*(B_\T)$ by Corollary \ref{phi-s}.\\
We have $\Phi_\xi(\tilde{\s})_p = \Phi_\xi(\p_{2\s_1})_p = \Phi_\xi(\p_{2\s_1}\cap \s_\pi^{-1}(\Delta))_p\forall p\in \T_*$ by Remark \ref{lin-sys}. Hence, we have 
$$\Phi_\xi(\p_{2\s_1}\cap \s_\pi^{-1}(\Delta)) = \Phi_\xi(\tilde{\s}), \Phi_\xi(\p_{2\s_1}\cap \s_\pi^{-1}(\Delta))|_{\mathcal{K}_1} = \Phi_\xi(\tilde{\s})|_{\mathcal{K}_1}.$$
On the other hand, $\Phi_\xi(\tilde{\s}_1)|_{\T_*}$ is reducible with irreducible components $\Phi_\xi(\tilde{\s}), B_{\mathcal{K}_1}$, each with multiplicity one on a general fiber over $\T_*$. Thus, 
$$[\Phi_\xi(\tilde{\s}_1)|_{\T_*}] = [\Phi_\xi(\tilde{\s})]+[B_{\mathcal{K}_1}]$$ in $A^*(B_\T)$ and 
$$[\Phi_\xi(\tilde{\s})|_{\mathcal{K}_1}] = [\Phi_\xi(\tilde{\s})].\dfrac{1}{2}det_\xi^* \T_3  = [\Phi_\xi(\tilde{\s}_1)|_{\T_*}].\dfrac{1}{2}det_\xi^* \T_3 = [\Phi_\xi(\tilde{\s}_1)|_{\mathcal{K}_2}].$$ 
\end{proof}
\subsubsection{\normalfont{\textbf{First excision}}}
We use results from \S \ref{B-V}, \S \ref{B_T}; specifically Corollary \ref{B_V-chow}, Corollary \ref{B_T-chow}, and the excision sequence (\ref{excision-B}).
Combining Corollary \ref{B-V}, and Corollary \ref{B_T}; we get the following theorem.
\begin{theorem}\label{B-*}
	The following cycles linearly span $A_*(B_*)$:
	$$[B_*],[B_*|_{\T_*}],[\Phi_\xi(\tilde{\s_1})], [B_*|_{\mathcal{K}_1}], [\Phi_\xi(\tilde{\s_1})|_{\T_*}], [\Phi_\xi(\tilde{\s}_1)|_{\mathcal{K}_2}].$$
\end{theorem}
\noindent We shall prove the linear independence of these cycles and find polynomial relations in \S \ref{U_*}.

\subsection{Chow rings of $\tilde{U}(2,3,2), U(2,3,2)$}\label{U_*}
As outlined in (\ref{excision-U}), we apply the second excision and Lemma \ref{pj-pxi} to get
\begin{equation}\label{ddag}A^{i-2}(B_*)\xrightarrow{i_{B *}} S^i(\tilde{U}(2,3,2))\ri S^i(\tilde{U}(2,3,2)\setminus B_*)\ri 0\quad \forall i\geq 0\tag{\ddag}.\end{equation}
We have $\tilde{U}(2,3,2)\setminus B_*\simeq \tilde{\p}_\xi\setminus E_\xi$. To calculate $S^*(\tilde{\p}_\xi\setminus E_\xi)$, we use the excision sequence
$$S^*(E_\xi)\ri S^*(\tilde{\p}_\xi)\ri S^*(\tilde{\p}_\xi\setminus E_\xi)\ri 0.$$
\begin{lemma}\label{minus-E-xi}
	Excision on $\tilde{\p}_\xi$ yields
	$$S^*(\tilde{\p}_\xi\setminus E_\xi)\simeq \Q[\s_\pi^*H_J,\s_\pi^*\T_\xi]/(\s_\pi^*(H_J^3+H_J^2\T_\xi+\dfrac{1}{2}H_J\T_\xi^2), \s_\pi^*(H_J^2\T_\xi^2)).$$
\end{lemma}
\begin{proof}
 Let $\mathcal{I}:= i_{\xi *}(A^*(E_\xi))$ be the ideal generated by pushing forward the cycles in $A^*(E_\xi)$ to $\tilde{\p}_\xi$ via the inclusion $i_\xi: E_\xi\hr \tilde{\p}_\xi$. It follows from (\ref{ddag}) that $A^*(\tilde{\p}_\xi\setminus E_\xi)\simeq A^*(\tilde{\p}_\xi)/\mathcal{I}$.\\
The ideal $\mathcal{I}$ is generated by 
$$[E_\xi],[\tilde{L}_{2\s_1}], [\tilde{L}_{\xi_1}], [\tilde{\s}_1], [\s_{\pi}^{-1}(\Delta)]$$
 in $A^*(\tilde{\p}_\xi)$ by Theorem \ref{classes-E_V} and Theorem \ref{classes-E_T}.\\
Thus, we get the additional relations 
$$E_\s - \s_\pi^*(2H_J+\T_\xi) = 0, \s_\pi^*(H_J^2-\dfrac{1}{2}\T_\xi^2)-\xi_\s =  \s_\pi^*(H_J^2+H_J\T_\xi+\dfrac{1}{2}\T_\xi^2)-K_\s= 0$$
 using Theorem \ref{tilde-s-pi-delta} and Theorem \ref{exp-blowup}. 
\end{proof}
\noindent Following Theorem \ref{blowup-chow}, Lemma \ref{pj-pxi}, Remark \ref{blowdown} and (\ref{universal}), we have 
\begin{equation}\label{p-xi}S^i(\tilde{\p}_\xi)\simeq S^i(\tilde{U}(2,3,2))\oplus A^{i-1}(B_*)\forall i\geq 1.\end{equation}
Theorem \ref{B-*} and (\ref{ddag}) give us the following upper bounds:
\begin{enumerate}
	\item $\dim S^1(\tilde{U}(2,3,2))= 2, \dim A^0(B_*) =1$
	\item $\dim S^2(\tilde{U}(2,3,2))\leq 4,\dim A^1(B_*)\leq 2$
	\item $\dim S^3(\tilde{U}(2,3,2))\leq 4, \dim A^2(B_*)\leq 2$
	\item $\dim S^4(\tilde{U}(2,3,2))\leq 2,\dim A^3(B_*)\leq 1$
	\item $\dim S^5(\tilde{U}(2,3,2))\leq 1, \dim A^4(B_*) = 0$
\end{enumerate}
For any $ 1\leq j\leq 5$, we have equality in bound $(j)$ for $\dim S^j(\tilde{U}(2,3,2))$ if and only if (\ref{ddag}) is left exact for codimension $j$ and the generators of $A^{j-2}(B_*)$ in Theorem \ref{B-*} are linearly independent.\\
We have the following dimensions of $S^i(\tilde{\p}_\xi)$ from Theorem \ref{chow-p-xi}:
\begin{enumerate}
	\item $\dim S^1(\tilde{\p}_\xi) = 3$
	\item $\dim S^2(\tilde{\p}_\xi) = 6$
	\item $\dim S^3(\tilde{\p}_\xi) = 6$
	\item $\dim S^4(\tilde{\p}_\xi) = 3$
	\item $\dim S^5(\tilde{\p}_\xi) = 1.$
\end{enumerate}
\noindent Applying (\ref{p-xi}) to the above dimensions, we get the following Corollary.
\begin{cor}\label{inject}
The aforementioned upper bounds on $\{\dim S^i(\tilde{U}(2,3,2))\}_{1\leq i\leq 5}$ are thus in fact equalities for all $1\leq i\leq 5$, implying:
\begin{itemize}
	\item The linear generators  for $A^*(B_*)$, as described in Theorem \ref{B-*}, are linearly independent.
	\item The pushforward map in (\ref{ddag}) is injective for all values of $i$.
\end{itemize}
\end{cor}
\noindent Let $\T_U:= det_\xi^*\T_3 = \Phi_{\xi *}\s_\pi^*\T_\xi, H_U:= \Phi_{\xi *}(\s_\pi^*H_J) \in A^1(\tilde{U}(2,3,2))$.
By Lemma \ref{minus-E-xi}, $\T_U, H_U$ form a basis of $A^1(\tilde{U}(2,3,2))$.
The relations $$\Phi_\xi^*\T_U = \s_\pi^*\T_\xi, \T_U^3=0$$ follow from the injectivity of $\Phi_\xi^*$ and shall prove useful in computations to follow.
\begin{lemma}\label{pullback-hu}
Using \cite[Proposition 0.1.3]{beauville} and Theorem \ref{exp-blowup}, we have
\begin{gather*}
 \Phi_\xi^*(H_U) = \s_\pi^*(3H_J+\T_\xi)-E_\s ,
\Phi_{\xi *}(E_\s) = 2H_U+\T_U \\
 \Phi_{\xi *}(\s_\pi^*H_J^2)+H_U\T_U+\dfrac{1}{2}\T_U^2-\Phi_{\xi *}(K_\s) = B_*,
 \Phi_{\xi *}(\s_\pi^*H_J^2) = \Phi_{\xi *}(\xi_\s) + \dfrac{1}{2}\T_U^2.
\end{gather*}
\end{lemma}
\begin{proof}
Consider $i_\xi^*[\tilde{H}_{\xi_2}]\in A^1(E_\xi) = <[\tilde{L}_{\xi_1}],[\tilde{L}_{2\s_1}],[E_\T]>$.
Restricting to a fiber over $V\subset J_*(3)$, we have 
\begin{equation}\label{h-xi-2}i_\xi^*[\tilde{H}_{\xi_2}] = [\tilde{L}_{\xi_1}]+[\tilde{L}_{2\s_1}]+t[E_\T]\end{equation}
for some $t\in\Q$. Since $\tilde{H}_{\xi_2}\cap E_\xi$ has $\tilde{L}_{2\s_2},L_{\xi_1}$ as its irreducible components, we have $t=0$.\\
Hence, $\Phi_{E*}(i_\xi^*[\tilde{H}_{\xi_2}]) = B_*$ since $\Phi_{\xi *}([\tilde{L}_{2\s_2}]) = [B_*], \Phi_{\xi *}([\tilde{L}_{\xi_1}]) = 0$. 
Therefore, we have 
\begin{gather*}
[\tilde{H}_{\xi_2}] = \Phi_\xi^*(H_U+\T_U) - E_\xi\implies \Phi_\xi^*(H_U) = \s_\pi^*(3H_J+\T_\xi)-E_\s \\
\Phi_{\xi *}(E_\xi) = 0 \implies \Phi_{\xi *}(E_\s) = 2H_U+\T_U \\
\Phi_{\xi *}(\tilde{L}_{2\s_1}) = B_*\implies \Phi_{\xi *}(\s_\pi^*H_J^2)+H_U\T_U+\dfrac{1}{2}\T_U^2-\Phi_{\xi *}(K_\s) = B_*\\
\Phi_{\xi *}(\tilde{L}_{\xi_1}) = 0\implies \Phi_{\xi *}(\s_\pi^*H_J^2) = \Phi_{\xi *}(\xi_\s) + \dfrac{1}{2}\T_U^2.
\end{gather*}
\end{proof}
\begin{rem}
The expression for $\Phi_\xi^*(H_U)$ can also be derived using \cite[Proposition 4.7]{bertram}.
\end{rem}
	\begin{lemma}
	The pushforwards of $B_*, \s_\pi^*H_J^2, \xi_\s, K_\s$ under $\Phi_\xi$ are given by the following relations:
\begin{gather*}\label{phi-xi-*}
	\Phi_\xi^*B_* = \s_\pi^*(H_J^2+H_J\T_\xi+\dfrac{1}{2}\T_\xi^2)-K_\s,
\Phi_{\xi *}(\s_\pi^*H_J^2) = H_U^2-B_*, \Phi_{\xi *}(\xi_\s) = H_U^2-\dfrac{1}{2}\T_U^2-B_*,\\ \Phi_{\xi *}(K_\s) = H_U^2+H_U\T_U+\dfrac{1}{2}\T_U^2-2B_*.
\end{gather*}
\end{lemma}
\begin{proof}
 Using Lemma \ref{minus-E-xi}, there exists $ m\in \Q$ such that $\Phi_{\xi *}(\s_\pi^*H_J^2) = H_U^2+mB_*$. Hence, we have 
$$\Phi_{\xi *}(K_\s) = H_U^2+H_U\T_U+\dfrac{1}{2}\T_U^2+(m-1)B_*, \Phi_{\xi *}(\xi_\s)=H_U^2-\dfrac{1}{2}\T_U^2+mB_*.$$
We show that $m=-1$. 
Applying \cite[Theorem 0.1.3]{beauville}, we get 
$$K_\s + i_{\xi *}\Phi_E^*\Phi_{E *}i_{\xi}^*K_\s = \Phi_\xi^*\Phi_{\xi *}(K_\s)= 4\s_\pi^*H_J^2+6\s_\pi^*(H_J\T_\xi)-2E_\s\s_\pi^*\T_\xi-K_\s-2\xi_\s+2\s_\pi^*\T_\xi^2+(m-1)\Phi_\xi^*B_*$$
using $\Phi_E^*(\Phi_\xi(\tilde{\s}_1)) = \tilde{L}_{\xi_1}$, and that $\tilde{\s}_1$ maps isomorphically into $B_*$ via $\Phi_{\xi}$. Thus, we have
$$i_{\xi *}\Phi_E^*\Phi_{E *}i_{\xi}^*K_\s = 2i_{\xi *}\Phi_E^*\Phi_{E *}i_{\xi}^*\s_\pi^{-1}(\pi^*\s_1) = 2i_{\xi *}\Phi_E^*\Phi_{E *}\tilde{\s}_1=2i_{\xi *}\Phi_E^*[\Phi_{\xi}(\tilde{\s}_1)] = 2[\tilde{L}_{\xi_1}].$$
 Plugging this into the equation above and using Theorem \ref{exp-blowup}, we get
\begin{equation}\label{m-1}
	(m-1)B_* = 2K_\s-2\s_\pi^*H_J^2-2\s_\pi^*(H_J\T_\xi)-\s_\pi^*\T_\xi^2.
\end{equation}
Applying similar computations to $\s_\pi^*H_J^2$, we get 
\begin{equation}\label{h^2}\s_\pi^*H_J^2 = (\s_\pi^*(3H_J+\T_\xi)-E_\s)^2+m\Phi_\xi^*B_*-i_{\xi *}\Phi_E^*\Phi_{E *}i_\xi^*H_J^2.\end{equation}
Using $i_\xi^*H_{\xi_2} = \tilde{L}_{2\s_2}+\tilde{L}_{\xi_1}$, as proven in (\ref{h-xi-2}),  we have 
$$i_{\xi}^*\s_\pi^*H_J = \tilde{L}_{2\s_2}+\tilde{L}_{\xi_1} - E_\T = \tilde{L}_{2\s_3}+\tilde{L}_{\xi_4}-E_\T $$ 
in $A^*(E_\xi)$. Considering different loci that have generically transverse intersections in $E_\xi$, we have 
\begin{gather*}
\tilde{L}_{2\s_i}\tilde{L}_{2\s_j}=\s_\pi^{-1}(\Delta)\implies \Phi_{E *}(\tilde{L}_{2\s_i}\tilde{L}_{2\s_j}) = B_\T,
\tilde{L}_{\xi_i}\tilde{L}_{2\s_j} = \tilde{\s}_i\implies \Phi_{E*}(\tilde{L}_{\xi_i}\tilde{L}_{2\s_j}) = \Phi_\xi(\tilde{\s}_i),\\ \Phi_{E *}(\tilde{L}_{\xi_i}\tilde{L}_{\xi_j}) = 0
\end{gather*} 
in $A^*(E_\xi), A^*(B_*)$ for all $1\leq i<j\leq 6$. Using $\Phi_E^*(\Phi_\xi(\tilde{\s}_1)) = \tilde{L}_{\xi_1}$ and $[\tilde{\s}_i]=[\tilde{\s}_1]\forall 1\leq i\leq 6$ in $A^*(E_\xi)$, we have the following equation in $A^*(\tilde{\p}_\xi)$
\begin{align*}i_{\xi *}\Phi_E^*\Phi_{E *}i_\xi^*\s_\pi^*H_J^2 = i_{\xi *}\Phi_E^*(2\Phi_\xi(\tilde{\s}_1-B_\T)) = 2[\tilde{L}_{\xi_1}]-[E_\T] =\\ 2[\tilde{L}_{\xi_1}]-E_\xi.\s_\pi^*\T_\xi = 2\s_\pi^*H_J^2-E_\s\s_\pi^*\T_\xi+2\s_\pi^*(H_J\T_\xi)-2\xi_\s.\end{align*}
Plugging this into (\ref{h^2}), we get 
\begin{equation}\label{m}
m\Phi_\xi^*B_* = K_\s-\s_\pi^*H_J^2-\s_\pi^*(H_J\T_\xi)-\dfrac{1}{2}\s_\pi^*\T_\xi^2.
\end{equation}
Putting the equations (\ref{m-1}), (\ref{m}) together, we have $m=-1$ as claimed.
\end{proof} 
\begin{theorem}\label{u-odd}
	The ring $S^*(\tilde{U}(2,3,2))$ is generated by the cycles $H_U,\T_U,B_*$. In particular,
		\begin{gather*}S^*(\tilde{U}(2,3,2))=\Q[H_U,\T_U,B_*]/(\T_U^3,H_U^2\T_U^2-4B_*\T_U^2,B_*H_U^2-B_*H_U\T_U+\dfrac{1}{2}B_*\T_U^2,\\B_*H_U^2-B_*^2,H_U^3+H_U^2\T_U+\dfrac{1}{2}H_U\T_U^2-4B_*H_U).\end{gather*}
	\end{theorem}
\begin{proof}
We use injectivity of $\Phi_\xi^*$ to compute relations in $A^*(\tilde{U}(2,3,2))$.\\
We have $\Phi_\xi^*(B_*) = \tilde{L}_{2\s_1}, [\tilde{\s}_1] = \tilde{L}_{2\s_1}\s_\pi^*H_J$ in $A^*(\tilde{\p}_\xi)$ using Theorem \ref{exp-blowup}. Therefore,
$$[\Phi_\xi(\tilde{\s}_1)] = \Phi_{\xi *}(\tilde{\s}_1) = \Phi_{\xi *}(\tilde{L}_{2\s_1}\s_\pi^*H_J) = \Phi_{\xi *}(\Phi_\xi^*B_*\s_\pi^*H_J) = B_*H_U.$$ 
Applying Theorem \ref{B-*}, Corollary \ref{inject}  and Theorem \ref{minus-E-xi}, we have the following  $\mathbb{Q}$-basis of $S^*(\tilde{U}(2,3,2))$
$$\{[\tilde{U}(2,3,2)],H_U,\T_U,H_U^2,H_U\T_U,\T_U^2,B_*,H_U^2\T_U,H_U\T_U^2,B_*H_U,B_*\T_U,B_*H_U\T_U,B_*\T_U^2,B_*H_U\T_U^2\}.$$
Hence, it is enough to express $H_U^3,B_* H_U^2,H_U^2 \T_U^2,H_U^3\T_U ,B_*^2$ as linear combinations of the basis elements above to describe the ring $S^*(\tilde{U}(2,3,2))$.
Using (\ref{phi-xi-*}), we have 
$$\tilde{L}_{2\s_1}\s_\pi^*H_J^2 =0\implies \Phi_{\xi *}(\s_\pi^*H_J^2.\Phi_\xi^*B_*)=0\implies B_*\Phi_{\xi *}(\s_\pi^*H_J^2)=0\implies B_*H_U^2 = B_*^2$$
in $A^*(\tilde{U}(2,3,2))$. Furthermore, 
$$B_*H_U^2 = [\Phi_\xi(\tilde{\s}_1)]H_U = \Phi_{\xi *}(\tilde{\s}_1.\Phi_\xi^*H_U).$$ 
On the other hand, we have the following 
\begin{gather*}
	\Phi_{\xi *}(\tilde{\s}_1.\Phi_\xi^*H_U) = \Phi_{\xi *}(\s_\pi^*(H_J^3+H_J^2\T_\xi+\dfrac{1}{2}H_J\T_\xi^2)(\s_\pi^*(3H_J+\T_\xi)-E_\s)) = \\\Phi_{\xi *}(\s_\pi^*(H_J^3\T_\xi+H_J^2\T_\xi^2)+\dfrac{1}{2}(K_\s-\xi_\s)\s_\pi^*\T_\xi^2),\\
	\Phi_{\xi *}(\tilde{\s}_1) = B_*H_U=\Phi_{\xi *}(\s_\pi^*(H_J^3+H_J^2\T_\xi+\dfrac{1}{2}H_J\T_\xi^2)) \implies \Phi_{\xi *}(\s_\pi^*(H_J^3)) = B_*(H_U+\T_U)-H_U^2\T_U-\dfrac{1}{2}H_U\T_U^2\\
	\implies B_*H_U^2 = \Phi_{\xi *}(\s_\pi^*H_J^3)\T_U+\Phi_{\xi *}(\s_\pi^*H_J^2)\T_U^2+\dfrac{1}{2}\Phi_{\xi *}(K_\s-\xi_\s)\T_U^2= B_*H_U\T_U-\dfrac{1}{2}B_*\T_U^2.
\end{gather*}
A straightforward computation in $A^*(\tilde{\p}_\xi)$ shows that 
\begin{gather*}\Phi_\xi^*(H_U^3+H_U^2\T_U+\dfrac{1}{2}H_U\T_U^2) = 4\Phi_\xi^*(B_*H_U)\implies H_U^3+H_U^2\T_U+\dfrac{1}{2}H_U\T_U^2 = 4B_*H_U\text{ in } A^*(\tilde{U}(2,3,2)),\\
\Phi_\xi^*(4B_*\T_U^2) = 4(\s_\pi^*(H_J^2\T_\xi^2)-K_\s\s_\pi^*\T_\xi^2) = \Phi_\xi^*(H_U^2\T_U^2)\implies H_U^2\T_U^2 = 4B_*\T_U^2\text{ in } A^*(\tilde{U}(2,3,2)),\\
\implies H_U^3\T_U+H_U^2\T_U^2=4B_*H_U\T_U\implies H_U^3\T_U = 4B_*\T_U(H_U-\T_U)
\end{gather*}
proving the claimed relations.
\end{proof}
\section{Generators in terms of tautological classes}\label{tautological}
\noindent We use the tautological classes as defined in \S\ref{tautological-intro} to express the generators of $A^*(\tilde{U}(2,d,2))$ as presented in Theorem \ref{u-even} and Theorem \ref{u-odd}.\\ 
\subsection{Even Degree}\label{even-taut}
We follow the notation as established in \S\ref{even-degree}.
As shown in Theorem \ref{u-odd}, $A^*(\tilde{U}(2,2d,2))$ is generated by the cycles $\nu,\tilde{Z}_d$ in codimension $1$. 
Hence, it is enough to express the generators in Theorem \ref{fringuelli-main}, which are tautological, in terms of $\nu, \tilde{Z}_d$.\\
We further divide our computations based on the parity of $d$. However, in either case, we shall obtain the same expression in Remark \ref{rem-chern}.
\begin{lemma}\label{chern-e0}
	Put $\xi_i := (1\times \tau_{i,3})^*\xi\in Pic(\C_*\times_{M_*}J_*(i))$ for $i=0,1$ and let $\pi_{PJ},\pi_{CP},\pi_{CJ}$ be the projections from $\C_*\times_{M_*}\p(\e)\times_{M_*}J_*(0)$ onto $\p(\e)\times_{M_*}J_*(0),\C_*\times_{M_*}\p(\e), \C_*\times_{M_*}J_*(0)$, respectively.\\
	Using (\ref{diag_1}), the pullback 
	$$(f\times 1)^*:Pic(|2\T_*|_\pi\times_{M_*}J_*(0))\ri Pic(\p(\e)\times_{M_*}J_*(0))$$
	 is injective. Furthermore,
	$$(f\times 1)^*(\nu-\dfrac{1}{8}\tilde{Z}_0) = (f\times 1)^*(\pi_\T^*\zeta-2\pi_J^* \T_0 )= \pi_{PJ*}(c_2(\pi_{CP}^*(\mathcal{P}\otimes (1\times p_J)^*\xi_1^{-1})\otimes \pi_{CJ}^*(\xi_0))).$$
\end{lemma}
\begin{proof}
	The bundle $\mathcal{P}$ is determined by an extension class on $\C_J\times_{J_*(1)}\p(\e)$. We have the exact sequence
	$$0\ri \pi_{P(\e)}^*\mo_{P(\e)}(1)\ri \mathcal{P}\ri(1\times p_J)^*\xi_1^2  \ri 0.$$
Thus, we have 
	 $$c(\mathcal{P}\otimes (1\times p_J)^*\xi_1^{-1}) = (1+\pi_{P(\e)}^*\mo(1)-(1\times p_J)^*\xi_1)(1+(1\times p_J)^*\xi_1)\implies $$
	\begin{gather}\label{pushforward-chern}\pi_{PJ*}(c_2(\pi_{CP}^*(\mathcal{P}\otimes (1\times p_J)^*\xi_1^{-1})\otimes \pi_{CJ}^*(\xi_0))) = \pi_{PJ*}((\pi_{P(\e)}^*\mo(1)-\pi_{CP}^*(1\times p_J)^*\xi_1+\pi_{CJ}^*\xi_0)\times
		\\(\pi_{CP}^*(1\times p_J)^*\xi_1+\pi_{CJ}^*\xi_0)).\nonumber\end{gather}
	Applying arguments analogous to Theorem \ref{chow-curve}, we have
	\begin{align*}\pi_{PJ*}(\pi_{CP}^*(1\times p_J)^*c_1(\xi_1)^2) =-2\pi_{P(\e)}^*p_J^*\T_1, \pi_{PJ*}(\pi_{CJ}^*c_1(\xi_0)^2)=-2\pi_{J}^*\T_0,\\
		 \pi_{PJ*}(\pi_{CP}^*(1\times p_J)^*c_1(\xi_1)) = 1, \pi_{PJ*}(\pi_{CJ}^*c_1(\xi_0)) =0 \in A^1(P(\e)\times_{M_*}J_*(0)) \implies 
\\\pi_{PJ*}(c_2(\pi_{CP}^*(\mathcal{P}\otimes (1\times p_J)^*\xi_1^{-1})\otimes \pi_{CJ}^*(\xi_0))) = \pi_{P(\e)}^*\mo(1)+2\pi_{P(\e)}^*p_J^*\T_1-2\pi_J^*\T_0.\end{align*}
The embedding $\p(\e)\hr JU_*$ in Theorem \ref{isom} was given by the injection 
\begin{equation}\label{zeta-c1}
	\e\otimes L_J^\vee\otimes L_{\T_*}^2\simeq \mathcal{F}^\vee \hr \pi^*\pi_*L_{\T_*}^2.
\end{equation}
Let $\zeta$ be the relative hyperplane divisor on $\p(\E)$. We have 
$$f^*\zeta = \mo_{\p(\e)}(1)+p_J^*(c_1(L_J)-2\T_1) = \mo_{\p(\e)}(1)+2p_J^*\T_1.$$
 by Lemma \ref{c_1}. Plugging these values in (\ref{pushforward-chern}) and using (\ref{t-Z}), we get
$$\pi_{PJ*}(c_2(\pi_{CP}^*(\mathcal{P}\otimes (1\times p_J)^*\xi_1^{-1})\otimes \pi_{CJ}^*(\xi_0))) = \pi_{P(\e)}^*f^*\zeta-2\pi_J^*\T_0 = (f\otimes 1)^*(\nu-\dfrac{1}{8}\tilde{Z}_0).$$ 
The injectivity of $(f\times 1)^*$ on Picard groups follows from (\ref{zeta-c1}).
\end{proof}
\begin{rem}\label{rem-chern}
	The Poincar\'e bundle on $\C_*\times_{M_*}\p(\e)\times_{M_*}J_*(0)$ defining the map $f\otimes 1: \p(\e)\times J_*(0)\ri \tilde{U}(2,0,2)$ is given by $\pi_{CP}^*(\mathcal{P}\otimes (1\times p_J)^*\xi_1^{-1})\otimes \pi_{CJ}^*(\xi_0)$. Hence, we have 
	$$(f\otimes 1)^*(\nu-\dfrac{1}{8}\tilde{Z}_0) = (f\otimes 1)^*\kappa_{-1,0,1}\implies \nu -\dfrac{1}{8}\tilde{Z}_0 = \kappa_{-1,0,1}.$$
	Furthermore, we have 
	$$c_1(\mathcal{P}\otimes (1\times p_J)^*\xi_1^{-1}\otimes \pi_{CJ}^*\xi_i)= \pi_{P(\e)}^*\mo(1)+2\pi_{CJ}^*\xi_i\in A^1(\C_*\times_{M_*0}\p(\e)\times_{M_*}J_*(i))\quad\forall i\in \mathbb{Z}.$$
\end{rem}
\begin{theorem}\label{nu-xi-even}
	Let $\Xi,\T$ be the generators of $Pic(U(2,2d,2))$ as described in Theorem \ref{fringuelli-main}. For $d\in \{0,1\}$, we have 
	$$\T = -\nu-\dfrac{3+2d-4\lfloor\dfrac{d+1}{2}\rfloor}{8}{\tilde{Z}_d}, \Xi = -\dfrac{1}{2}\tilde{Z}_d\in A^1(U(2,2d,2)).$$
\end{theorem}
\begin{proof}
	Applying the Grothendieck-Riemann-Roch theorem, we have
	$$\Lambda(0,1,0) = \dfrac{1}{2}(\kappa_{-1,2,0}-\kappa_{0,1,0}), \Lambda(1,1,0) = \dfrac{1}{2}(\kappa_{-1,2,0}+\kappa_{0,1,0}), \Lambda(0,0,1) = \Lambda(0,1,0)-\kappa_{-1,0,1}.$$
By Theorem \ref{pullback-nu}, we have the following:
\begin{itemize}
	\item $\Xi = \kappa_{-1,2,0}-2\kappa_{0,1,0} \in Im(\det_1^*)$ for $d=1$
	\item $\Xi = \kappa_{-1,2,0} \in Im(\det_0^*)$ for $d=0$. 
\end{itemize}
	Using the Poincar\'e bundle $\xi_{d}\ri \C_*\times_{M_*}J_*(d)$, where $\xi_{d} := (1\times \tau_{d,3})^*\xi$. 
	Applying Theorem \ref{chow-curve}, we get
	$$\xi_d^2 = -\pi_C^*K_\pi.\pi_{J_*(d)}^*\T_d \in A^2(\C_*\times_{M_*}J_*(d)).$$
	Applying (\ref{t-Z}) for $d=0,1$, we get
	$$\Xi = -2det_d^*\tau_{2d,1}^*(\T_*) = - \dfrac{1}{2} \tilde{Z}_d.$$ 
	For $d$ odd, we fix $d=1$ and apply Theorem \ref{u-even} and Remark \ref{rem-chern} to get 
	$$\T=\Lambda (0,0,1) = -\T_{gen} = -\nu-\dfrac{1}{8}\tilde{Z}_1,$$
	which proves the claimed expression for $\T$. \\
	For $d$ even, we fix $d=0$ and apply Lemma \ref{chern-e0} to get 
	$$\T = \Lambda(0,0,1)+\Lambda(1,1,0) =\kappa_{-1,2,0}-\kappa_{-1,0,1} = -\nu-\dfrac{3}{8}\tilde{Z}_0 .$$ 
\end{proof}

\subsection{Odd Degree}\label{odd-taut}
We follow the notation in \S\ref{odd-degree}.
As established in Theorem \ref{u-odd} and Theorem \ref{artin-chow}, the rings $S^*(\tilde{U}(2,3,2))\simeq A^*(U(2,3,2))\simeq A^*(\mathcal{U}(2,3,2))$ are generated by $H_U,\T_U,\text{ and } B_*$. 
\begin{theorem}\label{b-*-taut}
Let $\pi_{\mathscr{U}}:\C_2\times_{\mathcal{M}_2}\mathscr{U}(2,3,2)\ri \mathscr{U}(2,3,2)$ be the projection map.
 We have $$B_* = \dfrac{1}{2}\left(\dfrac{1}{2}(\kappa_{-1,2,0}-\kappa_{0,1,0})-\kappa_{-1,0,1}\right)^2- \dfrac{1}{2}\left( \dfrac{\kappa_{-1,3,0}}{3}-{\ka_{-1,1,1}}-\dfrac{\ka_{0,2,0}}{2}+{\ka_{0,0,1}}\right)\in A^2(\mathcal{U}(2,3,2)).$$ 
\end{theorem}
\begin{proof}
Following a similar argument to that of \cite[Lemma 6.4]{larson-24}, we apply Grothendieck-Riemann-Roch to $\mathcal{E}_3$ and the map $\pi_{\mathscr{U}}$
\begin{gather*}
	ch(\pi_{\mathscr{U}!}\e_3) = \pi_{\mathscr{U}*}\left(ch(\e_3)\left(1-\dfrac{K_{\pi_{\mathscr{U}}}}{2}\right)\right)\implies
	 c_1(\pi_{\mathscr{U}*}\e_3) = \dfrac{1}{2}(\kappa_{-1,2,0}-\kappa_{0,1,0})-\kappa_{-1,0,2} ,\\
	  B_*=-c_2(R^1\pi_{\mathscr{U}*}\e_3) = \dfrac{1}{2}\left(\dfrac{1}{2}(\kappa_{-1,2,0}-\kappa_{0,1,0})-\kappa_{-1,0,1}\right)^2- \dfrac{1}{2}\left( \dfrac{\kappa_{-1,3,0}}{3}-{\ka_{-1,1,1}}-\dfrac{\ka_{0,2,0}}{2}+{\ka_{0,0,1}}\right)\in A^2(\mathcal{U}(2,3,2))\end{gather*}
since $c_1(R^1\pi_{\mathscr{U}*}\e_3) = 0, c_2(\pi_{\mathscr{U}*}\e_3) = 0$. The tautological cycle on the right hand side can be shown to be in the image of $\nu_{2,3}^*$ by Theorem \ref{pullback-nu}.
\end{proof}
\begin{thm}[{{\cite[Lemma 6.4]{larson-24}}}]
The universal $\T$-divisor $\T_U\in A^1(\tilde{U}(2,3,2))$ is given by
\begin{equation*}\label{t-u}\T_U = \dfrac{1}{2}\left(3\ka_{0,1,0}-\ka_{-1,2,0}\right)\in A^1(\tilde{U}(2,3,2)).\end{equation*}
\end{thm}
\noindent We compute the chern classes of the Poincar\'e bundle as described in \S\ref{poincare-bertram}. 
\begin{lemma}\label{chern-classes-odd}
	Let $\s_E:= \s_\pi|_{E_\s}:E_\s\ri \C_J$ be the restriction of $\s_\pi$, and $\pi_{\C_J}: \C_J\times_{J_*(3)}\tilde{\p}_\xi\ri \C_J, \pi_{\tilde{\p}_\xi}: \C_J\times_{J_*(3)}\tilde{\p}_\xi\ri \tilde{\p}_\xi$ be the projection maps. \\
	The chern classes of the bundle $\e^1$ defined in Theorem \ref{poincare-bertram} are given by 
	$$c_1(\e^1) = \pi_{\tilde{\p}_\xi}^*(\s_\pi^*H_J-E_\s)+\pi_{\C_J}^*\xi, c_2(\e^1) = \pi_{\C_J}^*\xi\pi_{\tilde{\p}_\xi}^*(\s_\pi^*H_J-E_\s)+A$$
where $A:= (1\times i_\s)_*(1\times \s_E)^*\Delta_J$ and $\Delta_J\hr \C_J\times_{J_*(3)}\C_J$ is the diagonal embedding of $\C_J$.
\end{lemma}
\begin{proof}
	The bunde $\e^1$ is defined using the following exact sequences
	\begin{gather*}
		0\ri \e^1\ri (1\times \s_\pi)^*\e^0\ri (1\times i_\s)_*(1\times \s_E)^*(\mo(\Delta_J)\otimes \pi_2^*\mo(1)|_{j(\C_J)})\ri 0,\\
		0\ri \pi_{\tilde{\p}_\xi}^*\s_\pi^*\mo_{\p_\xi}(1)\ri  (1\times \s_\pi)^*\e^0\ri \pi_{\C_J}^*\xi\ri 0.
	\end{gather*}
	Let $\mathcal{F}:= (1\times i_\s)_*(1\times \s_E)^*(\mo(\Delta_J)\otimes \pi_{E_\s}^*(K_\s\otimes\xi_\s))$ where $\pi_{E_\s}:\C_J\times_{J_*(3)}E_\s\ri E_\s$ be the projection map. We have
	\begin{gather*}c(\e^1)= c(\mathcal{F})^{-1}(1\times\s_\pi)^*c(\e^0) = c(\mathcal{F})^{-1}(1+\pi_{\tilde{\p}_\xi}^*\s_\pi^*H_J)(1+\pi_{\C_J}^*\xi)
	 = 1+\pi_{\tilde{\p}_\xi}^*\s_\pi^*H_J+\pi_{\C_J}^*\xi-c_1(\mathcal{F}) -\\
	 c_1(\mathcal{F})(\pi_{\tilde{\p}_\xi}^*\s_\pi^*H_J+\pi_{\C_J}^*\xi)-c_2(\mathcal{F})+c_1(\mathcal{F})^2+\pi_{\tilde{\p}_\xi}^*\s_\pi^*H_J\pi_{\C_J}^*\xi.\end{gather*}
	Applying the Grothendieck-Riemann-Roch Theorem to the embedding $(1\times i_\s)$, we get 
\begin{gather*}ch(\mathcal{F}) = (1\times i_\s)_*\left((1\times\s_E)^*ch(\Delta_J\otimes \pi_{E_\s}^*\s^*(K_{\pi_J}\otimes\xi))\pi_{E_\s}^*td(\mathcal{N}_{E_\s/\tilde{\p}_\xi})^{-1}\right)\implies
	c_1(\mathcal{F}) = \pi_{\tilde{\p}_\xi}^*E_\s\implies \\
	 c_2(\mathcal{F}) = \dfrac{1}{2}\pi_{\tilde{\p}_\xi}^*E_\s^2 -(1\times i_\s)_*(1\times\s_E)^*\Delta_J - (1\times i_\s)_*\pi_{E_\s}^*\s^*(K_{\pi_J}+\xi)+\dfrac{1}{2}(1\times i_\s)_*\pi_{E_\s}^*\mathcal{N}_{E_\s/\tilde{\p}_\xi}=\\
	  \dfrac{1}{2}\pi_{\tilde{\p}_\xi}^*E_\s^2-A-\pi_{\tilde{\p}_\xi}^*(K_\s+\xi_\s)+\dfrac{1}{2}\pi_{\tilde{\p}_\xi}^*i_{\s *}\mo_{E_\s}(E_\s) = \pi_{\tilde{\p}_\xi}^*(E_\s^2-K_\s-\xi_\s)-A .\end{gather*}
	  Plugging these values back into the expression for $c(\e^1)$, we get 
	  $$c_1(\e^1) = \pi_{\tilde{\p}_\xi}^*(\s_\pi^*H_J-E_\s)+\pi_{\C_J}^*\xi\text{ and } $$
	  $$c_2(\e^1)= \pi_{\tilde{\p}_\xi}^*\s_\pi^*H_J\pi_{\C_J}^*\xi-\pi_{\tilde{\p}_\xi}^*E_\s(\pi_{\tilde{\p}_\xi}^*\s_\pi^*H_J+\pi_{\C_J}^*\xi)+\pi_{\tilde{\p}_\xi}^*E_\s^2+A+\pi_{\tilde{\p}_\xi}^*(K_\s+\xi_\s-E_\s^2). $$
Simplifying the expression, we get $c_2(\e^1)= \pi_{\tilde{\p}_\xi}^*\s_\pi^*H_J\pi_{\C_J}^*\xi+A-\pi_{\tilde{\p}_\xi}^*E_\s\pi_{\C_J}^*\xi$.
\end{proof}
\begin{rem}
	Lemma \ref{b-*-taut} can be further verified by applying $\Phi_\xi^*(B_*) = \s_\pi^*(H_J^2+H_J\T_\xi+\dfrac{1}{2}\T_\xi^2)-K_\s,$ and calculating the associated $\kappa$-classes in $A^*(\tilde{\p}_\xi)$ using Lemma \ref{chern-classes-odd}. 
\end{rem}
\begin{theorem}\label{h-u}
	We have the following expression of $H_U$ in terms of tautological classes 
	$$H_U = \dfrac{1}{2}\left(4\ka_{-1,0,1}-\ka_{-1,2,0}\right)\in A^1(\tilde{U}(2,3,2)).$$
\end{theorem}
\begin{proof}
The claimed tautological class is in the image of $\nu_{2,3}^*$ by Theorem \ref{pullback-nu} since 
$$ \pi_{S*}(4c_2(\e_S)-c_1(\e_S)^2) =  \pi_{S*}(4c_2(\e_S\otimes \pi_{S}^*\lb)-c_1(\e_S\otimes \pi_{S}^*\lb)^2)$$
for any family $\e_S\ri \C_S\xrightarrow{\pi_S} S$ and $\lb\in Pic(S)$.\\
Using Lemma \ref{chern-classes-odd} and Lemma \ref{pullback-hu}, we have 
\begin{gather*}
	\pi_{\tilde{\p}_\xi *}A =\pi_{\tilde{\p}_\xi *}(1\times i_\s)_*(1\times \s_E)^*\Delta_J =\pi_{E_\s *}(1\times \s_E)^*\Delta_J = E_\s\implies\\
\pi_{\tilde{\p}_\xi *}(4c_2(\e^1)-c_1(\e^1)^2) = 4\pi_{\tilde{\p}_\xi *}(\pi_{\C_J}^*\xi\pi_{\tilde{\p}_\xi}^* (\s_\pi^*H_J-E_\s)+A)-\pi_{\tilde{\p}_\xi *}(\pi_{\tilde{\p}_\xi}^*(\s_\pi^*H_J-E_\s)+\pi_{\C_J}^*\xi)^2 \\
=6\s_\pi^*H_J+2\s_\pi^*\T_\xi - 2E_\s = 2\Phi_\xi^*(H_U).
\end{gather*}
The theorem now follows from the injectivity of $\Phi_\xi^*$. 
\end{proof}
\begin{cor}
	Similar computations also yield the following relation in $A^2(\mathscr{U}(2,3,2))$
	$$4\ka_{0,0,1}-\ka_{0,2,0} = \dfrac{1}{2}\left(4\ka_{-1,0,1}-\ka_{-1,2,0}\right)^2.$$
\end{cor}
\section*{References}
	\nocite{*}
	\renewcommand{\section}[2]{}%
	\bibliographystyle{plainnat}
	{\small\bibliography{draft_6.bib}}

\end{document}